\documentclass[11pt]{amsart}

\usepackage[T1]{fontenc}
\usepackage[utf8]{inputenc}
\usepackage[english]{babel}
\usepackage{lmodern}
\usepackage{microtype}
\usepackage{amsmath,amssymb,amsthm,mathtools,bm}
\usepackage[shortlabels]{enumitem}
\usepackage[hidelinks]{hyperref}
\usepackage[margin=1.05in]{geometry}
\usepackage{graphicx}
\usepackage{float}
\graphicspath{{figures/}}
\hypersetup{
  pdftitle={Laplace Asymptotics near Stratified Minimum Sets},
  pdfauthor={Pavel Ievlev},
  pdfsubject={Laplace Asymptotics near Stratified Minimum Sets},
  pdfkeywords={Laplace method, Riemannian manifold, stratified set,
    normal-cell decomposition, tubular neighborhood, Weibullian chaos}
}

\allowdisplaybreaks
\numberwithin{equation}{section}

\newtheorem{theorem}{Theorem}[section]
\newtheorem{proposition}[theorem]{Proposition}
\newtheorem{lemma}[theorem]{Lemma}
\newtheorem{corollary}[theorem]{Corollary}
\theoremstyle{definition}
\newtheorem{definition}[theorem]{Definition}
\newtheorem{assumption}[theorem]{Assumption}
\newtheorem{example}[theorem]{Example}
\theoremstyle{remark}
\newtheorem{remark}[theorem]{Remark}

\DeclareMathOperator{\dist}{dist}
\DeclareMathOperator{\Sym}{Sym}
\DeclareMathOperator{\diag}{diag}
\DeclareMathOperator{\erf}{erf}
\DeclareMathOperator{\erfc}{erfc}
\DeclareMathOperator{\rank}{rank}
\DeclareMathOperator{\relint}{relint}
\newcommand{\R}{\mathbb{R}}
\newcommand{\C}{\mathbb{C}}
\newcommand{\ind}{\mathbf{1}}
\newcommand{\dd}{\,\mathrm{d}}
\newcommand{\dvol}{\,\mathrm{dvol}}
\newcommand{\vol}{\operatorname{vol}}
\newcommand{\eps}{\varepsilon}
\newcommand{\cA}{\mathcal{A}}
\newcommand{\cE}{\mathcal{E}}
\newcommand{\cT}{\mathcal{T}}
\newcommand{\cD}{\mathcal{D}}
\newcommand{\fD}{\mathfrak{D}}
\newcommand{\cI}{\mathcal{I}}
\newcommand{\Afib}{\mathsf{A}}
\newcommand{\cL}{\mathcal{L}}
\newcommand{\inner}[2]{\left\langle #1,#2\right\rangle}

\title[Laplace Asymptotics near Stratified Minimum Sets]{Laplace Asymptotics near Stratified Minimum Sets}
\author{Pavel Ievlev}
\address{Department of Actuarial Science, University of Lausanne,\\
UNIL-Dorigny, 1015 Lausanne, Switzerland}
\email{ievlev.pn@gmail.com}
\date{}

\begin{document}

\begin{abstract}
We develop a real Laplace method for integrals
\begin{equation*}
  I(z)=\int_N a(x)e^{-z f(x)}\dvol_g(x),\qquad z\to\infty,
\end{equation*}
near a compact, possibly stratified, minimum set of the phase.  The basic input
is an adapted normal-cell disintegration, defined up to null sets, together with
piecewise anisotropic dilations and homogeneous limiting models.  The framework
simultaneously permits finite $C^2$ stratifications; Borel base refinements and
$C^1$ adapted charts; measurable base-dependent cells, relative domains, and
incident-stratum assignments; different fiber dimensions and weights on
different pieces; a continuous phase and a merely measurable amplitude; and
limiting models controlled either by ambient coercivity or directly by
cellwise exponential tightness.  The cells may converge only in weighted
measure, or pointwise under an integrable base majorant.  To the best of our
knowledge, no existing theorem for real Laplace integrals accommodates this
full combination of geometric and analytic assumptions.  The tradeoff is that
the adapted disintegration and limiting models are supplied hypotheses to be
verified in each application.

On each piece, the coefficient is obtained by integrating the limiting profile
and amplitude over the limiting cell and then over the base.  We prove a
normalized-profile leading theorem under cellwise tightness, a pointwise-cell
theorem with an integrable base majorant, and an additive global comparison
theorem that remains valid when leading coefficients vanish or cancel.  Joint
frontier asymptotics follow by rescaling frontier and normal variables together.
Finite expansions are obtained by representing each scaled cell on a fixed
reference cell and expanding the complete transported density.  The framework
recovers the Morse--Bott formula in the full quadratic normal case and gives
conic, Euclidean, and radial specializations.

The examples show why lower-dimensional and boundary strata cannot be ignored.
For a planar three-link Bayesian inverse-kinematics posterior, a smooth curve of
exact configurations and an isolated joint-limit fold contribute at the same
leading order.  A rank-deficient positive-semidefinite posterior yields the
Gaussian solid-angle factor.  For homogeneous Weibullian chaos, exact radial
integration produces an incomplete-gamma profile and leads to interior-zero and
fractional boundary-frontier tail laws beyond the positive-rate Gaussian-chaos
regime.
\end{abstract} 
\maketitle

\section{Introduction}

\subsection*{Purpose and point of view}

Real Laplace asymptotics are selected by exponential decay away from the global
minimum set.  In the classical smooth nondegenerate case this localization has a
universal local model.  If $x_0$ is an interior minimum and the Hessian of $f$ at
$x_0$ is positive definite, then
\begin{equation}\label{eq:intro-quadratic-model}
  f(x_0+v)-f(x_0)=\frac12\inner{Hv}{v}+o(|v|^2),
\end{equation}
and the dominant scale is $v=z^{-1/2}u$.  For a Morse--Bott minimum manifold,
the same quadratic law holds in the normal directions.  The Gaussian constants
and Hessian determinants in the usual formula come from this canonical
quadratic tangent model.

Outside this class there is no comparable universal normal form.  A one-sided
well, a corner, a conic constraint, a finite-type flat direction, or a nonsmooth
phase may have leading decay of order different from one or two, with different
weights in different directions.  Without a structural assumption, even a
one-dimensional continuous nonnegative function can be chosen so that its
sublevel distribution produces polynomial, logarithmic or stretched-exponential, 
or other irregular asymptotics, and may even fail to possess a single
asymptotic equivalent.  Thus a nonsmooth or nonquadratic Laplace theorem must
either assume sublevel asymptotics directly or provide a geometric mechanism
for computing them.

The invariant scalar object is the sublevel distribution.  If $h=f-f_0\ge0$ and
$\nu$ is a positive measure, set
\begin{equation*}
  F_\nu(t)=\nu\{x:h(x)\le t\}.
\end{equation*}
Then $\int e^{-zh}\dd\nu$ is the Laplace--Stieltjes transform of the pushforward
measure $h_\#\nu$.  For example, if $F_\nu(t)\sim C t^\alpha L(t)$ with $\alpha\ge0$
and $L$ slowly varying, Tauberian theory for regularly varying functions
(see~\cite{BinghamGoldieTeugels}) gives
\begin{equation*}
  \int e^{-zh}\dd\nu
  \sim C\Gamma(\alpha+1)z^{-\alpha}L(1/z)
  \qquad (z\to\infty).
\end{equation*}
This formulation is coordinate-free and sharp,
but it is often not constructive: it does not by itself identify which sector,
stratum, boundary face, or normal cone creates the coefficient.  The present
paper develops a complementary constructive calculus.  Once an adapted
almost-everywhere disjoint normal-cell disintegration and a weighted limiting
model have been supplied, the relevant sublevel and Laplace coefficients are
computed from the corresponding blow-up.

\subsection*{The normal-cell calculus}

Let
\begin{equation}\label{eq:intro-main-integral}
  I(z)=\int_N a(x)e^{-z f(x)}\dvol_g(x),\qquad z\to\infty,
\end{equation}
where $(N,g)$ is a Riemannian manifold and the compact minimum set
$M=\{f=f_0\}$ may be stratified.  The geometric data consist of finitely many
base pieces $P_\alpha\subset M$, coordinate maps
\begin{equation*}
  (p,s)\longmapsto F_\alpha(p,s),\qquad p\in P_\alpha,
\end{equation*}
and admissible normal cells $\Afib_\alpha(p)$ in the fiber variable $s$.
The cells are chosen so that the resulting images cover a neighborhood of
$M$, up to null sets, without double counting.  A base variable $p$ records
where the nearby point is assigned on the minimum set; the fiber variable $s$
records its displacement inside the chosen admissible cell.

In this paper ``normal'' means such an adapted fiber coordinate, not
necessarily a full normal vector space.  In smooth Fermi coordinates the
cells may be polar cones or full normal spaces, but the abstract theory also
allows relative domains, boundary half-spaces, convex normal fans,
nearest-point-type assignments, and cells that depend on the base point.  This
freedom is essential when several strata or boundary faces meet: the pair
$(P_\alpha,\Afib_\alpha(p))$ specifies how nearby mass is assigned before any
asymptotic comparison is made.

\begin{figure}[!htb]
\centering
\includegraphics[width=0.65\linewidth]{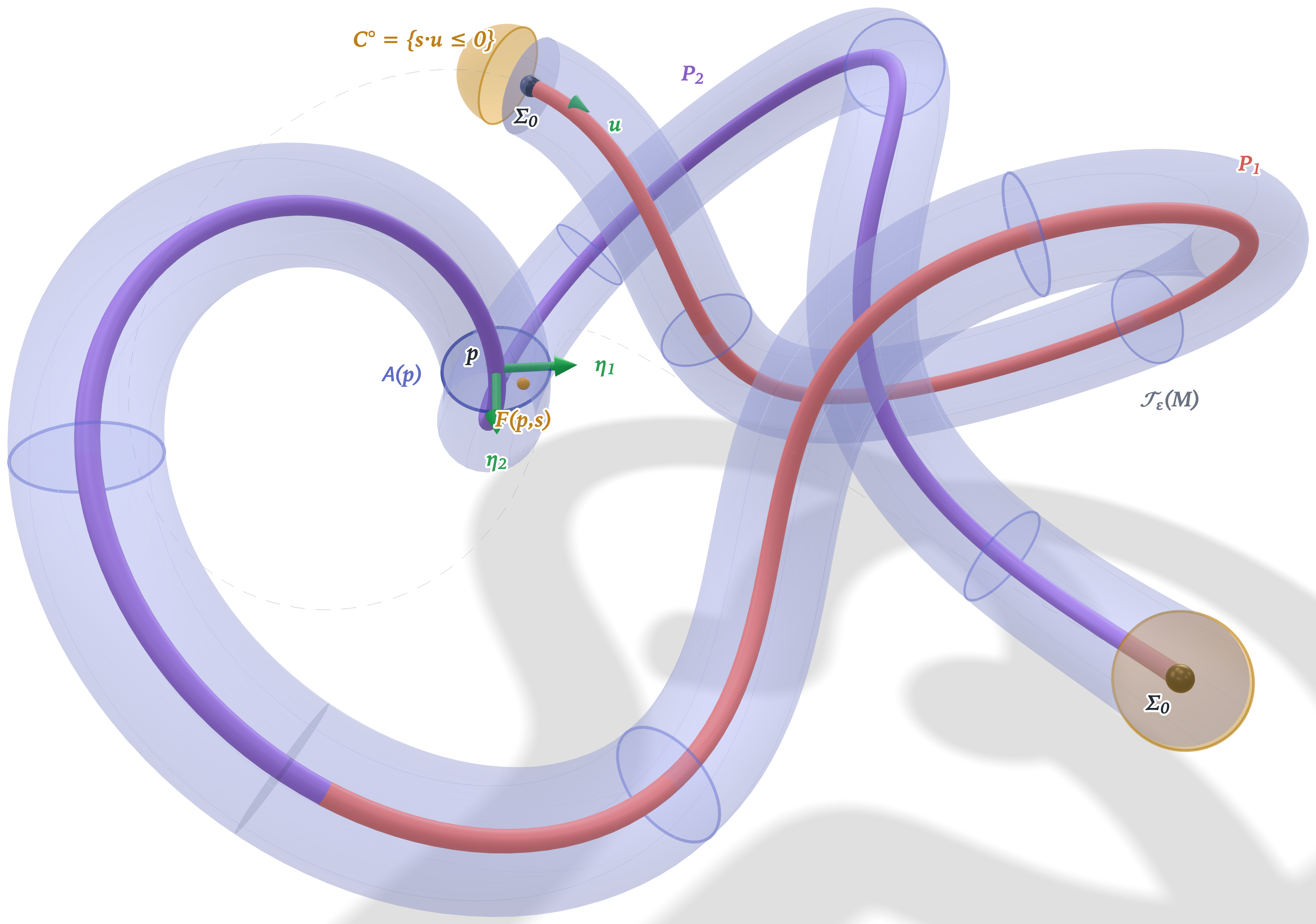}
\caption{The geometric data of the normal-cell calculus for a knotted curve
$M\subset\R^3$ with two endpoints.  The curve is split into base pieces $P_1$
and $P_2$, and the endpoints form a $0$-dimensional stratum $\Sigma_0$.  Over
an interior base point $p$, the coordinate map $F(p,s)$ carried by the normal
frame $\eta_1,\eta_2$ fills the disk cell $\Afib(p)$; over each endpoint the
admissible cell is the polar half-ball $C^\circ=\{s:\inner{s}{u}\le0\}$
opposite the inward tangent $u$.  Together the pieces cover the tube
$\cT_\eps(M)$ exactly once.}
\label{fig:adaptedcells}
\end{figure}

On each piece one chooses anisotropic weights
$\bm\beta=(\beta_1,\ldots,\beta_d)$.  For $u>0$, the associated dilation is
\begin{equation*}
  D_u^{\bm\beta}t
  :=(u^{\beta_1}t_1,\ldots,u^{\beta_d}t_d).
\end{equation*}
When a positive scale coefficient $c_\alpha(p)$ is present, the original
fiber variable is rescaled as
$s=D_{(zc_\alpha(p))^{-1}}^{\bm\beta}t$.  The leading hypotheses say that the
rescaled phase, effective amplitude, and admissible cells converge to limiting
data, with
\begin{equation*}
  z\bigl(f(F_\alpha(p,D_{(zc_\alpha(p))^{-1}}^{\bm\beta}t))-f_0\bigr)
  \longrightarrow G_\alpha(p,t),
\end{equation*}
where $G_\alpha$ is homogeneous with respect to the chosen weights.  In the
ordinary normal-cell regime it is coercive in the ambient fiber variables.  In
a joint frontier rescaling it may instead be integrable only on the limiting
admissible cell; Theorem~\ref{thm:profile-leading} then uses direct
cellwise tightness.  The local coefficient is obtained by integrating the limiting profile and
amplitude over the limiting cell and then over the base piece.  The cells are
part of the computational disintegration.  After the pieces have been summed,
the total at each exponent is a coefficient of the original integral.

The main results are organized around the ways in which this blow-up can
converge.  The profile theorem gives a uniform leading term for a normalized
kernel with a stable shifted limit.  It assumes local convergence together with
exponential tightness on the rescaled cells.  Ambient coercivity is a convenient
sufficient condition, but is not required.  This single theorem therefore
covers both ordinary normal-cell models and joint frontier rescalings, in which
a frontier variable is combined with the normal variables before
Theorem~\ref{thm:profile-leading} is applied.  When the normally rescaled
contribution of a higher stratum is not uniformly integrable near an incident
frontier, a joint collar--normal blow-up captures the mass that escapes toward
that frontier while remaining geometrically assigned to the higher-stratum
piece.  The pointwise-cell
theorem addresses a different loss of uniformity:
the base-dependent coefficient may be unbounded but integrable.  After the
leading terms are obtained piece by piece, the global theorem compares their
powers and adds the first surviving order.  This additive form is important for
signed or complex amplitudes and for cases in which leading coefficients vanish
or cancel.  Higher-order expansions are obtained by first representing the
rescaled cell on a fixed reference cell and then expanding the complete
transported density.  This includes moving-cell corrections when a suitable
straightening is available; weighted jet criteria make the resulting
coefficients explicit.

\subsection*{Relation to existing real-Laplace theory}

The comparisons in this paper are with real, nonoscillatory Laplace integrals.
Stationary-phase integrals have a different selection mechanism: oscillatory
cancellation, rather than positive exponential decay, determines the contributing
critical sets.  We use Morse--Bott terminology only for the real-Laplace case of
a smooth minimum manifold with positive normal Hessian.

We use Wong~\cite{Wong} and Breitung~\cite{Breitung} as standard
references for the classical Laplace method, including multivariate
probability integrals and boundary-sensitive forms.  Higher-order and geometric
refinements include Kirwin's treatment of isolated, possibly nonquadratic
minima~\cite{Kirwin}, Novoseltsev's generalized Laplace calculation for
deformed heat traces near smooth fixed-point submanifolds~\cite{Novoseltsev},
Ludewig's additive expansions for separated Morse--Bott minimum
manifolds~\cite{Ludewig}, and the more recent geometric formulation of
L\'eger--Vialard~\cite{LegerVialard}.  These works cover smooth
isolated-well and clean minimum-manifold regimes.  The present paper adds
several anisotropic weights, supplied almost-everywhere disjoint admissible
cells, relative domains, incident-stratum assignments, and joint frontier
rescalings.

At the level of normalized exponential measures, Hwang gives a classical
weak-convergence formulation of Laplace's method and identifies the limit under
smooth compact minimizer-manifold hypotheses~\cite{Hwang}.  More recently, De
Bortoli--Desolneux obtain quantitative $W_1$ convergence for norm-like
potentials under a generalized-Jacobian condition, using the coarea
formula~\cite{DeBortoliDesolneux}.  These works concern the limiting normalized
measure; the results below primarily compute unnormalized piecewise
coefficients, from which normalized posterior limits follow when the leading
total coefficient is nonzero.

For moving rare-event sets, Barbe develops asymptotic approximations for
integrals of probability densities over sets receding to infinity, with
applications including quadratic forms, suprema of random linear forms, and
random matrices~\cite{BarbeAsymptoticSets}.  This moving-set viewpoint is close
in spirit to the base--fiber geometry used here; the local framework below
further allows supplied a.e.-disjoint cells, relative or incident strata, and
anisotropic nonquadratic dilations.

For the homogeneous integration layer, Lasserre relates sublevel integrals of
positively homogeneous functions to non-Gaussian whole-space
integrals~\cite{Lasserre}, while Bui--Randles develop a generalized
polar-coordinate formula for continuous one-parameter groups of generally
anisotropic dilations~\cite{BuiRandles}.  Tubular and Fermi-coordinate geometry
are standard; we cite Gray and Federer for geometric background
only~\cite{GrayTubes,Federer}.  Error bounds for nondegenerate interior and
smooth-boundary Laplace approximations are developed by
\L{}api\'nski~\cite{Lapinski}; boundary and rank-deficient positive-semidefinite
Laplace asymptotics in quantum-state tomography are treated by
Six--Rouchon~\cite{SixRouchon}.  Recent perturbative higher-order estimates for
nondegenerate minima are given by
Fukuda--Kagaya--Ueda~\cite{FukudaKagayaUeda}.  Bell-polynomial formulas for
classical Laplace coefficients are reviewed by Nemes, together with earlier
recurrences~\cite{Nemes}; the recurrence below is a weighted graded adaptation.

The novelty claim is at the level of this \emph{combined} hypothesis package.
The assumptions are not linearly ordered with every alternative method:
resolution of singularities, for example, treats analytic singularities that
need not be represented by a single supplied homogeneous normal-cell model.
Along the geometric and analytic axes relevant to the present calculus,
however, the assumptions are substantially weaker and the admitted local
configurations substantially more general than in any individual real-Laplace
result cited above.  They simultaneously allow a compact minimum set with a
finite $C^2$ stratification; Borel base refinements, $C^1$ adapted charts, and
measurable base-dependent cells; relative domains and incident-stratum
assignments; different fiber dimensions and anisotropic weights on different
pieces; a continuous phase and merely measurable amplitude; homogeneous models
that need not be smooth or ambiently coercive; moving-cell convergence only in
weighted measure, or pointwise convergence with an integrable base majorant;
and joint rescaling of frontier and normal variables.  To the best of our
knowledge, no existing theorem for real Laplace integrals accommodates this
full combination of geometric and analytic assumptions.  The price for this
generality is explicit: the adapted normal-cell disintegration and the weighted
limiting model are supplied hypotheses to be verified in each application,
rather than consequences of a general stratification theorem.

Resolution-of-singularities methods form a different but important neighboring
approach for positive real Laplace integrals.  In singular learning theory they
produce powers and logarithmic multiplicities through real log canonical
thresholds; see Watanabe and Lin~\cite{WatanabeAGSLT,LinRLCT}.  The present
hypotheses concern a complementary regime in which a weighted principal model
is supplied directly and its tails are controlled either by ambient coercivity
or on the limiting admissible cells.  Geometrically degenerate local models also arise in sub-Riemannian heat-kernel
asymptotics; Neel--Sacchelli provide a recent comparison through the structure
of minimizing geodesics and the cut locus~\cite{NeelSacchelli}.

For the homogeneous Weibullian-chaos application, Barbe's asymptotic-set
framework is an earlier broad geometric treatment of Weibull-type tail
integrals, with applications to quadratic forms, random linear forms, and
random matrices~\cite{BarbeAsymptoticSets}.  Hashorva--Korshunov--Piterbarg
derive tail and density expansions for sufficiently smooth homogeneous
functionals of centered Gaussian vectors~\cite{HashorvaKorshunovPiterbarg}.
In the Gaussian setting the radial law is independent of direction, so angular
selection is governed by maximizers of the homogeneous functional.  The
direction-dependent radial degeneracies studied below are outside that setting.

Neither mechanism studied in Section~\ref{sec:weibull-frontier} can arise there.
Both are produced by letting the radial cost depend on direction and degenerate,
which the Gaussian and elliptic settings exclude by construction.  If the cost
vanishes on a set interior to $\{H>0\}$, the argument of the incomplete-gamma
function stays of order one on the dominant angular scale, the large-argument
expansion loses its uniformity, and the tail becomes polynomial rather than
exponentially small.  If the cost instead vanishes as the minimizing directions
approach $\partial\{H>0\}$, the inward coordinate contributes a one-sided scale
alongside the tangential Gaussian one.  The regularity used here is
correspondingly weaker: the angular data enter through measurability away from
the dominant set together with a continuous homogeneous limit model for the
effective cost near it, and derivatives are used only for the $C^2$ regularity
of $H$ and the $C^1$ collar that describe the frontier.

The inverse-kinematics example is positioned similarly.  Probabilistic and
Bayesian formulations of inverse kinematics, including sequential Monte Carlo
and biomechanical posterior inference, appear in
\cite{CourtyArnaudSMCIK,PatakyBayesianIK}.  Self-motion manifolds of redundant
manipulators are standard in the robotics literature~\cite{BurdickSelfMotion},
and planar $3R$ self-motion curves with joint limits were studied by
Lenar\v{c}i\v{c}~\cite{Lenarcic3R}.  What Section~\ref{sec:robotics-example}
contributes is the small-noise asymptotic splitting of this posterior: a regular
self-motion curve and an isolated one-sided joint-limit fold contribute to the
normalizing constant at the same power of the precision, so the limiting
posterior has both a continuous component and an atom.

\subsection*{Organization of the paper}

Section~\ref{sec:setup} gives the localization and normal-cell assumptions,
including relative domains and polyhedral normal fans.  Section~\ref{sec:leading}
proves the profile theorem under cellwise tightness, its coercive normal-cell
specialization, pointwise-cell limits, joint frontier applications, global
comparison, quantitative estimates, and vanishing amplitudes.
Section~\ref{sec:expansions} proves finite expansions after a general
scaled-cell representation and gives weighted jet criteria after cell
straightening.  Sections~\ref{sec:quadratic-models}--\ref{sec:radial} recover
smooth Morse--Bott, conic quadratic, Euclidean, and radial special cases; the
positive-semidefinite example appears in Section~\ref{sec:quadratic-models}.
Section~\ref{sec:geometric-scope} records the scope of the geometric hypotheses
and the excluded singular features.  Sections~\ref{sec:robotics-example} and
\ref{sec:weibull-frontier} contain the inverse-kinematics and Weibullian-chaos
applications.

\section{Geometric and analytic setup}\label{sec:setup}

Throughout, $(N,g)$ is a $C^2$ Riemannian manifold of dimension $n$ and
$\dvol_g$ denotes its Riemannian volume measure.  No orientability assumption is
needed.  Let $f:N\to\R$ be continuous and bounded below, let
$a:N\to\C$ be measurable, and write
\begin{equation}\label{eq:def-M}
  f_0:=\inf_N f,\qquad M:=\{x\in N:f(x)=f_0\}.
\end{equation}
We assume that $M$ is nonempty and compact.  For $z>0$, write
\begin{equation}\label{eq:standing-global-integral}
  I(z):=\int_N a(x)e^{-zf(x)}\dvol_g(x).
\end{equation}

We use the following notation throughout.  For $E\subset N$ and $\eps>0$, put
\begin{equation*}
  \cT_\eps(E):=\{x\in N:\dist_g(x,E)<\eps\}.
\end{equation*}
In a Euclidean space of the currently indicated dimension,
$B_r(x):=\{s:|s-x|<r\}$ and $B_r:=B_r(0)$.  We write $\R_+:=[0,\infty)$,
$\mathbb N:=\{1,2,\ldots\}$, and
$\mathbb N_0:=\{0,1,2,\ldots\}$.  For $x\in\R$,
$x_+:=\max\{x,0\}$ and $x_-:=\max\{-x,0\}$.  For a measurable set $E$,
$\ind_E$ denotes its indicator and $|E|$ its Lebesgue measure.  We write
$\mathcal H^k$ for $k$-dimensional Hausdorff measure and
$\vol_V(E)$ for the volume of $E$ in a specified Euclidean or Riemannian
space $V$; the subscript is omitted when the ambient space is clear.  If $P$
is a smooth base piece or submanifold, $\dvol_P$ denotes the Riemannian volume
measure induced by the ambient metric.  The notation $A\Subset B$ means that
$\overline A$ is compact and contained in the interior of $B$.  Euclidean
fiber integrals are taken with respect to Lebesgue measure unless stated
otherwise.

\subsection{Localization}

The first step is the standard positivity localization for real Laplace
integrals; compare Ludewig's localization lemma in the Morse--Bott setting and
classical accounts of Laplace's method~\cite{Ludewig,Wong}.

\begin{assumption}[Global integrability and separation]\label{ass:global}
There is $z_*>0$ such that
\begin{equation}\label{eq:A0}
  \int_N |a(x)|e^{-z_*f(x)}\dvol_g(x)<\infty.
\end{equation}
Moreover, for every sufficiently small tubular radius $\eps>0$ there is
$\delta_\eps>0$ such that
\begin{equation}\label{eq:A1}
  f(x)\ge f_0+\delta_\eps
  \qquad\text{for }x\in N\setminus \cT_\eps(M).
\end{equation}
\end{assumption}

\begin{lemma}[Exponential localization]\label{lem:localization}
Under Assumption~\ref{ass:global},
\begin{equation}\label{eq:localization}
  I(z)=I_\eps(z)+O\bigl(e^{-z(f_0+\delta_\eps)}\bigr),
  \qquad
  I_\eps(z):=\int_{\cT_\eps(M)}a(x)e^{-zf(x)}\dvol_g(x).
\end{equation}
The implicit constant is independent of $z\ge z_*$.
\end{lemma}

\begin{proof}
For $z\ge z_*$, Assumption~\ref{ass:global} gives
\begin{align*}
 |I(z)-I_\eps(z)|
 &\le \int_{N\setminus\cT_\eps(M)}
       |a|e^{-z_*f}e^{-(z-z_*)f}\dvol_g \\
 &\le e^{-(z-z_*)(f_0+\delta_\eps)}
       \int_N|a|e^{-z_*f}\dvol_g.
\end{align*}
Absorbing the fixed factor $e^{z_*(f_0+\delta_\eps)}$ proves the claim.
\end{proof}

\subsection{Stratified adapted normal-cell coordinates}

Assume that $M$ is the disjoint union of finitely many $C^2$ embedded strata,
\begin{equation}\label{eq:stratification}
  M=\bigsqcup_{j=0}^{\ell}\Sigma_j,
\end{equation}
where $\Sigma_j$ is a possibly disconnected $j$-dimensional submanifold.  Empty
strata are allowed.  A Whitney stratification is a natural source of the smooth
pieces, but the proofs below use only the explicit normal-cell hypothesis that
follows.  We call strata \emph{incident} when their closures meet along a lower
dimensional stratum; this is the situation in which ordinary disjoint tubular
neighborhoods around the smooth pieces do not by themselves provide a usable
local decomposition.

The next hypothesis is formulated in supplied coordinate pieces and measurable
admissible cells, rather than in canonical polar cones.  In nonsmooth,
finite-type, relative-domain, or incident-stratum problems the useful local
coordinates may be adapted to a curved valley, an active constraint, a
singular-value chart, or a collar.  The familiar Fermi-coordinate construction,
with tangent cones and polar normal cells, is treated below as an important
special source of such charts; the general theorem requires only the stated
change-of-variables formula.

Let $\cA$ be a finite index set.  For each $\alpha\in\cA$ we are given:
\begin{itemize}[leftmargin=2.2em]
\item a Borel set $P_\alpha$ contained in a stratum of dimension $j_\alpha$,
  with finite $j_\alpha$-dimensional Riemannian volume;
\item an open neighborhood $U_\alpha$ of $P_\alpha$ inside that stratum;
\item an integer $d_\alpha:=n-j_\alpha$ and a $C^1$ map
  \begin{equation}\label{eq:F-alpha}
    F_\alpha:U_\alpha\times B_\eps\longrightarrow N,
    \qquad F_\alpha(p,0)=p;
  \end{equation}
\item a measurable family of admissible coordinate cells
  $\Afib_\alpha(p)\subset\R^{d_\alpha}$ such that
  \begin{equation}\label{eq:measurable-cell-graph}
    \{(p,s):p\in P_\alpha,\ s\in\Afib_\alpha(p)\}
  \end{equation}
  is measurable in the product $\sigma$-algebra.
\end{itemize}
Here $p$ is the \emph{base variable}, $s$ is the \emph{fiber} or
\emph{normal-cell variable}, and $\Afib_\alpha(p)$ is the admissible cell in
the local $s$-coordinates.  A cell is simply a measurable fiber set;
convexity or conicity is imposed only when explicitly stated.

The family $\{P_\alpha\}_{\alpha\in\cA}$ is a finite disjoint Borel
refinement of the stratification:
\[
  M=\bigsqcup_{\alpha\in\cA}P_\alpha,
  \qquad P_\alpha\subset \Sigma_{j_\alpha}.
\]
Thus the index $\alpha$ need not label a stratum.  A single smooth stratum may
be split into several $P_\alpha$'s in order to use different coordinate charts,
normal frames, relative-domain cells, or measurable tie-breaking rules; only in
the simplest case does one take $P_\alpha$ to be an entire stratum or connected
component.  When $d_\alpha=0$, we use the conventions $\R^0=\{0\}$,
$\Afib_\alpha(p)=\{0\}$, $F_\alpha(p,0)=p$, and
$J_\alpha(p,0)=1$.

For $d_\alpha>0$, put
\begin{equation}\label{eq:D-alpha}
  \fD_\alpha(\eps):=
  \{(p,s):p\in P_\alpha,\ s\in\Afib_\alpha(p)\cap B_\eps\}.
\end{equation}
For $d_\alpha=0$, put $\fD_\alpha(\eps)=P_\alpha\times\{0\}$.

\begin{assumption}[Stratified adapted normal-cell change of variables]\label{ass:tube}
After decreasing $\eps>0$ if necessary, the following hold.
\begin{enumerate}[(T1),leftmargin=3em]
\item Each $F_\alpha$ is defined and $C^1$ on $U_\alpha\times B_\eps$, and
  $F_\alpha(p,0)=p$.
\item The sets $F_\alpha(\fD_\alpha(\eps))$ are pairwise disjoint up to
  $\dvol_g$-null sets and cover $\cT_\eps(M)$ up to a $\dvol_g$-null set.
\item There is a measurable Jacobian
  $J_\alpha:\fD_\alpha(\eps)\to(0,\infty)$ such that for every nonnegative
  measurable $h$,
  \begin{equation}\label{eq:cov}
    \int_{F_\alpha(\fD_\alpha(\eps))}h(x)\dvol_g(x)
    =\int_{P_\alpha}\int_{\Afib_\alpha(p)\cap B_\eps}
      h(F_\alpha(p,s))J_\alpha(p,s)\dd s\dvol_{P_\alpha}(p).
  \end{equation}
\end{enumerate}
All suprema over an empty set are understood as zero.
\end{assumption}

\begin{figure}[!htb]
\centering
\includegraphics[width=0.62\linewidth]{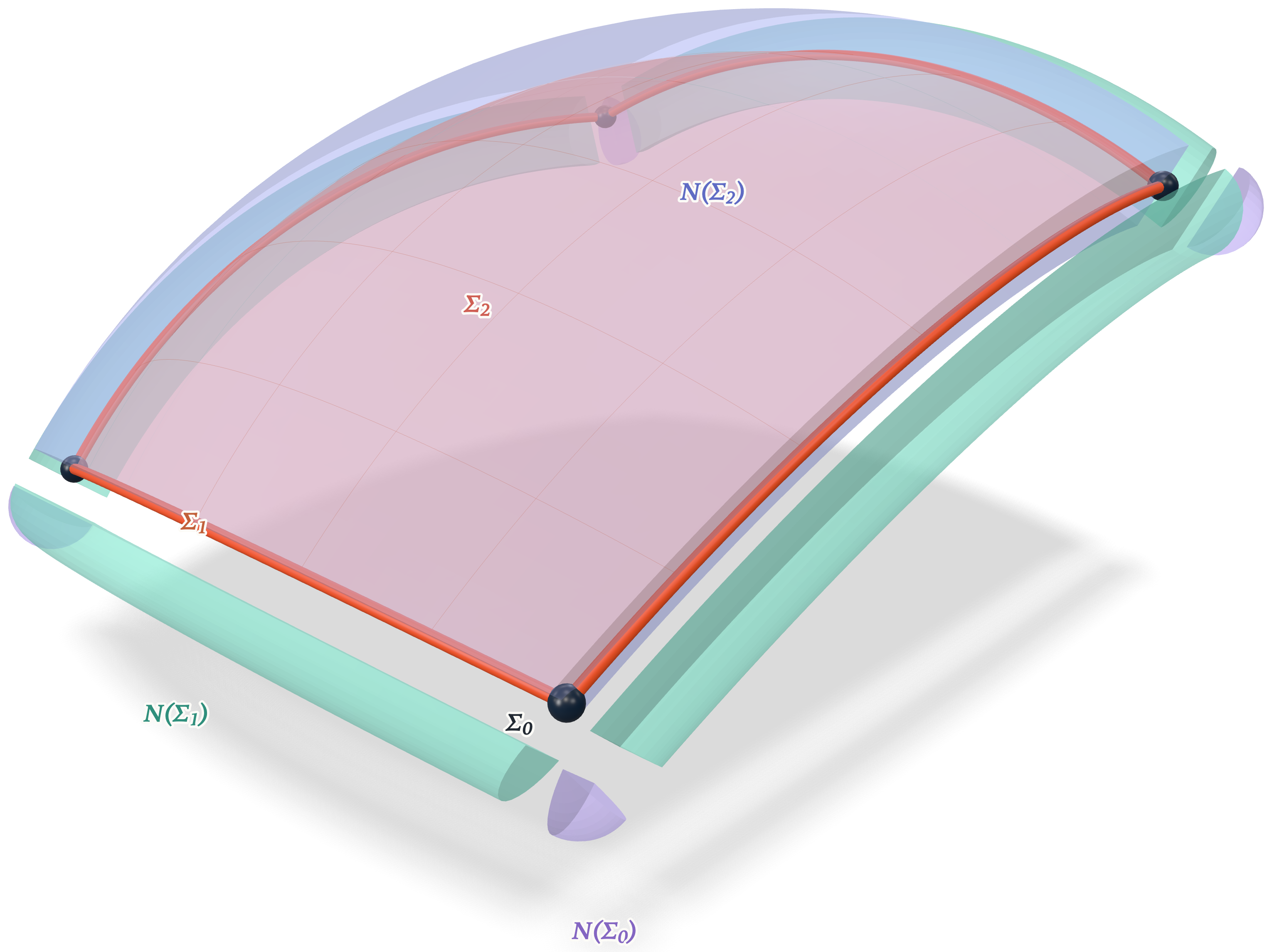}
\caption{The stratified normal-cell decomposition of Assumption~\ref{ass:tube}
for a curved patch $M\subset\R^3$ with corners, stratified as
$M=\Sigma_2\sqcup\Sigma_1\sqcup\Sigma_0$ (face, edge arcs, corner points).
Splitting the tube $\cT_\eps(M)$ by nearest stratum yields a two-sided full
slab $N(\Sigma_2)$ over the face, half-tubes $N(\Sigma_1)$ over the edges, and
quarter-balls $N(\Sigma_0)$ over the corners; the corresponding admissible
cells $\Afib_\alpha(p)$ are the polar cones of
Remark~\ref{rem:fermi-polar-cells}: the full normal line ($d_\alpha=1$, both
sides), a half-plane ($d_\alpha=2$), and a solid quarter-cone
($d_\alpha=3$).  The pieces are pulled slightly apart for visibility; they
meet only along shared null walls, so the tube is covered exactly once, as
required in (T2).}
\label{fig:corners}
\end{figure}

\begin{remark}[Interpretation of the normal-cell data]
\label{rem:normal-cell-interpretation}
The maps $F_\alpha$ and the cells $\Afib_\alpha(p)$ are part of the supplied
local description.  They may come from Fermi coordinates, from nonlinear
coordinates adapted to a finite-type valley or a constraint, or from a collar
near a frontier.  The asymptotic theorems below use only the
change-of-variables formula, the weighted limits of the cells, and the effective
amplitude
\[
  b_\alpha(p,s):=a(F_\alpha(p,s))J_\alpha(p,s).
\]
Thus the leading homogeneous expansion is imposed in the chosen adapted
coordinates.
\end{remark}

\begin{remark}[Fermi coordinates, tangent cones, and polar cells]
\label{rem:fermi-polar-cells}
Classical Fermi and tube-coordinate geometry provides one common source of
Assumption~\ref{ass:tube}~\cite{GrayTubes,Federer}.  Suppose that
$P_\alpha$ lies in a $C^2$ stratum and admits a $C^1$ orthonormal frame
$\eta_{\alpha,1},\ldots,\eta_{\alpha,d_\alpha}$ of the normal bundle of that
stratum.  The Fermi map is
\[
  F_\alpha(p,s)=\exp_p\left(\sum_{i=1}^{d_\alpha}s_i\eta_{\alpha,i}(p)\right).
\]
At $p\in M$, let $T_pM$ be the Bouligand tangent cone of the full minimum set
and define its polar cone by
\begin{equation}\label{eq:normal-cone}
  N_pM:=\{v\in T_pN:\inner{v}{w}_p\le0\text{ for every }w\in T_pM\}.
\end{equation}
When the nearest-point/Fermi construction is valid, an admissible coordinate
cell is the representation of this cone in the chosen normal frame,
\[
 C_\alpha(p):=\left\{s\in\R^{d_\alpha}:
      \sum_i s_i\eta_{\alpha,i}(p)\in N_pM\right\}.
\]
At a smooth interior point of the minimum set this gives the full normal space;
at a boundary, corner, or lower incident stratum it gives a proper cone or a
cone cut out by the surrounding pieces.  For an integral over a relative domain
$\Omega$, the same construction may be intersected with the inward condition
imposed by $\Omega$.  The abstract theorem does not require this Fermi-polar
construction: one may split a stratum into several chart pieces, assign an
overlap region to one piece by a measurable rule, restrict a polar cone to an
inward sector, or use nonlinear adapted coordinates in which the phase has its
homogeneous leading model.
\end{remark}

\begin{remark}[Jacobian normalization as a sufficient condition]
\label{rem:jacobian-normalization}
In many Fermi tubes with uniform Jacobian control, the coordinate Jacobian
itself is normalized at the zero section:
\begin{equation}\label{eq:J-uniform}
  \lim_{r\downarrow0}
  \sup_{\alpha\in\cA}\sup_{p\in P_\alpha}
  \sup_{s\in\Afib_\alpha(p),\,|s|\le r}
  |J_\alpha(p,s)-1|=0.
\end{equation}
This is a sufficient condition, not part of Assumption~\ref{ass:tube}.  The
asymptotic theorems use the Jacobian through the effective amplitude
$(a\circ F_\alpha)J_\alpha$, or through the corresponding factor in product,
collar, or singular-value coordinates.
\end{remark}

\begin{remark}[Relative-domain convention and boundary maxima]\label{rem:relative-domain}
Every result below has the following relative-domain form.  For an integral over
a measurable set $\Omega\subset N$, conditions (T2) and (T3) in
Assumption~\ref{ass:tube} are read with
$\Omega\cap\cT_\eps(M)$ in place of $\cT_\eps(M)$ and with the corresponding
inward cells $\Afib_\alpha(p)$.  The proofs are unchanged because they use only
the resulting change-of-variables formula, convergence of the effective
amplitude, and weighted limits of the cells.  Smooth-boundary Laplace formulas
and rank-deficient positive-semidefinite boundary examples provide standard
special cases~\cite{Lapinski,SixRouchon}.  A boundary maximum of a function $h$
is reduced to the present minimum convention by taking $f=h_{\max}-h$.
\end{remark}

The terminology and basic structure of normal fans are standard in polyhedral
convex geometry~\cite{LuRobinson}.

\begin{figure}[!htb]
\centering
\includegraphics[width=0.48\linewidth]{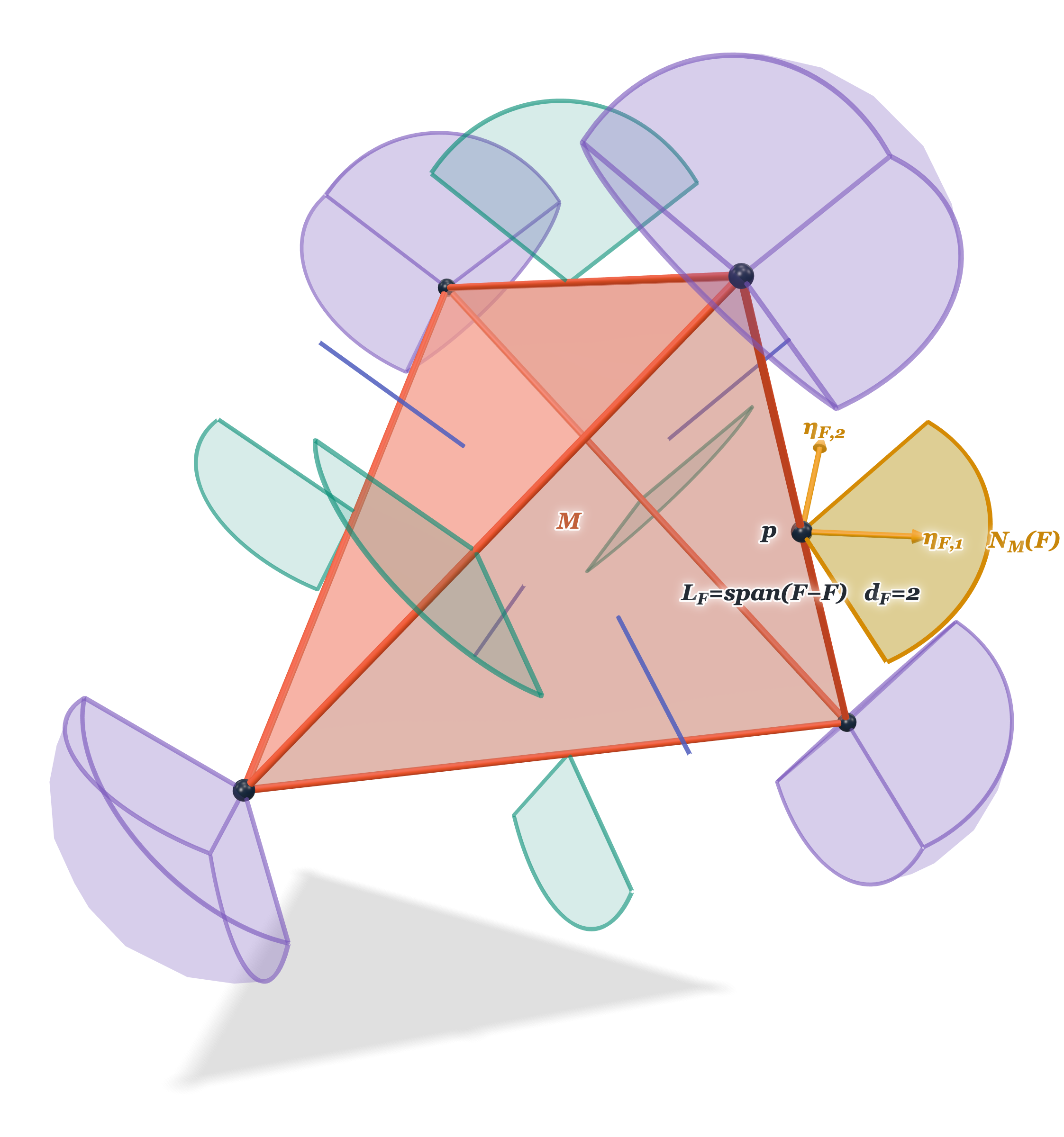}
\caption{The convex-polyhedral normal fan of
Proposition~\ref{prop:polyhedral-normal-fan} for a tetrahedron
$M\subset\R^3$: rays over the facets ($d=1$, blue), planar wedges over the
edges ($d=2$, green), solid cones over the vertices ($d=3$, violet).  The
highlighted face shows $p\in\relint F$, the frame $\eta_{F,i}$, and the
normal cone $N_M(F)$ (amber).  Translating every cone to a common origin
tiles $\R^n$ --- the normal fan.}
\label{fig:normalfan}
\end{figure}

\begin{proposition}[Convex-polyhedral normal fan]
\label{prop:polyhedral-normal-fan}
Let $N=\R^n$ and let $M\subset\R^n$ be a nonempty compact convex polyhedron,
possibly of dimension smaller than $n$.  Stratify $M$ by the relative interiors
$P_F:=\relint F$ of its nonempty faces.  For a face $F$, put
$L_F:=\operatorname{span}(F-F)$, choose a fixed orthonormal frame
$\eta_{F,1},\ldots,\eta_{F,n-\dim F}$ of $L_F^\perp$, and let
$C_F\subset\R^{n-\dim F}$ be the coordinate representation of the convex
normal cone
\[
 N_M(F):=\{v\in\R^n:\inner{v}{y-p}\le0\text{ for every }y\in M\},
 \qquad p\in\relint F.
\]
The cone is independent of the chosen $p\in\relint F$.  With
$\Afib_F(p):=C_F$,
\[
 F_F(p,s)=p+\sum_{i=1}^{n-\dim F}s_i\eta_{F,i},
 \qquad J_F\equiv1,
\]
Assumption~\ref{ass:tube} holds for every $\eps>0$.
\end{proposition}

\begin{proof}
Every $x\in\R^n$ has a unique metric projection $\pi_M(x)\in M$.  The convex
projection characterization gives
\[
 p=\pi_M(x)
 \quad\Longleftrightarrow\quad
 \inner{x-p}{y-p}\le0\quad\text{for every }y\in M.
\]
The point $p$ belongs to the relative interior of a unique face $F$, so
$x-p\in N_M(F)$.  Hence the sets
\[
 \{p+v:p\in\relint F,\ v\in N_M(F),\ |v|<\eps\}
\]
are pairwise disjoint and their union is exactly $\cT_\eps(M)$.  Relative to
the orthogonal splitting $L_F\oplus L_F^\perp$, the map $(p,s)\mapsto p+s$
has determinant one.  The change-of-variables formula therefore holds with
$J_F\equiv1$; in particular, the optional normalization~\eqref{eq:J-uniform} holds.
\end{proof}

By Assumption~\ref{ass:tube},
\begin{equation}\label{eq:sum-pieces}
  I_\eps(z)=\sum_{\alpha\in\cA}I_\alpha(z),
  \qquad
  I_\alpha(z)=\int_{P_\alpha}\cI_\alpha(z,p)\dvol_{P_\alpha}(p),
\end{equation}
where
\begin{equation}\label{eq:fiber-integral}
  \cI_\alpha(z,p)
  :=\int_{\Afib_\alpha(p)\cap B_\eps}
  a(F_\alpha(p,s))J_\alpha(p,s)e^{-zf(F_\alpha(p,s))}\dd s.
\end{equation}
Roman $I_\alpha(z)$ denotes the contribution of the whole base piece,
whereas $\cI_\alpha(z,p)$ denotes its fiber integral at $p$.  Once one
positive-codimensional piece is fixed, we suppress the index $\alpha$ from
$P_\alpha$, $F_\alpha$, $J_\alpha$, $\Afib_\alpha$, and $\cI_\alpha$.

\section{Leading-order normal-cell asymptotics}\label{sec:leading}

The auxiliary moment and tail estimates in this section use homogeneous shell
arguments under the weighted dilation.  Related homogeneous integration and anisotropic polar
formulas appear in Lasserre and in Bui--Randles~\cite{Lasserre,BuiRandles}.

We first introduce the anisotropic dilations used throughout the paper.
Fix $d\ge1$ and $\bm\beta=(\beta_1,\ldots,\beta_d)\in(0,\infty)^d$.  For
$u>0$, put
\begin{equation}\label{eq:dilation}
  D_u^{\bm\beta}t:=(u^{\beta_1}t_1,\ldots,u^{\beta_d}t_d),
  \qquad Q_{\bm\beta}:=\sum_{i=1}^d\beta_i.
\end{equation}
Throughout Sections~\ref{sec:leading} and~\ref{sec:expansions}, $s$ denotes an
unscaled fiber coordinate and $t$ a rescaled fiber coordinate.  After the
weights are fixed, we often abbreviate $Q:=Q_{\bm\beta}$.  The letter $q>0$
denotes a small scaling parameter tending to zero, typically
$q=(zc(p))^{-1}$ when a scale coefficient $c(p)$ is present.
When they occur, $c_-$ and $c_+$ bound positive scale coefficients,
$g_-$ and $g_+$ bound homogeneous models on the weighted unit sphere, and
$\kappa_G$ denotes a lower comparison constant between an exact or rescaled
phase and its model $G$.  We reserve $\mathcal K$ for profile kernels,
$\mathcal L$ for coefficients obtained after integration over one fiber, and
roman $K_\alpha$ for coefficients obtained after integration over the base
piece $P_\alpha$.
For $t\ne0$, let $\rho_{\bm\beta}(t)$ be the unique number $r>0$ such that
$|D_{1/r}^{\bm\beta}t|=1$, and set $\rho_{\bm\beta}(0)=0$.  Then
\begin{equation}\label{eq:gauge-properties}
  \rho_{\bm\beta}(D_u^{\bm\beta}t)=u\rho_{\bm\beta}(t),
  \qquad
  \{t:\rho_{\bm\beta}(t)<r\}=D_r^{\bm\beta}B_1,
\end{equation}
and the latter set has Euclidean volume $r^{Q_{\bm\beta}}\vol(B_1)$.

\begin{lemma}[Uniform exponential moments]\label{lem:moments}
Let $P$ be any parameter set and let
$G:P\times\R^d\to[0,\infty)$ be measurable.  Suppose
\begin{equation}\label{eq:G-homogeneous}
  G(p,D_u^{\bm\beta}t)=uG(p,t)
\end{equation}
for all $p,t,u$, and suppose that for some $0<g_-\le g_+<\infty$,
\begin{equation}\label{eq:G-sphere-bounds}
  g_-\le G(p,t)\le g_+
  \qquad\text{whenever }\rho_{\bm\beta}(t)=1.
\end{equation}
Then
\begin{equation}\label{eq:G-gauge-bounds}
  g_-\rho_{\bm\beta}(t)\le G(p,t)\le g_+\rho_{\bm\beta}(t)
  \qquad(p\in P,\ t\in\R^d),
\end{equation}
and for every $r\ge0$ and $\gamma>0$,
\begin{equation}\label{eq:uniform-moments}
  \sup_{p\in P}\int_{\R^d}(1+\rho_{\bm\beta}(t))^r
  e^{-\gamma G(p,t)}\dd t<\infty.
\end{equation}
The tails in~\eqref{eq:uniform-moments} converge to zero uniformly in $p$.
\end{lemma}

\begin{proof}
For $t\ne0$ write $t=D_{\rho_{\bm\beta}(t)}^{\bm\beta}\theta$ with
$\rho_{\bm\beta}(\theta)=1$; then~\eqref{eq:G-homogeneous}
and~\eqref{eq:G-sphere-bounds} give~\eqref{eq:G-gauge-bounds}.  Homogeneity
forces $G(p,0)=0$, so~\eqref{eq:G-gauge-bounds} holds at $t=0$ as well.
The volume identity in~\eqref{eq:gauge-properties} and a
sum over the shells $k\le\rho_{\bm\beta}(t)<k+1$ show that the left-hand side
of~\eqref{eq:uniform-moments} is bounded by a constant multiple of
\(
  \sum_{k\ge0}(1+k)^{r+Q_{\bm\beta}}e^{-\gamma g_- k},
\)
which is finite.  The same estimate, with the sum started at a large $k$,
gives the uniform tail statement.
\end{proof}

\begin{remark}[Coercivity in the weighted gauge]
The lower bound in~\eqref{eq:G-sphere-bounds} is coercivity with respect to the
weighted gauge: the principal model is bounded away from zero on the weighted
unit sphere, uniformly in the base parameter.  This uniformity is what makes the
constants in~\eqref{eq:G-gauge-bounds}, and hence the moment and tail bounds,
independent of $p$.
\end{remark}

The next identity is a diagonal-anisotropic conic version of Lasserre's
homogeneous sublevel formula; Bui--Randles provides a broader polar-integration
framework for anisotropic one-parameter dilations~\cite{Lasserre,BuiRandles}.

\begin{proposition}[Anisotropic sublevel formula]\label{prop:sublevel-formula}
Let $C\subset\R^d$ be measurable and invariant under $D_u^{\bm\beta}$.
Let $G:C\to[0,\infty)$ be homogeneous of degree one and suppose that
\[
 g_-\rho_{\bm\beta}(t)\le G(t)\le g_+\rho_{\bm\beta}(t)\qquad(t\in C)
\]
for some $0<g_-\le g_+<\infty$.  Suppose that $B_\kappa:C\to\C$ is measurable,
\[
 B_\kappa(D_u^{\bm\beta}t)=u^\kappa B_\kappa(t),
 \qquad \kappa>-Q_{\bm\beta},
\]
and is bounded on $C\cap\{\rho_{\bm\beta}=1\}$.  Then both integrals in~\eqref{eq:sublevel-formula} are absolutely
convergent and
\begin{equation}\label{eq:sublevel-formula}
 \int_C B_\kappa(t)e^{-G(t)}\dd t
 =\Gamma(Q_{\bm\beta}+\kappa+1)
  \int_{C\cap\{G\le1\}}B_\kappa(t)\dd t.
\end{equation}
In particular,
\begin{equation}\label{eq:sublevel-volume}
 \int_C e^{-G(t)}\dd t
 =\Gamma(Q_{\bm\beta}+1)\,
  \vol\bigl(C\cap\{G\le1\}\bigr).
\end{equation}
\end{proposition}

\begin{proof}
Coercivity and homogeneity imply absolute integrability at infinity, while
$\kappa>-Q_{\bm\beta}$ gives local integrability at the origin.  Hence Fubini's
theorem may be applied to
$e^{-G(t)}=\int_{G(t)}^\infty e^{-r}\dd r$.  If
\[
 V(r):=\int_{C\cap\{G\le r\}}B_\kappa(t)\dd t,
\]
then the change of variables $t=D_r^{\bm\beta}u$ gives
$V(r)=r^{Q_{\bm\beta}+\kappa}V(1)$.  Consequently,
\[
 \int_C B_\kappa e^{-G}\dd t
 =\int_0^\infty e^{-r}V(r)\dd r
 =\Gamma(Q_{\bm\beta}+\kappa+1)V(1).
\]
\end{proof}

We use the standard notation for the upper incomplete gamma function; see the
NIST DLMF for definitions and asymptotics~\cite{DLMF}.

\begin{proposition}[Homogeneous profile sublevel formula]
\label{prop:profile-sublevel-formula}
Retain the assumptions and notation of
Proposition~\ref{prop:sublevel-formula}, and put
\[
 h:=Q_{\bm\beta}+\kappa>0.
\]
Let $\mathcal K:[0,\infty)\to\C$ be measurable and suppose that
\begin{equation}\label{eq:profile-mellin-integrability}
 \int_0^\infty |\mathcal K(r)|r^{h-1}\dd r<\infty.
\end{equation}
Then the integral on the left of~\eqref{eq:profile-sublevel-formula}
is absolutely convergent and
\begin{equation}\label{eq:profile-sublevel-formula}
 \int_C B_\kappa(t)\mathcal K(G(t))\dd t
 =h\left[\int_0^\infty \mathcal K(r)r^{h-1}\dd r\right]
  \int_{C\cap\{G\le1\}}B_\kappa(t)\dd t.
\end{equation}
In particular, for $\lambda>0$, let
\begin{equation}\label{eq:upper-incomplete-gamma}
 \Gamma(\lambda,r):=\int_r^\infty s^{\lambda-1}e^{-s}\dd s
\end{equation}
denote the upper incomplete gamma function.  Then
\begin{align}
 \int_C B_\kappa(t)\Gamma(\lambda,G(t))\dd t
 &=\Gamma(\lambda+h)
   \int_{C\cap\{G\le1\}}B_\kappa(t)\dd t
   \label{eq:gamma-profile-sublevel}\\
 &=\frac{\Gamma(\lambda+h)}{\Gamma(h+1)}
   \int_C B_\kappa(t)e^{-G(t)}\dd t.
   \label{eq:gamma-profile-exponential}
\end{align}
\end{proposition}

\begin{proof}
Let
\[
 V(r):=\int_{C\cap\{G\le r\}}B_\kappa(t)\dd t.
\]
As in the proof of Proposition~\ref{prop:sublevel-formula}, anisotropic
homogeneity gives $V(r)=r^hV(1)$.  Hence, first for interval indicators and
then for step functions,
\[
 \int_C B_\kappa(t)\mathcal K(G(t))\dd t
 =hV(1)\int_0^\infty \mathcal K(r)r^{h-1}\dd r.
\]
The same scaling argument applied to $|B_\kappa|$ shows that the error made
by replacing $\mathcal K$ with a step function is bounded by a constant multiple of its
$L^1((0,\infty),r^{h-1}\dd r)$ error.  Thus
\eqref{eq:profile-mellin-integrability} permits approximation by step functions
and proves~\eqref{eq:profile-sublevel-formula} for general measurable $\mathcal K$.
For $\mathcal K(r)=\Gamma(\lambda,r)$, Tonelli's theorem gives
\[
 h\int_0^\infty \Gamma(\lambda,r)r^{h-1}\dd r
 =\int_0^\infty s^{\lambda-1}e^{-s}s^h\dd s
 =\Gamma(\lambda+h).
\]
This proves~\eqref{eq:gamma-profile-sublevel}; combining it with
\eqref{eq:sublevel-formula} proves~\eqref{eq:gamma-profile-exponential}.
\end{proof}

\begin{remark}[Separable coefficients]\label{rem:separable-coefficients}
For
\[
 G(t)=\sum_{i=1}^d c_i|t_i|^{m_i},\qquad
 \beta_i=\frac1{m_i},\qquad c_i,m_i>0,
\]
and a product cone whose $i$th factor is either $\R$ or a coordinate
half-line, the leading coefficient factors into one-dimensional integrals.
The $i$th factor is
\[
 2c_i^{-1/m_i}\Gamma\!\left(1+\frac1{m_i}\right)
 \quad\text{on }\R,
 \qquad
 c_i^{-1/m_i}\Gamma\!\left(1+\frac1{m_i}\right)
 \quad\text{on a half-line}.
\]
This gives a direct closed form for many corner and lower-stratum models.
\end{remark}

\begin{remark}[Affine inequalities in weighted cell limits]
The following proposition is stated for conic, homogeneous
linear inequalities.  If affine offsets are present, their leading scaled part
must be included in the limiting form before the slab estimate is applied.
\end{remark}

\begin{proposition}[Uniform polyhedral weighted cell limits]
\label{prop:polyhedral-cone-limit}
Suppose
\[
 \Afib(p)=\{s\in\R^d:\ell_k(p,s)\ge0,\ 1\le k\le m\},
 \qquad
 \ell_k(p,s)=\sum_{i=1}^d a_{ki}(p)s_i,
\]
where the coefficients are measurable and uniformly bounded.  For every $k$,
assume that there is a number $\delta_k\in\{\beta_1,\ldots,\beta_d\}$,
independent of $p$, such that
\begin{equation}\label{eq:polyhedral-leading-coefficients}
 a_{ki}(p)=0\quad\text{if }\beta_i<\delta_k,
 \qquad
 \sum_{\beta_i=\delta_k}|a_{ki}(p)|^2\ge a_0^2>0.
\end{equation}
Define
\[
 \bm a_k^0(p):=
 \bigl(a_{ki}(p)\ind_{\{\beta_i=\delta_k\}}\bigr)_{i=1}^d,
 \qquad
 \ell_k^0(p,t):=\inner{\bm a_k^0(p)}{t},
\]
and
\[
 \Afib^\infty(p):=\{t:\ell_k^0(p,t)\ge0,\ 1\le k\le m\}.
\]
Let $G:P\times\R^d\to[0,\infty)$ satisfy the hypotheses of
Lemma~\ref{lem:moments} uniformly in $p$.  Then, for every $\eps>0$ and
$0<\delta<1$,
\begin{equation}\label{eq:polyhedral-convergence}
 \sup_{p\in P}\int_{\R^d}
 \left|\ind_{D_u^{\bm\beta}(\Afib(p)\cap B_\eps)}(t)
       -\ind_{\Afib^\infty(p)}(t)\right|
 e^{-(1-\delta)G(p,t)}\dd t\longrightarrow0
 \qquad(u\to\infty).
\end{equation}
In particular, Assumption~\ref{ass:leading-model}
\eqref{eq:weighted-cell-limit} follows when $u=zc(p)$ and $c(p)\ge c_->0$.
\end{proposition}

\begin{proof}
We separate the proof into the bounded part of the scaled variables and the
uniform tail.  First ignore the Euclidean truncation $B_\eps$ and write the
inequalities defining the scaled cell in the $t$--variables.  Since
$s=D_{1/u}^{\bm\beta}t$,
\[
 t\in D_u^{\bm\beta}\Afib(p)
 \quad\Longleftrightarrow\quad
 \ell_{k,u}(p,t):=u^{\delta_k}\ell_k(p,D_{1/u}^{\bm\beta}t)\ge0,
 \quad 1\le k\le m.
\]
The normalization by $u^{\delta_k}$ keeps the first nonzero weighted part of
the $k$th inequality at order one.  Indeed
\[
 \ell_{k,u}(p,t)=\ell_k^0(p,t)+r_{k,u}(p,t),
\]
where $r_{k,u}$ contains only variables whose weights are strictly larger than
$\delta_k$.  Because there are only finitely many weights, there is
$\eta>0$ such that, for every fixed anisotropic ball
$B_T^{\bm\beta}:=\{\rho_{\bm\beta}\le T\}$,
\begin{equation}\label{eq:polyhedral-remainder}
 \sup_{p,k,t\in B_T^{\bm\beta}}|r_{k,u}(p,t)|
 \le C_{T,\mathrm{rem}}u^{-\eta}.
\end{equation}

Now fix $T$.  If, for some $p$ and $t\in B_T^{\bm\beta}$, the indicators of
$D_u^{\bm\beta}\Afib(p)$ and $\Afib^\infty(p)$ differ, then the truth value of
at least one defining inequality changes between
$\ell_{k,u}(p,t)=\ell_k^0(p,t)+r_{k,u}(p,t)$ and $\ell_k^0(p,t)$.  Hence either
$\ell_k^0<0\le \ell_k^0+r_{k,u}$ or
$\ell_k^0\ge0>\ell_k^0+r_{k,u}$, and in both cases
\[
 |\ell_k^0(p,t)|\le |r_{k,u}(p,t)|
 \le C_{T,\mathrm{rem}}u^{-\eta}.
\]
Thus the exceptional set lies in the thin slab around the zero hyperplane of
this limiting inequality.

The vector $\bm a_k^0(p)$ is the Euclidean normal to the hyperplane
$\ell_k^0(p,t)=0$, and its norm is at least $a_0$ by
\eqref{eq:polyhedral-leading-coefficients}.  Since
$B_T^{\bm\beta}=D_T^{\bm\beta}B_1$ is contained in the Euclidean ball of
radius $R_T:=\max_iT^{\beta_i}$, every hyperplane section of
$B_T^{\bm\beta}$ has $(d-1)$-dimensional measure at most
$\omega_{d-1}R_T^{d-1}$, where $\omega_{d-1}$ is the volume of the unit ball
in $\R^{d-1}$.  The coarea formula for
$t\mapsto\ell_k^0(p,t)$ therefore gives
\begin{align}
 \bigl|\{t\in B_T^{\bm\beta}:|\ell_k^0(p,t)|\le h\}\bigr|
 &=\frac1{|\bm a_k^0(p)|}\int_{-h}^{h}
   \mathcal H^{d-1}
   \bigl(B_T^{\bm\beta}\cap\{\ell_k^0(p,t)=r\}\bigr)\dd r\notag\\
 &\le 2\omega_{d-1}R_T^{d-1}\frac{h}{a_0}.
 \label{eq:polyhedral-slab-bound}
\end{align}
Taking $h=C_{T,\mathrm{rem}}u^{-\eta}$ in
\eqref{eq:polyhedral-slab-bound} and summing over the finitely many
inequalities gives
\[
 \sup_{p\in P}\int_{B_T^{\bm\beta}}
 \left|\ind_{D_u^{\bm\beta}\Afib(p)}-
       \ind_{\Afib^\infty(p)}\right|\dd t\longrightarrow0.
\]
Multiplying by the bounded weight $e^{-(1-\delta)G(p,t)}$ proves the same
local convergence with the weighted measure.

It remains to remove the restriction to $B_T^{\bm\beta}$ and to restore the
original truncation $\Afib(p)\cap B_\eps$.  The uniform coercive lower bound in
Lemma~\ref{lem:moments} makes
\[
 \sup_{p\in P}\int_{\{\rho_{\bm\beta}>T\}}
 e^{-(1-\delta)G(p,t)}\dd t
\]
tend to zero as $T\to\infty$.  Finally, for each fixed $T$ and all sufficiently
large $u$, the bounded set $B_T^{\bm\beta}$ is contained in
$D_u^{\bm\beta}B_\eps$.  Thus the Euclidean truncation can only affect the
same uniformly small tail.  Combining the local slab estimate with this tail
bound proves~\eqref{eq:polyhedral-convergence}.
\end{proof}

\begin{remark}[Uniformity of the leading faces]
When all weights are equal, $D_u^{\bm\beta}$ is a scalar dilation and every
ordinary cone is invariant.  For unequal weights, coordinate orthants and
product cones remain invariant.  Proposition~\ref{prop:polyhedral-cone-limit}
also allows some moving polyhedral cells.  More concretely, each defining
inequality must have a first nonzero weighted part, that part must involve the
same weight level for all $p$, and its coefficient vector must stay uniformly
away from zero.  Thus a rotating half-space
$\cos\theta(p)s_1+\sin\theta(p)s_2\ge0$ is allowed when
$\beta_1=\beta_2$ and the coefficient vector has norm one.  If
$\beta_1<\beta_2$, however, the coefficient of $s_1$ must either vanish
identically or be uniformly separated from zero; otherwise the limiting face
can switch type as $p$ varies.
\end{remark}

Fix a normal-cell piece $\alpha\in\cA$ with $d_\alpha>0$ and suppress
the index $\alpha$ where no confusion can arise.  Thus $P=P_\alpha$, $d=d_\alpha$, $\Afib(p)=\Afib_\alpha(p)$,
$F=F_\alpha$, $J=J_\alpha$, and $\cI(z,p)=\cI_\alpha(z,p)$.

\begin{assumption}[Uniform coercive anisotropic model]\label{ass:leading-model}
There exist
\begin{itemize}[leftmargin=2.2em]
\item an anisotropy vector $\bm\beta\in(0,\infty)^d$ and
  $Q:=Q_{\bm\beta}$;
\item a measurable function $c:P\to(0,\infty)$ satisfying
  \begin{equation}\label{eq:c-bounds}
    0<c_-\le c(p)\le c_+<\infty;
  \end{equation}
\item a measurable function $G:P\times\R^d\to[0,\infty)$ satisfying
  \eqref{eq:G-homogeneous} and~\eqref{eq:G-sphere-bounds} uniformly in $p$;
\item a measurable remainder $\mathcal R$ such that, on $\Afib(p)\cap B_\eps$,
  \begin{equation}\label{eq:phase-leading}
    f(F(p,s))-f_0=c(p)G(p,s)(1+\mathcal R(p,s)),
  \end{equation}
  with $\mathcal R(p,0)=0$, and
  \begin{equation}\label{eq:R-uniform}
    \lim_{r\downarrow0}\sup_{p\in P}
    \sup_{s\in\Afib(p),\,\rho_{\bm\beta}(s)\le r}|\mathcal R(p,s)|=0;
  \end{equation}
\item a measurable family of limit cells
  $\Afib^\infty(p)\subset\R^d$ such that, for every $0<\delta<1$,
  \begin{equation}\label{eq:weighted-cell-limit}
  \begin{split}
   \sup_{p\in P}\int_{\R^d}
   \Big|&\ind_{D_{zc(p)}^{\bm\beta}(\Afib(p)\cap B_\eps)}(t)
        -\ind_{\Afib^\infty(p)}(t)\Big|\\
   &\hspace{4em}\times e^{-(1-\delta)G(p,t)}\dd t
   \longrightarrow0
   \qquad(z\to\infty).
  \end{split}
  \end{equation}
\end{itemize}
The radius $\eps$ is chosen small enough that $|\mathcal R(p,s)|\le\delta_0<1$
throughout the coordinate domain.  Finally, define the effective amplitude
\begin{equation}\label{eq:effective-amplitude}
  b(p,s):=a(F(p,s))J(p,s).
\end{equation}
Assume that $b$ is bounded on the truncated cell graph and that there is a
bounded measurable function $b_0:P\to\C$ such that
\begin{equation}\label{eq:amplitude-uniform}
  \lim_{r\downarrow0}\sup_{p\in P}
  \sup_{s\in\Afib(p),\,\rho_{\bm\beta}(s)\le r}
  |b(p,s)-b_0(p)|=0.
\end{equation}
\end{assumption}

\begin{remark}[Fermi amplitudes]
If the adapted chart is a Fermi map, written after choosing an orthonormal
normal frame as $F(p,s)=\exp_p(\sum_i s_i\eta_i(p))$, and the tube is uniform
near the zero section--that is, $F(p,s)\to p$ uniformly as $s\to0$ and the
Jacobian normalization~\eqref{eq:J-uniform} holds--then continuity of $a$ on
the closure of the tube implies~\eqref{eq:amplitude-uniform} with
$b_0(p)=a(p)$.  In general adapted coordinates, $b_0$ includes both the
coordinate Jacobian and the leading change in the amplitude.
\end{remark}

\begin{lemma}[Scaled convergence principle]
\label{lem:scaled-integral-convergence}
Let $P$ be a parameter set and let $E_z(p),E_\infty(p)\subset\R^d$ be
measurable families of sets.  Let $H_z,H_\infty:P\times\R^d\to\C$ be
ambient measurable functions.  Their integrals are restricted to $E_z(p)$
and $E_\infty(p)$, respectively; replacement of the two cells is controlled
separately in~(iii).  Suppose that, for the weighted gauge
$\rho_{\bm\beta}$,
\begin{enumerate}[label=(\roman*),leftmargin=2.2em]
\item for every $T<\infty$,
\[
 \sup_{p\in P}\int_{E_z(p)\cap\{\rho_{\bm\beta}\le T\}}
 |H_z(p,t)-H_\infty(p,t)|\dd t\longrightarrow0;
\]
\item the two families are uniformly tight,
\[
 \lim_{T\to\infty}\limsup_{z\to\infty}\sup_{p\in P}
 \left[
 \int_{E_z(p)\cap\{\rho_{\bm\beta}>T\}}|H_z(p,t)|\dd t
 +\int_{E_\infty(p)\cap\{\rho_{\bm\beta}>T\}}|H_\infty(p,t)|\dd t
 \right]=0;
\]
\item the limiting integrand is stable under replacement of the cells,
\[
 \sup_{p\in P}\int_{\R^d}
 |\ind_{E_z(p)}(t)-\ind_{E_\infty(p)}(t)|\,
 |H_\infty(p,t)|\dd t\longrightarrow0.
\]
\end{enumerate}
Then
\begin{equation}\label{eq:scaled-convergence-principle}
 \sup_{p\in P}\left|
 \int_{E_z(p)}H_z(p,t)\dd t
 -\int_{E_\infty(p)}H_\infty(p,t)\dd t
 \right|\longrightarrow0.
\end{equation}
Moreover, if $P$ is a measure space, the fiber convergence in
\eqref{eq:scaled-convergence-principle} holds for almost every $p$, and the
fiber integrals are dominated by an integrable function on $P$, then the same convergence
holds after integration over $P$.
\end{lemma}

\begin{proof}
Put $B_T:=\{\rho_{\bm\beta}\le T\}$.  For each $p$,
\begin{align*}
 \int_{E_z(p)}H_z(p,t)\dd t-\int_{E_\infty(p)}H_\infty(p,t)\dd t
 ={}&\int_{E_z(p)\cap B_T}\bigl(H_z-H_\infty\bigr)(p,t)\dd t\\
 &+\int_{B_T}\bigl(\ind_{E_z(p)}-\ind_{E_\infty(p)}\bigr)(t)
 H_\infty(p,t)\dd t\\
 &+\int_{E_z(p)\setminus B_T}H_z(p,t)\dd t
 -\int_{E_\infty(p)\setminus B_T}H_\infty(p,t)\dd t.
\end{align*}
After taking absolute values and the supremum over $p$, the first term tends to
zero by~(i), the second by~(iii), and the last two are uniformly small by~(ii)
when $T$ is large.  Letting first $z\to\infty$ and then $T\to\infty$ proves the
uniform statement.  The final claim is the dominated convergence theorem
applied to the fiber integrals.
\end{proof}

The rescaled model need not be coercive on the whole ambient space.  The
scaled cells are instead required to converge and to retain their mass in a
weighted integrable region.  Uniform coercivity is one sufficient condition;
direct tightness on the rescaled cells is another.

\begin{theorem}[Uniform profile asymptotics under cellwise tightness]
\label{thm:profile-leading}
Let $P$ be a measurable parameter space, let $U\subset\R^d$ be a fixed bounded
neighborhood of the origin, and let $\Afib(p)\subset U$ and
$\Afib^\infty(p)\subset\R^d$ be measurable families of admissible cells; that is,
both cell graphs are measurable in the corresponding product $\sigma$-algebras.
Fix an anisotropy $\bm\beta$, put $Q:=Q_{\bm\beta}$, and let
$c:P\to(0,\infty)$ be measurable and satisfy~\eqref{eq:c-bounds}.  Let
$G:P\times\R^d\to[0,\infty)$ be measurable and suppose that, for some
$g_+<\infty$,
\begin{equation}\label{eq:profile-model-upper-bound}
 G(p,D_u^{\bm\beta}t)=uG(p,t),
 \qquad
 \sup_{p\in P}\sup_{\rho_{\bm\beta}(t)=1}G(p,t)\le g_+.
\end{equation}

Let $\mu\ge0$.  Let $\phi$ and $b$ be measurable on
\[
 \{(p,s)\in P\times\R^d:s\in\Afib(p)\},
\]
with $\phi(p,s)\ge\mu$, and let $b_0:P\to\C$ be bounded and measurable.
Let $\mathcal K:[0,\infty)\to\C$ be measurable, and put
\begin{equation}\label{eq:profile-integral-definition}
 \cI_{\mathcal K}(z,p):=\int_{\Afib(p)} b(p,s)\mathcal K(z\phi(p,s))\dd s.
\end{equation}
For
$q_z(p):=(zc(p))^{-1}$ define
\begin{equation}\label{eq:profile-rescaled-data}
 E_z(p):=D_{zc(p)}^{\bm\beta}\Afib(p),
 \qquad
 X_z(p,t):=z\bigl[\phi(p,D_{q_z(p)}^{\bm\beta}t)-\mu\bigr].
\end{equation}
Assume that $b$ is uniformly bounded on the graph of $\Afib$ and, for every
$T<\infty$,
\begin{equation}\label{eq:profile-local-convergence}
 \sup_{p\in P}
 \sup_{\substack{t\in E_z(p)\\ \rho_{\bm\beta}(t)\le T}}
 \left(
 |b(p,D_{q_z(p)}^{\bm\beta}t)-b_0(p)|
 +|X_z(p,t)-G(p,t)|
 \right)\longrightarrow0.
\end{equation}
Suppose also that for some $\kappa_G>0$ and all sufficiently large $z$,
\begin{equation}\label{eq:profile-scaled-lower-bound}
 X_z(p,t)\ge\kappa_G G(p,t),
 \qquad p\in P,\quad t\in E_z(p).
\end{equation}

Let $\mathcal N_{\mathcal K}(z)>0$ and define the shifted normalized profiles
\begin{equation}\label{eq:normalized-profile-family}
 \mathcal K_z(v):=\frac{\mathcal K(z\mu+v)}{\mathcal N_{\mathcal K}(z)},
 \qquad v\ge0.
\end{equation}
Assume that $\mathcal K_z\to \mathcal K_\infty$ locally uniformly on $[0,\infty)$ for a
continuous function $\mathcal K_\infty$, and that for some constants $C<\infty$,
$r\ge0$, and $\eta>0$,
\begin{equation}\label{eq:profile-envelope}
 |\mathcal K_z(v)|\le C(1+v)^r e^{-\eta v},
 \qquad v\ge0,
\end{equation}
for all sufficiently large $z$.

Assume finally that there exists
\begin{equation}\label{eq:profile-gamma-range}
 0<\gamma<\min\{1,\eta\kappa_G,\eta\}
\end{equation}
such that the cells converge in the weighted measure
\begin{equation}\label{eq:profile-cell-convergence}
 \sup_{p\in P}\int_{\R^d}
 \left|\ind_{E_z(p)}(t)-\ind_{\Afib^\infty(p)}(t)\right|
 e^{-\gamma G(p,t)}\dd t\longrightarrow0,
\end{equation}
and the exact scaled cells are uniformly tight:
\begin{equation}\label{eq:profile-cell-tightness}
 \lim_{T\to\infty}\limsup_{z\to\infty}\sup_{p\in P}
 \int_{E_z(p)\cap\{\rho_{\bm\beta}>T\}}
 e^{-\gamma G(p,t)}\dd t=0.
\end{equation}
Then the exact and limiting fiber integrals are absolutely convergent for all
sufficiently large $z$, and
\begin{equation}\label{eq:uniform-profile-leading}
 \sup_{p\in P}\left|
 \frac{(zc(p))^Q}{\mathcal N_{\mathcal K}(z)}\cI_{\mathcal K}(z,p)
 -b_0(p)\int_{\Afib^\infty(p)}\mathcal K_\infty(G(p,t))\dd t
 \right|\longrightarrow0.
\end{equation}
If $P$ is equipped with a finite measure $\nu$, then the coefficient below is
absolutely integrable and
\begin{multline}\label{eq:integrated-profile-leading}
 \int_P\cI_{\mathcal K}(z,p)\dd\nu(p)\\
 =\mathcal N_{\mathcal K}(z)z^{-Q}\left[
 \int_P c(p)^{-Q}b_0(p)
 \left(\int_{\Afib^\infty(p)}\mathcal K_\infty(G(p,t))\dd t\right)
 \dd\nu(p)+o(1)\right].
\end{multline}
\end{theorem}

\begin{proof}
The change of variables $s=D_{q_z(p)}^{\bm\beta}t$ gives
\begin{equation}\label{eq:profile-scaled-proof}
 \frac{(zc(p))^Q}{\mathcal N_{\mathcal K}(z)}\cI_{\mathcal K}(z,p)
 =\int_{E_z(p)}H_z(p,t)\dd t,
\end{equation}
where
\[
 H_z(p,t):=b(p,D_{q_z(p)}^{\bm\beta}t)\mathcal K_z(X_z(p,t)),
 \qquad
 H_\infty(p,t):=b_0(p)\mathcal K_\infty(G(p,t)).
\]
We verify Lemma~\ref{lem:scaled-integral-convergence}.  By
\eqref{eq:profile-model-upper-bound}, $G(p,t)\le g_+T$ whenever
$\rho_{\bm\beta}(t)\le T$.  Hence
\eqref{eq:profile-local-convergence}, local uniform convergence of $\mathcal K_z$, and
uniform continuity of $\mathcal K_\infty$ on compact intervals imply local $L^1$
convergence of $H_z$ to $H_\infty$, uniformly in $p$.

Passing to the locally uniform limit in~\eqref{eq:profile-envelope} gives
$|\mathcal K_\infty(v)|\le C(1+v)^r e^{-\eta v}$.  For every
$0<\theta<\eta$,
\[
 (1+v)^r e^{-\eta v}\le C_\theta e^{-\theta v},
 \qquad v\ge0.
\]
Apply this estimate with $\theta=\gamma/\kappa_G$ to $H_z$ and with
$\theta=\gamma$ to $H_\infty$.  The inequalities in
\eqref{eq:profile-gamma-range} and the lower bound
\eqref{eq:profile-scaled-lower-bound} then yield
\begin{equation}\label{eq:profile-common-majorant}
 |H_z(p,t)|\le C'e^{-\gamma G(p,t)}
 \quad(t\in E_z(p)),
 \qquad
 |H_\infty(p,t)|\le C'e^{-\gamma G(p,t)}.
\end{equation}
Thus~\eqref{eq:profile-cell-tightness} controls the exact tails.  The limiting
cells satisfy the same tightness estimate because, for every $T$,
\begin{align*}
 &\int_{\Afib^\infty(p)\cap\{\rho_{\bm\beta}>T\}}
 e^{-\gamma G(p,t)}\dd t\\
 &\quad\le
 \int_{E_z(p)\cap\{\rho_{\bm\beta}>T\}}
 e^{-\gamma G(p,t)}\dd t
 +\int_{\R^d}|\ind_{E_z(p)}-\ind_{\Afib^\infty(p)}|
 e^{-\gamma G(p,t)}\dd t.
\end{align*}
Taking the supremum over $p$, then the upper limit as $z\to\infty$, and finally
$T\to\infty$ proves uniform tightness of the limiting integrands.  Finally,
\eqref{eq:profile-cell-convergence} and
\eqref{eq:profile-common-majorant} give the cell-replacement condition in
Lemma~\ref{lem:scaled-integral-convergence}.  The lemma proves
\eqref{eq:uniform-profile-leading}.

The same estimates give a uniform bound for the limiting fiber integrals;
measurability follows from the measurable graph of $\Afib^\infty$ and Tonelli's
theorem.  The integrated statement follows from uniform convergence, the
bounds~\eqref{eq:c-bounds}, and finiteness of $\nu(P)$.
\end{proof}

\begin{remark}[Two routes to cellwise tightness]
\label{rem:profile-two-routes}
There are two principal ways to verify
\eqref{eq:profile-cell-convergence}--\eqref{eq:profile-cell-tightness}.  If $G$
satisfies the coercivity bounds~\eqref{eq:G-sphere-bounds}, then
Lemma~\ref{lem:moments} supplies tightness and the weighted cell
limit~\eqref{eq:weighted-cell-limit} supplies cell convergence.  Alternatively,
$G$ may vanish along unbounded ambient directions while $e^{-\gamma G}$ remains
integrable and tight on the admissible cells.  This occurs in joint frontier
rescalings after the frontier and normal variables have been combined into one
scaled fiber.
\end{remark}

\begin{corollary}[Coercive exponential leading term on one normal-cell piece]
\label{cor:coercive-leading-piece}
Under Assumptions~\ref{ass:tube} and~\ref{ass:leading-model}, define
\begin{equation}\label{eq:leading-profile}
  \cL(p):=b_0(p)\int_{\Afib^\infty(p)}e^{-G(p,t)}\dd t.
\end{equation}
Then
\begin{equation}\label{eq:uniform-leading}
  \sup_{p\in P}\left|
  e^{zf_0}(zc(p))^Q\cI(z,p)-\cL(p)
  \right|\longrightarrow0.
\end{equation}
Equivalently, there are functions $r_z:P\to\C$ with
$\sup_P|r_z|\to0$ such that
\begin{equation}\label{eq:additive-leading}
  \cI(z,p)=e^{-zf_0}z^{-Q}c(p)^{-Q}\bigl(\cL(p)+r_z(p)\bigr).
\end{equation}
In particular,~\eqref{eq:additive-leading} remains valid when $\cL(p)=0$.  If
$\inf_{p\in P}|\cL(p)|>0$, then the corresponding asymptotic equivalence is
uniform in $p$.
\end{corollary}

\begin{proof}
Apply Theorem~\ref{thm:profile-leading} with the admissible cell
$\Afib(p)\cap B_\eps$ and with
limiting cell $\Afib^\infty(p)$.
\[
 \phi(p,s)=f(F(p,s))-f_0,
 \qquad \mu=0,
 \qquad \mathcal K(y)=e^{-y},
 \qquad \mathcal N_{\mathcal K}(z)=1.
\]
Then $\mathcal K_z=\mathcal K_\infty=e^{-v}$ and the envelope~\eqref{eq:profile-envelope} is
immediate.  Assumption~\ref{ass:leading-model} gives
\[
 X_z(p,t)=G(p,t)\bigl(1+\mathcal R(p,D_{q_z(p)}^{\bm\beta}t)\bigr)
\]
with local uniform convergence to $G$ and, after shrinking $\eps$ if needed,
$X_z\ge (1-\delta_0)G$.  Choose
$0<\gamma<\min\{1,(1-\delta_0)/2,1/2\}$.  The weighted cell limit
\eqref{eq:weighted-cell-limit}, with $\delta=1-\gamma$, gives
\eqref{eq:profile-cell-convergence}, while Lemma~\ref{lem:moments} gives
\eqref{eq:profile-cell-tightness}.  Restoring the external factor $e^{-zf_0}$
gives~\eqref{eq:uniform-leading} and hence~\eqref{eq:additive-leading}.
\end{proof}

\begin{corollary}[Leading coefficient in Fermi coordinates]
\label{cor:fermi-leading-coefficient}
In Corollary~\ref{cor:coercive-leading-piece}, suppose that $F$ is a Fermi map
on the piece, that the optional normalization~\eqref{eq:J-uniform} holds on that
piece, and that $a$ is continuous at the zero section uniformly over $P$.
Then $b_0(p)=a(p)$ and
\[
 \cL(p)=a(p)\int_{\Afib^\infty(p)}e^{-G(p,t)}\dd t.
\]
\end{corollary}

\begin{proof}
The assumptions give $F(p,s)\to p$ and $J(p,s)\to1$ uniformly as
$s\to0$ in the adapted gauge.  Hence
$(a\circ F)J\to a|_P$, so
Corollary~\ref{cor:coercive-leading-piece} applies with $b_0=a|_P$.
\end{proof}

\begin{remark}[Adapted coordinates and effective amplitudes]
The profile theorem starts after a local change of variables has been
supplied; all coordinate Jacobians are included in the effective amplitude
$b(p,s)=a(F(p,s))J(p,s)$.  Product, collar, singular-value, and other adapted
coordinates are treated in the same way.  The exponential corollary is the case
$\mathcal K(y)=e^{-y}$.  When $f_0\ge0$, it may equivalently be obtained from
Theorem~\ref{thm:profile-leading} by taking $\phi=f\circ F$, $\mu=f_0$, and
$\mathcal N_{\mathcal K}(z)=e^{-zf_0}$.
\end{remark}

\begin{example}[A curved cell survives the anisotropic blow-up]
\label{ex:curved-anisotropic-cell}
Let
\[
 \Omega_\eps:=\{(s_1,s_2)\in\R^2:s_2\ge s_1^2,\ |s|<\eps\},
 \qquad
 f(s_1,s_2)=s_1^4+s_2^2,
 \qquad a\equiv1.
\]
The minimum set consists only of the origin.  We use the relative-domain
formulation with $P=\{0\}$, $F(s)=s$, $J\equiv1$, $c\equiv1$, and
$\mathcal R\equiv0$.

Corollary~\ref{cor:coercive-leading-piece} applies with
$\bm\beta=(1/4,1/2)$, $Q=3/4$, and
\[
 G(s_1,s_2)=s_1^4+s_2^2.
\]
Indeed,
\[
 D_u^{\bm\beta}\{s_2\ge s_1^2\}=\{s_2\ge s_1^2\}
 \qquad(u>0),
\]
so the relative cell is exactly invariant under the anisotropic dilation and
its truncated rescalings converge to
$\Afib^\infty=\{t_2\ge t_1^2\}$.
Corollary~\ref{cor:coercive-leading-piece} therefore gives
\[
 \int_{\Omega_\eps}e^{-zf(s)}\dd s
 =z^{-3/4}\left[
 \int_{t_2\ge t_1^2}e^{-(t_1^4+t_2^2)}\dd t+o(1)
 \right].
\]
The Euclidean tangent cone of $\Omega_\eps$ at the origin is the half-space
$\{t_2\ge0\}$, but it is not the limiting cell relevant to
Corollary~\ref{cor:coercive-leading-piece}: the curved boundary survives in
the leading coefficient because its geometry is balanced by the anisotropic
phase weights.
\end{example}

\begin{corollary}[Anisotropically invariant cones]\label{cor:invariant-cones}
In Corollary~\ref{cor:coercive-leading-piece}, suppose that the admissible cells are cones,
$\Afib(p)=C(p)$, and
\begin{equation}\label{eq:cone-invariance}
  D_u^{\bm\beta}C(p)=C(p)
  \qquad(p\in P,\ u>0).
\end{equation}
Then~\eqref{eq:weighted-cell-limit} holds with
$\Afib^\infty(p)=C(p)$.  Consequently,
\begin{equation}\label{eq:invariant-cone-leading}
  \cI(z,p)=e^{-zf_0}z^{-Q}c(p)^{-Q}
  \left[b_0(p)\int_{C(p)}e^{-G(p,t)}\dd t+o(1)\right]
\end{equation}
uniformly in $p$.
\end{corollary}

\begin{proof}
Under~\eqref{eq:cone-invariance},
\[
 D_{zc(p)}^{\bm\beta}(C(p)\cap B_\eps)
 =C(p)\cap D_{zc(p)}^{\bm\beta}B_\eps.
\]
The complements of the expanding ellipsoids
$D_{zc(p)}^{\bm\beta}B_\eps$ escape to infinity in the gauge
$\rho_{\bm\beta}$ uniformly in $p$, because $c(p)\ge c_->0$.
Lemma~\ref{lem:moments} therefore gives~\eqref{eq:weighted-cell-limit}.
\end{proof}

Quantitative real-Laplace error estimates for nondegenerate interior and
smooth-boundary extrema are developed by \L{}api\'nski~\cite{Lapinski}, while
recent perturbative higher-order estimates for nondegenerate minima are given
by Fukuda--Kagaya--Ueda~\cite{FukudaKagayaUeda}.  The estimate below is the
corresponding bound inside the present homogeneous normal-cell regime.

\begin{theorem}[Quantitative leading estimate]\label{thm:quantitative-leading}
Retain the hypotheses of Corollary~\ref{cor:invariant-cones}.  Put
\[
 b(p,s):=a(F(p,s))J(p,s),\qquad
 \phi(p,s):=\frac{f(F(p,s))-f_0}{c(p)}.
\]
Suppose that there is a bounded measurable function $b_0:P\to\C$ and
that, for some $\eta>0$ and constants $L<\infty$ and $\kappa_G>0$, uniformly
on the truncated fibers,
\begin{equation}\label{eq:quantitative-jets}
 |b(p,s)-b_0(p)|\le L\rho_{\bm\beta}(s)^\eta,
 \qquad
 |\phi(p,s)-G(p,s)|\le L\rho_{\bm\beta}(s)^{1+\eta},
\end{equation}
and
\begin{equation}\label{eq:quantitative-lower-bound}
 \phi(p,s)\ge\kappa_G G(p,s).
\end{equation}
Then there is a constant $C$ such that, for all sufficiently large $z$,
\begin{equation}\label{eq:quantitative-leading-result}
 \sup_{p\in P}\left|
 e^{zf_0}(zc(p))^Q\cI(z,p)
 -b_0(p)\int_{C(p)}e^{-G(p,t)}\dd t
 \right|\le Cz^{-\eta}.
\end{equation}
Consequently, after integration over a finite-volume base,
\[
 I_\alpha(z)=e^{-zf_0}z^{-Q}\bigl(K_\alpha+O(z^{-\eta})\bigr),
\]
where
\[
 K_\alpha:=\int_P c(p)^{-Q}b_0(p)
 \left(\int_{C(p)}e^{-G(p,t)}\dd t\right)\dvol_P(p).
\]
\end{theorem}

\begin{proof}
Set $q=(zc(p))^{-1}$ and
$x_q(p,t):=q^{-1}\phi(p,D_q^{\bm\beta}t)$.  On the rescaled coordinate
domain, homogeneity and~\eqref{eq:quantitative-jets} give
\[
 |b(p,D_q^{\bm\beta}t)-b_0(p)|
 \le Lq^\eta\rho_{\bm\beta}(t)^\eta,
 \qquad
 |x_q(p,t)-G(p,t)|
 \le Lq^\eta\rho_{\bm\beta}(t)^{1+\eta}.
\]
Also $x_q\ge\kappa_G G$.  With $\kappa_*:=\min\{1,\kappa_G\}$, the elementary
mean-value estimate
\[
 |e^{-x_q}-e^{-G}|
 \le |x_q-G|e^{-\min\{x_q,G\}}
 \le Lq^\eta\rho_{\bm\beta}(t)^{1+\eta}e^{-\kappa_* G}
\]
and the boundedness of $b_0$ give
\begin{align*}
 &\left|b(p,D_q^{\bm\beta}t)e^{-x_q(p,t)}
       -b_0(p)e^{-G(p,t)}\right|\\
 &\qquad\le Cq^\eta
 \bigl(\rho_{\bm\beta}(t)^\eta
       +\rho_{\bm\beta}(t)^{1+\eta}\bigr)
 e^{-\kappa_* G(p,t)}.
\end{align*}
Thus the integral error on the rescaled domain is $O(q^\eta)$, uniformly in
$p$, by Lemma~\ref{lem:moments}.  The omitted part of $C(p)$ lies outside
$D_{1/q}^{\bm\beta}B_\eps$.  Since
$\min_{|s|=\eps}\rho_{\bm\beta}(s)>0$, its gauge tends to infinity at rate
$q^{-1}$, and Lemma~\ref{lem:moments} makes this tail smaller than every power
of $q$.  Thus the error is $O(q^\eta)$ uniformly.  Finally,
$q^\eta\le c_-^{-\eta}z^{-\eta}$, proving~\eqref{eq:quantitative-leading-result}.
\end{proof}

\subsection{Integrated leading terms under pointwise cells or base degeneration}

The uniform profile theorem and its exponential corollary are convenient when
the cells, local models, and scale coefficient vary uniformly over the base.  The next result allows two common
degenerations.  The rescaled cells may converge only pointwise in the base,
and the coefficient $c(p)$ may approach zero near a frontier so that
$c(p)^{-Q}$ is not uniformly bounded.  Both effects are permitted as long as
the rescaled fiber integrals admit an integrable base majorant.

\begin{theorem}[Integrated leading term under pointwise cell convergence and an integrable base majorant]
\label{thm:integrated-pointwise-cells}
Fix one normal-cell piece with a common anisotropy $\bm\beta$ and put
$Q=Q_{\bm\beta}$.  Let $c:P\to(0,\infty)$ be measurable, with no uniform lower
bound assumed, and define
\[
 b(p,s):=a(F(p,s))J(p,s),\qquad
 \phi(p,s):=\frac{f(F(p,s))-f_0}{c(p)}.
\]
Let $G:P\times\R^d\to[0,\infty)$ be homogeneous and satisfy the uniform
coercivity bounds of Lemma~\ref{lem:moments}.  Suppose that there are a
measurable function $b_0:P\to\C$, measurable limit cells
$\Afib^\infty(p)\subset\R^d$, and constants $\kappa_G>0$ and
$0<\gamma\le\min\{1,\kappa_G\}$ with the following properties.
For almost every $p$,
\begin{equation}\label{eq:pointwise-cell-local-limits}
 b(p,s)\longrightarrow b_0(p),\qquad
 \frac{\phi(p,s)}{G(p,s)}\longrightarrow1
 \quad\text{as }s\to0\text{ within }\Afib(p),
\end{equation}
and throughout the truncated cells,
\begin{equation}\label{eq:pointwise-cell-lower}
 \phi(p,s)\ge\kappa_G G(p,s).
\end{equation}
Moreover, for almost every $p$,
\begin{equation}\label{eq:pointwise-weighted-cell-limit}
 \int_{\R^d}
 \left|\ind_{D_{zc(p)}^{\bm\beta}(\Afib(p)\cap B_\eps)}(t)
       -\ind_{\Afib^\infty(p)}(t)\right|
 e^{-\gamma G(p,t)}\dd t\longrightarrow0.
\end{equation}
Finally, suppose that there is a measurable $B:P\to[0,\infty)$ such that
\begin{equation}\label{eq:pointwise-cell-majorant}
 |b(p,s)|\le B(p),\qquad |b_0(p)|\le B(p),\qquad
 B(p)c(p)^{-Q}\in L^1(P,\dvol_P).
\end{equation}
Then, with
\begin{equation}\label{eq:pointwise-cell-profile}
 \cL(p):=b_0(p)\int_{\Afib^\infty(p)}e^{-G(p,t)}\dd t,
\end{equation}
one has
\begin{equation}\label{eq:pointwise-cell-result}
 I_\alpha(z)=e^{-zf_0}z^{-Q}
 \left[\int_P c(p)^{-Q}\cL(p)\dvol_P(p)+o(1)\right].
\end{equation}
\end{theorem}

\begin{proof}
Fix $p$ outside the exceptional null set and put $q=(zc(p))^{-1}$.  After the
change of variables $s=D_q^{\bm\beta}t$,
\begin{equation}\label{eq:pointwise-cell-scaled}
 e^{zf_0}z^Q\cI(z,p)
 =c(p)^{-Q}\int_{E_z(p)}H_z(p,t)\dd t,
\end{equation}
where
\[
 E_z(p)=D_{zc(p)}^{\bm\beta}(\Afib(p)\cap B_\eps),\qquad
 H_z(p,t)=b(p,D_q^{\bm\beta}t)
 e^{-q^{-1}\phi(p,D_q^{\bm\beta}t)}.
\]
Let $H_\infty(p,t)=b_0(p)e^{-G(p,t)}$.  Fix $T<\infty$.  For almost every
$t$ with $\rho_{\bm\beta}(t)\le T$, whenever $t\in E_z(p)$ one has
$D_q^{\bm\beta}t\to0$ within $\Afib(p)$, and
\eqref{eq:pointwise-cell-local-limits} gives
$H_z(p,t)\to H_\infty(p,t)$.  Moreover,
\[
 \ind_{E_z(p)}|H_z-H_\infty|(p,t)
 \le B(p)\bigl(e^{-\kappa_G G(p,t)}+e^{-G(p,t)}\bigr).
\]
Dominated convergence on the fixed ball therefore verifies condition~(i) of
Lemma~\ref{lem:scaled-integral-convergence}.  The lower bound
\eqref{eq:pointwise-cell-lower}, the majorant
\eqref{eq:pointwise-cell-majorant}, and Lemma~\ref{lem:moments} give the two
tail bounds required in condition~(ii).  Finally,
\eqref{eq:pointwise-weighted-cell-limit} verifies condition~(iii), since
$|H_\infty(p,t)|\le B(p)e^{-G(p,t)}\le B(p)e^{-\gamma G(p,t)}$.
Lemma~\ref{lem:scaled-integral-convergence}, applied with this single value of
$p$, yields
\[
 e^{zf_0}z^Q\cI(z,p)\longrightarrow c(p)^{-Q}\cL(p)
\]
for almost every $p$.

The same lower bound gives the uniform-in-$z$ base estimate
\[
 e^{zf_0}z^Q|\cI(z,p)|
 \le C B(p)c(p)^{-Q},
\]
where $C$ is independent of $p$ by Lemma~\ref{lem:moments}.  The right-hand
side is integrable by~\eqref{eq:pointwise-cell-majorant}; the integrated part
of Lemma~\ref{lem:scaled-integral-convergence}, equivalently dominated
convergence in $p$, proves~\eqref{eq:pointwise-cell-result}.
\end{proof}

\begin{remark}[Scope of the pointwise-cell theorem]
Theorem~\ref{thm:integrated-pointwise-cells} determines the integrated leading
term even when the usable cells collapse nonuniformly near a frontier.  It does
not by itself compute a subsequent correction created by the collapsing
collar; such a correction may require a joint rescaling followed by
Theorem~\ref{thm:profile-leading}, or a sectorwise decomposition into regions
with different scaled limits.
\end{remark}

\begin{example}[Collapsing cells: pointwise convergence without uniformity]
\label{ex:collapsing-cells}
Consider
\[
 I_{\mathrm{collapse}}(z):=\int_0^1\int_{-p}^{p}e^{-zs^2}\dd s\dd p.
\]
This is a single relative normal-cell piece with $P=(0,1)$,
$F(p,s)=(p,s)$, $\Afib(p)=(-p,p)$, $J\equiv1$, $b_0\equiv1$, and phase
$f(F(p,s))=s^2$.  For every fixed $p>0$, the rescaled cell
$\sqrt z\,(-p,p)$ converges in Gaussian-weighted indicator norm to $\R$.
The convergence is not uniform in $p$: for $p=z^{-1/2}$ the rescaled cell is
only $(-1,1)$.

The remaining hypotheses of
Theorem~\ref{thm:integrated-pointwise-cells} hold with
$c\equiv1$, $G(t)=t^2$, $B\equiv1$, and
$\Afib^\infty(p)=\R$.  In particular, the base majorant is integrable because
$P$ has finite length.  Hence that theorem, rather than the uniform
fiberwise theorem, gives
\[
 I_{\mathrm{collapse}}(z)\sim z^{-1/2}\int_0^1\int_\R e^{-t^2}\dd t\dd p
 =\sqrt\pi\,z^{-1/2}.
\]
This example isolates the purpose of the pointwise-cell theorem: the usable
normal cells may collapse arbitrarily close to a lower stratum while the
integrated leading coefficient remains unchanged.  No assertion is made here about lower-order collar corrections.
\end{example}

\subsection{The global leading term}

The global comparison needs only integrated piece expansions.  For each
$\alpha\in\cA$, assume that a local theorem or direct computation supplies
numbers $Q_\alpha\ge0$ and $K_\alpha\in\C$ such that
\begin{equation}\label{eq:piece-integrated-leading}
 I_\alpha(z)=e^{-zf_0}z^{-Q_\alpha}\bigl(K_\alpha+o(1)\bigr).
\end{equation}
For a zero-codimensional piece this holds with $Q_\alpha=0$ and
$K_\alpha=\int_{P_\alpha}a(p)\dvol_{P_\alpha}(p)$.  When the expansion comes
from Corollary~\ref{cor:coercive-leading-piece} or
Theorem~\ref{thm:integrated-pointwise-cells}, the coefficient has the form
\begin{equation}\label{eq:piece-coefficient}
  K_\alpha:=\int_{P_\alpha}c_\alpha(p)^{-Q_\alpha}
  \cL_\alpha(p)\dvol_{P_\alpha}(p).
\end{equation}
Joint rescalings, vanishing-amplitude results, and finite expansions may also
supply expansions of the form
\eqref{eq:piece-integrated-leading}.  In the clean
locally separated Morse--Bott case, an additive global expansion of this type is
given in Ludewig's Appendix~A~\cite{Ludewig}.

\begin{theorem}[Global leading term]\label{thm:global-leading}
Under Assumptions~\ref{ass:global} and~\ref{ass:tube}, and assuming the
piecewise expansions~\eqref{eq:piece-integrated-leading}, let
\begin{equation}\label{eq:Qstar}
  Q_*:=\min_{\alpha\in\cA}Q_\alpha,
  \qquad
  K_*:=\sum_{\alpha:Q_\alpha=Q_*}K_\alpha.
\end{equation}
Then
\begin{equation}\label{eq:global-additive}
 I(z)=e^{-zf_0}\sum_{\alpha\in\cA}
 z^{-Q_\alpha}\bigl(K_\alpha+r_\alpha(z)\bigr)
 +O\bigl(e^{-z(f_0+\delta_\eps)}\bigr),
\end{equation}
where $r_\alpha(z)\to0$ for every $\alpha$.  If $K_*\ne0$, then
\begin{equation}\label{eq:global-equivalence}
  I(z)\sim e^{-zf_0}z^{-Q_*}K_*.
\end{equation}
If $K_*=0$,~\eqref{eq:global-additive} yields only
$I(z)=o(e^{-zf_0}z^{-Q_*})$.  Identifying the first nonzero term requires
higher-order expansions on the pieces with $Q_\alpha=Q_*$; the leading
terms of pieces with larger $Q_\alpha$ cannot in general be compared with
the unspecified remainders of the dominant pieces.
\end{theorem}

\begin{proof}
Sum the finitely many piecewise expansions
\eqref{eq:piece-integrated-leading} using~\eqref{eq:sum-pieces}, and apply
Lemma~\ref{lem:localization}.  Pieces with $Q_\alpha>Q_*$ are
$o(e^{-zf_0}z^{-Q_*})$, while those with $Q_\alpha=Q_*$ sum to
$e^{-zf_0}z^{-Q_*}(K_*+o(1))$.  This proves~\eqref{eq:global-equivalence}
when $K_*\ne0$.
\end{proof}

\begin{remark}[Why dimension alone is misleading]
It is tempting to expect the top-dimensional stratum of $M$ to dominate: in a
purely quadratic model each normal variable contributes $1/2$ to $Q_\alpha$, so
larger strata have fewer normal variables and hence smaller exponents.  The
anisotropic case breaks this intuition.  A lower-dimensional component has
more transverse variables, but those variables may enter the phase only at high
orders; a variable of order $m$ contributes $1/m$, not $1/2$, to
$Q_\alpha$.  Thus sufficiently flat high-order directions can make a
lower-dimensional component dominate a quadratically nondegenerate
higher-dimensional one.  Moreover, even if the top stratum has the smallest
exponent, its integrated coefficient may cancel for a signed or complex
amplitude.
\end{remark}

\begin{example}[A lower-dimensional component dominates]\label{ex:lower-dim}
Let $N=\R^3$, $a\equiv1$, and
\[
 f(x,y,w)=\bigl((x^2+y^2-1)^2+w^2\bigr)
          \bigl(x^2+y^2+(w-2)^2\bigr)^4.
\]
Then
\[
 M=\{x^2+y^2=1,\ w=0\}\sqcup\{(0,0,2)\}.
\]
On the circular component we use the usual tubular coordinates
\[
 (x,y,w)=((1+\xi)\cos\theta,(1+\xi)\sin\theta,\eta),
 \qquad 0\le\theta<2\pi.
\]
Here $\xi$ is the signed radial displacement from the unit circle and $\eta$ is
the vertical displacement.  Then
\[
 x^2+y^2-1=(1+\xi)^2-1=2\xi+\xi^2,
 \qquad
 x^2+y^2+(w-2)^2=5+2\xi-4\eta+O(|(\xi,\eta)|^2),
\]
and therefore
\[
 f=625(4\xi^2+\eta^2)+O\bigl(|(\xi,\eta)|^3\bigr).
\]
Thus Corollary~\ref{cor:coercive-leading-piece} applies there with quadratic weights,
$Q_{\mathrm{circ}}=1$, and coefficient
\[
 K_{\mathrm{circ}}
 =(2\pi)\int_{\R^2}e^{-(2500t_1^2+625t_2^2)}\dd t
 =\frac{\pi^2}{625}.
\]
At $p=(0,0,2)$,
\[
 f(p+h)=5|h|^8\bigl(1+O(|h|)\bigr),
\]
so Corollary~\ref{cor:coercive-leading-piece} applies with weights
$(1/8,1/8,1/8)$,
$Q_p=3/8$, and positive coefficient
\[
 K_p
 =5^{-3/8}\int_{\R^3}e^{-|t|^8}\dd t
 =5^{-3/8}\frac{\pi}{2}\Gamma\!\left(\frac38\right).
\]
The global comparison in Theorem~\ref{thm:global-leading} selects the smaller
exponent $Q_p=3/8$ and therefore yields
\[
 I(z)\sim
 5^{-3/8}\frac{\pi}{2}\Gamma\!\left(\frac38\right)z^{-3/8}.
\]
Hence an isolated, sufficiently degenerate minimum can dominate a smooth
positive-dimensional minimum component.
\end{example}

\begin{corollary}[Comparison after cancellation]\label{cor:comparison-after-cancellation}
Suppose that every normal-cell piece satisfies, for some $\eta_\alpha>0$,
\begin{equation}\label{eq:piece-rate}
 I_\alpha(z)=e^{-zf_0}z^{-Q_\alpha}
 \bigl(K_\alpha+O(z^{-\eta_\alpha})\bigr).
\end{equation}
Let $Q^\dagger$ be one of the distinct exponents $Q_\alpha$.  Assume that for
every distinct exponent $Q'<Q^\dagger$,
\begin{equation}\label{eq:lower-level-cancellation}
 \sum_{\alpha:Q_\alpha=Q'}K_\alpha=0,
 \qquad
 Q'+\min_{\alpha:Q_\alpha=Q'}\eta_\alpha>Q^\dagger.
\end{equation}
Then
\begin{equation}\label{eq:comparison-after-cancellation}
 I(z)=e^{-zf_0}z^{-Q^\dagger}
 \left[\sum_{\alpha:Q_\alpha=Q^\dagger}K_\alpha+o(1)\right].
\end{equation}
If the displayed coefficient is nonzero, it is the true leading coefficient.
The rates in~\eqref{eq:piece-rate} are supplied, for example, by
Theorem~\ref{thm:quantitative-leading}.
\end{corollary}

\begin{proof}
After summing~\eqref{eq:piece-rate} over pieces with the same exponent, the
leading terms below $Q^\dagger$ vanish by~\eqref{eq:lower-level-cancellation},
and their remainders are $o(e^{-zf_0}z^{-Q^\dagger})$ by the strict exponent
gap.  Pieces with exponent larger than $Q^\dagger$, as well as the localization
error, are also of smaller order.  The pieces at exponent $Q^\dagger$ give the
stated coefficient.
\end{proof}

\begin{example}[Cancellation on a connected minimum set]\label{ex:cancellation}
Let $N=\R^2$,
\[
 f(x,y)=y^2+(-x)_+^2+(x-1)_+^2,
 \qquad
 a(x,y)=x^2-\frac13.
\]
Then $M=[0,1]\times\{0\}$.  The open segment is a normal-cell piece with one
quadratic normal variable and exponent $Q=1/2$.  Its leading coefficient from
Corollary~\ref{cor:coercive-leading-piece} is
\[
 \sqrt\pi\int_0^1\left(x^2-\frac13\right)\dd x=0.
\]
In fact, the whole product-form contribution of this piece vanishes, so its
remainder cannot mask the endpoint scale.

Let $I_0(z)$ and $I_1(z)$ denote the left- and right-endpoint pieces.  Their
normal cells are invariant half-planes and their phases are exactly
quadratic.  At the left endpoint the amplitude has no linear normal term, so
Theorem~\ref{thm:quantitative-leading} gives
\[
 I_0(z)=z^{-1}\left(-\frac{\pi}{6}+O(z^{-1})\right).
\]
At the right endpoint the amplitude has a linear normal term, and the same
theorem with $\eta=1/2$ gives
\[
 I_1(z)=z^{-1}\left(\frac{\pi}{3}+O(z^{-1/2})\right).
\]
The hypotheses of
Corollary~\ref{cor:comparison-after-cancellation} are therefore satisfied with
$Q^\dagger=1$.  It follows that
\[
 I(z)\sim\frac{\pi}{6}z^{-1}.
\]
The first nonzero term comes entirely from the lower-dimensional endpoint
pieces.
\end{example}

\begin{remark}[Positivity prevents cancellation]\label{rem:positivity}
If $a\ge0$, then every leading coefficient $K_\alpha$ is nonnegative.  If, on
some piece with $Q_\alpha=Q_*$, the leading profile $\cL_\alpha(p)$ is positive
on a set of positive base measure, then $K_*>0$ and
Theorem~\ref{thm:global-leading} gives the leading equivalence without any
higher-order calculation.  Under Corollary~\ref{cor:coercive-leading-piece}, this positivity follows,
for example, when the limiting amplitude $b_{0,\alpha}(p)>0$ and $\Afib_\alpha^\infty(p)$ has positive
Lebesgue measure on such a set.  In Fermi coordinates with $J\to1$, this reduces
to the familiar condition $a(p)>0$.  For signed or complex amplitudes,
Corollary~\ref{cor:comparison-after-cancellation} or the finite-expansion
results of Section~\ref{sec:expansions} are needed.
\end{remark}

\subsection{Joint rescaling at a nonintegrable frontier}

The normal cone of a lower stratum need not contain the directions tangent to
an incident higher stratum.  Consequently, a degenerating collar in the base
cannot in general be reassigned to the lower-stratum normal piece.  When the
fiberwise leading coefficient is not integrable over the collar, the frontier
coordinate must be rescaled together with the normal variables.  Analytically,
this joint rescaling is again handled by Theorem~\ref{thm:profile-leading};
no separate convergence theorem is needed.

\begin{figure}[!htb]
\centering
\includegraphics[width=\linewidth]{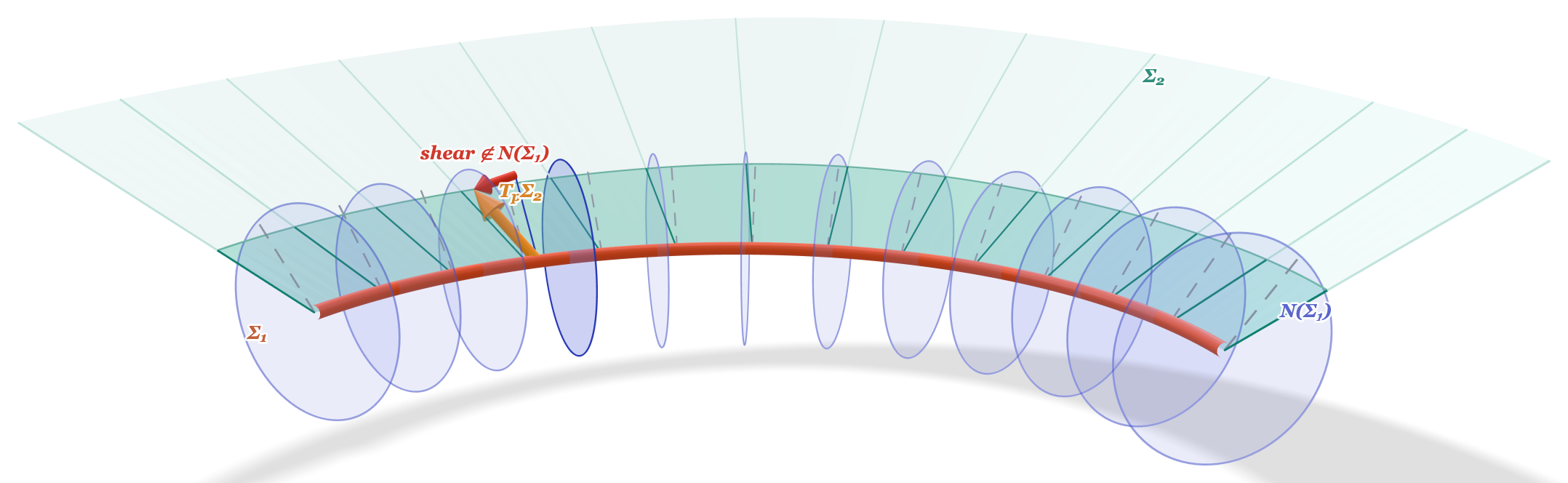}
\caption{Why a nonuniform collar cannot be reassigned to the lower-stratum
normal piece.  The normal cone of the lower stratum $\Sigma_1$ is the stack
of disk fibers orthogonal to $\Sigma_1$ (blue).  Since
$\Sigma_1\subset\partial\Sigma_2$, the tangent plane $T_p\Sigma_2$ of the
incident sheet (orange) contains the along-$\Sigma_1$ direction, so it
carries a shear component (red) that is not normal to $\Sigma_1$: the normal
cone need not contain the tangent directions of $\Sigma_2$.  When the shear
varies along $\Sigma_1$, no reparametrization removes it: the collar
directions must be retained as part of a joint model, as in
Remark~\ref{rem:joint-collar-construction}.}
\label{fig:collar}
\end{figure}

\begin{remark}[Joint rescaling from a stratum collar]
\label{rem:joint-collar-construction}
Suppose a neighborhood of a frontier set $S\subset\overline P$ inside a base
piece $P$ is parameterized by a collar $p=\Upsilon(y,u)$, where $y\in S$ and
$u\in\mathcal U\subset\R^r$, and equip $S$ with a finite measure $\nu$.
Combining the collar and transverse variables into
$w=(u,s)$ rewrites the contribution as
\begin{equation}\label{eq:joint-collar-representation}
 \int_S e^{-zf_0}\int_{\Afib_{\mathrm{joint}}(y)}
 b(y,w)e^{-zc(y)\phi(y,w)}\dd w\dd\nu(y).
\end{equation}
Here $b$ contains both the collar Jacobian of $\Upsilon$ and the transverse
coordinate Jacobian, while $\Afib_{\mathrm{joint}}(y)$ records simultaneously
the collar condition on $u$ and the normal-cell condition on $s$.  Apply
Theorem~\ref{thm:profile-leading} with parameter $y$, admissible cell
$\Afib_{\mathrm{joint}}(y)$,
phase $c(y)\phi(y,w)$, scaling coefficient $c(y)$, $\mathcal K(v)=e^{-v}$,
$\mu=0$, and $\mathcal N_{\mathcal K}(z)=1$.  Global coercivity of the joint model is unnecessary;
the required information is the weighted convergence and tightness of the
joint scaled cells.  The collar directions remain tangent to the higher
stratum geometrically, but they belong to the scaled fiber for this asymptotic
calculation.
\end{remark}

\begin{example}[A linear cone suppresses the soft direction]
\label{ex:cone-changes-anisotropy}
Let $0<\eps<1$ and
\[
 C:=\{(s_1,s_2)\in\R^2:s_2\ge|s_1|\},
 \qquad
 f(s_1,s_2)=s_1^4+s_2^2,
\]
and integrate over $C\cap B_\eps$.

First apply Corollary~\ref{cor:coercive-leading-piece} with the unconstrained phase
weights $(1/4,1/2)$.  The rescaled cells are
\[
 D_z^{(1/4,1/2)}C
 =\{(t_1,t_2):t_2\ge z^{1/4}|t_1|\},
\]
which converge almost everywhere to a Lebesgue-null ray.
Corollary~\ref{cor:coercive-leading-piece} therefore gives a zero coefficient at order $z^{-3/4}$, and hence only
$o(z^{-3/4})$ at that nominal scale.

To identify the true scale, apply Theorem~\ref{thm:profile-leading} with
$P$ a singleton,
$\Afib=C\cap B_\eps$, $\Afib^\infty=C$, $c\equiv1$, equal weights
$\bm\beta=(1/2,1/2)$, $G(t)=t_2^2$, and the exponential profile.  The rescaled
phase is
\[
 X_z(t)=t_2^2+z^{-1}t_1^4,
\]
so it converges locally uniformly to $G$ and satisfies $X_z\ge G$.  The
rescaled truncations are $C\cap B_{\eps\sqrt z}$ and increase to $C$.  Since
$|t_1|\le t_2$ on $C$, the weight $e^{-\gamma t_2^2}$ is integrable on $C$
for every $\gamma>0$; hence the cell convergence and tightness assumptions
\eqref{eq:profile-cell-convergence}--\eqref{eq:profile-cell-tightness} hold.
Theorem~\ref{thm:profile-leading} gives
\[
 \int_{C\cap B_\eps}e^{-z(s_1^4+s_2^2)}\dd s
 \sim z^{-1}\int_C e^{-t_2^2}\dd t
 =z^{-1},
\]
because $\int_C e^{-t_2^2}\dd t=1$.
The geometric constraint makes the nominally soft quartic direction no wider
than the quadratic direction and changes the exponent from $3/4$ to $1$.
\end{example}

\begin{proposition}[Critical logarithm for the model collar]
\label{prop:critical-collar-logarithm}
For $k>1$ and $\eps>0$, define
\[
 I^{\mathrm{col}}_{2,k}(z):=\int_0^\eps\int_{-\eps}^{\eps}
 e^{-z(u^2s^2+|s|^{2k})}\dd s\dd u.
\]
Then
\[
 I^{\mathrm{col}}_{2,k}(z)
 \sim \sqrt\pi\frac{k-1}{2k}\,z^{-1/2}\log z.
\]
\end{proposition}

\begin{proof}
Integrating first in $u$ and then using symmetry in $s$ gives
\[
 I^{\mathrm{col}}_{2,k}(z)
 =\frac{\sqrt\pi}{\sqrt z}\int_0^\eps
 e^{-zs^{2k}}\frac{\erf(\eps\sqrt z\,s)}{s}\dd s,
\]
where
\[
 \erf(x):=\frac{2}{\sqrt\pi}\int_0^x e^{-r^2}\dd r,
 \qquad
 \erfc(x):=1-\erf(x).
\]
Set
\[
 t=z^{1/(2k)}s,
 \qquad
 A_z:=\eps z^{(k-1)/(2k)},
 \qquad
 B_z:=\eps z^{1/(2k)}.
\]
Then
\[
 I^{\mathrm{col}}_{2,k}(z)=\frac{\sqrt\pi}{\sqrt z}\,L_z^{\mathrm{crit}},
 \qquad
 L_z^{\mathrm{crit}}:=\int_0^{B_z}
 e^{-t^{2k}}\frac{\erf(A_z t)}{t}\dd t.
\]
Since $k>1$, both $A_z$ and $B_z$ tend to infinity.  For all sufficiently
large $z$, split the integral at $A_z^{-1}$ and at $1$.  The first part is
uniformly bounded because $\erf(r)\le C r$ for $r\ge0$:
\[
 \int_0^{A_z^{-1}}
 e^{-t^{2k}}\frac{\erf(A_z t)}{t}\dd t
 \le C A_z\int_0^{A_z^{-1}}\dd t=O(1).
\]
On $A_z^{-1}<t<1$, the error made by replacing the error function by one is
also uniformly bounded:
\[
 \int_{A_z^{-1}}^1
 \frac{\erfc(A_z t)}{t}\dd t
 =\int_1^{A_z}\frac{\erfc(r)}{r}\dd r
 \le\int_1^\infty\frac{\erfc(r)}{r}\dd r<\infty.
\]
The exponential factor may be replaced by one on the same interval at bounded
cost, since $1-e^{-x}\le x$:
\[
 \int_{A_z^{-1}}^1
 \frac{1-e^{-t^{2k}}}{t}\dd t
 \le\int_0^1t^{2k-1}\dd t=\frac1{2k}.
\]
Finally,
\[
 \int_1^{B_z}e^{-t^{2k}}
 \frac{\erf(A_z t)}{t}\dd t
 \le\int_1^\infty\frac{e^{-t^{2k}}}{t}\dd t=O(1).
\]
Combining the four bounds yields
\[
 L_z^{\mathrm{crit}}=\int_{A_z^{-1}}^1\frac{\dd t}{t}+O(1)
 =\log A_z+O(1)
 =\frac{k-1}{2k}\log z+O(1),
\]
where the constant $\log\eps$ is absorbed into $O(1)$.  Multiplication by
$\sqrt\pi z^{-1/2}$ proves the result.
\end{proof}

\begin{example}[A frontier trichotomy in a collar]
\label{ex:frontier-trichotomy}
Let $k>1$, $m>0$, and consider
\begin{equation}\label{eq:frontier-model-integral}
 I^{\mathrm{col}}_{m,k}(z):=\int_0^\eps\int_{-\eps}^{\eps}
 e^{-z(u^m s^2+|s|^{2k})}\dd s\dd u.
\end{equation}
Here $u$ approaches a lower stratum inside the closure of a higher stratum,
whereas $s$ is normal to that higher stratum.  For each fixed $u>0$, the
normal model is quadratic with $c(u)=u^m$ and $Q=1/2$.

If $m<2$, then $c(u)^{-Q}=u^{-m/2}$ is integrable.  Since the cells are
independent of $u$, Theorem~\ref{thm:integrated-pointwise-cells} applies
directly to the fiberwise quadratic model and gives
\begin{equation}\label{eq:frontier-subcritical}
 I^{\mathrm{col}}_{m,k}(z)\sim
 \sqrt\pi\frac{\eps^{1-m/2}}{1-m/2}\,z^{-1/2}.
\end{equation}

If $m>2$, the fiberwise coefficient is not integrable.  Use the joint weights
\begin{equation}\label{eq:frontier-joint-weights}
 \beta_u=\frac{k-1}{km},
 \qquad
 \beta_s=\frac1{2k},
 \qquad
 Q=\frac{k-1}{km}+\frac1{2k}.
\end{equation}
Use on the ambient plane the nonnegative homogeneous extension
\[
 \widetilde G(u,s):=|u|^m s^2+|s|^{2k}.
\]
It agrees with the phase on the admissible half-plane $u\ge0$ and is weighted
homogeneous for~\eqref{eq:frontier-joint-weights}.  Its coefficient on the
limiting cell is
\begin{equation}\label{eq:frontier-model-coefficient}
 \int_0^\infty\int_\R e^{-(u^m s^2+|s|^{2k})}\dd s\dd u
 =\frac1k\Gamma\!\left(1+\frac1m\right)
  \Gamma\!\left(\frac{m-2}{2km}\right)<\infty.
\end{equation}
The rescaled rectangles increase to $\R_+\times\R$.  The same calculation
with $\gamma\widetilde G$ in place of $\widetilde G$ shows that
$e^{-\gamma\widetilde G}$ is integrable on the limiting half-plane for every
$\gamma>0$.  Monotone exhaustion therefore verifies
\eqref{eq:profile-cell-convergence}--\eqref{eq:profile-cell-tightness}.
Apply Theorem~\ref{thm:profile-leading} with $P$ a singleton,
\[
 \Afib=(0,\eps)\times(-\eps,\eps),
 \qquad
 \Afib^\infty=\R_+\times\R,
 \qquad
 c\equiv1,
\]
with phase $\phi(u,s)=u^m s^2+|s|^{2k}$, model $G=\widetilde G$,
$b\equiv1$, and the exponential profile.  The rescaled phase equals $G$
exactly, and the rescaled rectangles increase to $\Afib^\infty$.  Consequently,
\begin{equation}\label{eq:frontier-supercritical}
 I^{\mathrm{col}}_{m,k}(z)\sim
 \frac1k\Gamma\!\left(1+\frac1m\right)
 \Gamma\!\left(\frac{m-2}{2km}\right)z^{-Q}.
\end{equation}
Since $Q<1/2$, the frontier collar dominates the uniformly quadratic part of
the higher stratum.

At $m=2$, Proposition~\ref{prop:critical-collar-logarithm} gives
\begin{equation}\label{eq:frontier-critical}
 I^{\mathrm{col}}_{2,k}(z)\sim
 \sqrt\pi\frac{k-1}{2k}\,z^{-1/2}\log z.
\end{equation}
Thus the same model family displays three regimes: integrable fiberwise
asymptotics, a critical logarithm, and a joint rescaling with cellwise tightness.
\end{example}

\subsection{A leading homogeneous amplitude}

Higher-order Laplace expansions incorporate both phase and amplitude
jets~\cite{Kirwin,FukudaKagayaUeda}; boundary and rank-deficient analogues are
treated by Six--Rouchon~\cite{SixRouchon}.  The homogeneous normal-cell statement
below isolates the leading vanishing amplitude and avoids a multiplicative
formula of the form ``$a=A(1+o(1))$'', which is not meaningful at zeros of $A$.

\begin{corollary}[Vanishing amplitude]\label{cor:vanishing-amplitude}
Retain the phase assumptions of
Corollary~\ref{cor:coercive-leading-piece}, and put
\[
 b(p,s):=a(F(p,s))J(p,s).
\]
Assume in addition that the effective amplitude is uniformly bounded on
the truncated fibers:
\begin{equation}\label{eq:vanishing-amplitude-bounded}
 \sup_{p\in P}\sup_{s\in\Afib(p)\cap B_\eps}|b(p,s)|<\infty.
\end{equation}
Let $\kappa\ge0$ and let the measurable function
$B_\kappa:P\times\R^d\to\C$ satisfy
\begin{equation}\label{eq:B-kappa-hom}
  B_\kappa(p,D_u^{\bm\beta}t)=u^\kappa B_\kappa(p,t),
\end{equation}
with $B_\kappa$ uniformly bounded on
$\{(p,t):\rho_{\bm\beta}(t)=1\}$.  We fix the null-set representative
$B_\kappa(p,0)=0$.  Assume
\begin{equation}\label{eq:B-kappa-remainder}
 \lim_{r\downarrow0}\sup_{p\in P}
 \sup_{\substack{s\in\Afib(p)\\0<\rho_{\bm\beta}(s)\le r}}
 \frac{|b(p,s)-B_\kappa(p,s)|}
      {\rho_{\bm\beta}(s)^\kappa}=0,
\end{equation}
and assume that, for every $0<\delta<1$,
\begin{equation}\label{eq:weighted-cell-limit-kappa}
\begin{split}
 \sup_{p\in P}\int_{\R^d}
 \Big|&\ind_{D_{zc(p)}^{\bm\beta}(\Afib(p)\cap B_\eps)}(t)
      -\ind_{\Afib^\infty(p)}(t)\Big|\\
 &\hspace{2em}\times
 \bigl(1+\rho_{\bm\beta}(t)^\kappa\bigr)
 e^{-(1-\delta)G(p,t)}\dd t\longrightarrow0.
\end{split}
\end{equation}
Then
\begin{equation}\label{eq:vanishing-amplitude-result}
 \sup_{p\in P}\left|
 e^{zf_0}(zc(p))^{Q+\kappa}\cI(z,p)
 -\int_{\Afib^\infty(p)}B_\kappa(p,t)e^{-G(p,t)}\dd t
 \right|\longrightarrow0.
\end{equation}
\end{corollary}

\begin{proof}
Write $\rho=\rho_{\bm\beta}$ and $q_p=(zc(p))^{-1}$, and set
\[
 E_z(p):=D_{1/q_p}^{\bm\beta}(\Afib(p)\cap B_\eps).
\]
The zero section has zero Lebesgue measure in every positive-dimensional
fiber.  We therefore take $b(p,0)=0$, consistently with the chosen
representative $B_\kappa(p,0)=0$, without changing any fiber integral.  After
$s=D_{q_p}^{\bm\beta}t$, the normalized fiber integral is
\begin{equation}\label{eq:vanishing-amplitude-scaled}
 e^{zf_0}(zc(p))^{Q+\kappa}\cI(z,p)
 =\int_{E_z(p)}q_p^{-\kappa}b(p,D_{q_p}^{\bm\beta}t)
 e^{-G(p,t)(1+\mathcal R(p,D_{q_p}^{\bm\beta}t))}\dd t.
\end{equation}
Homogeneity gives
$q_p^{-\kappa}B_\kappa(p,D_{q_p}^{\bm\beta}t)=B_\kappa(p,t)$ and, by boundedness on
the anisotropic unit sphere,
\begin{equation}\label{eq:B-kappa-growth}
 |B_\kappa(p,t)|\le C(1+\rho(t)^\kappa).
\end{equation}

Let
\[
 \omega(r):=\sup_{p\in P}
 \sup_{\substack{s\in\Afib(p)\\0<\rho(s)\le r}}
 \frac{|b(p,s)-B_\kappa(p,s)|}{\rho(s)^\kappa},
\]
so that $\omega(r)\to0$.  Choose $r_0>0$ so small that $\omega$ is bounded on
$(0,r_0]$.  On the part of the scaled domain where
$q_p\rho(t)\le r_0$,
\begin{equation}\label{eq:vanishing-amplitude-near-bound}
 q_p^{-\kappa}|b(p,D_{q_p}^{\bm\beta}t)|
 \le |B_\kappa(p,t)|+\omega(q_p\rho(t))\rho(t)^\kappa
 \le C(1+\rho(t)^\kappa).
\end{equation}
Choose $\delta$ with $\delta_0<\delta<1$.  The phase bound
$|\mathcal R|\le\delta_0$ therefore makes the corresponding integrand in
\eqref{eq:vanishing-amplitude-scaled} dominated by
$C(1+\rho^\kappa)e^{-(1-\delta)G}$.

On the complementary part $q_p\rho(t)>r_0$, the uniform bound
\eqref{eq:vanishing-amplitude-bounded} and Lemma~\ref{lem:moments} give
\begin{align}
 &\sup_{p\in P}q_p^{-\kappa}
 \int_{E_z(p)\cap\{q_p\rho>r_0\}}
 |b(p,D_{q_p}^{\bm\beta}t)|
 e^{-G(p,t)(1+\mathcal R(p,D_{q_p}^{\bm\beta}t))}\dd t\notag\\
 &\qquad\le C\sup_{p\in P}q_p^{-\kappa}
 \int_{\{\rho>r_0/q_p\}}
 e^{-(1-\delta)G(p,t)}\dd t=o(1).
 \label{eq:vanishing-amplitude-far-tail}
\end{align}
The estimate is uniform in $p$; in fact the tail integral is smaller than
every power of $q_p$.  Equations~\eqref{eq:vanishing-amplitude-near-bound} and
\eqref{eq:vanishing-amplitude-far-tail} give uniform integrability with the
polynomial weight $1+\rho^\kappa$.

On each fixed anisotropic ball,~\eqref{eq:B-kappa-remainder} and
\eqref{eq:R-uniform} imply local $L^1$ convergence, uniformly in $p$, of the
integrand in~\eqref{eq:vanishing-amplitude-scaled} to
$B_\kappa(p,t)e^{-G(p,t)}$; the value at $t=0$ is immaterial.  The preceding
near--far split and Lemma~\ref{lem:moments} extend this convergence to the
whole scaled domain.  Finally,~\eqref{eq:B-kappa-growth} and
\eqref{eq:weighted-cell-limit-kappa} replace $E_z(p)$ by
$\Afib^\infty(p)$ uniformly in $p$.  This proves
\eqref{eq:vanishing-amplitude-result}.
\end{proof}

\section{Finite asymptotic expansions}\label{sec:expansions}

Kirwin gives higher expansions for isolated real Laplace minima from phase and
amplitude asymptotics~\cite{Kirwin}.  Leading-order results require only the
first rescaled limits of the phase, amplitude, and admissible cell.  At higher
orders these effects interact, and motion of the rescaled cell can contribute
through its boundary.  The general formulation therefore permits a representation, depending on a
small scale $q\downarrow0$, of the scaled cell on a fixed reference cell.  The
cell motion and the Jacobian of this representation enter the
same complete scaled density as the phase and amplitude.  An ambient indicator form
is also given for cases in which no straightening is used.

Fix one positive-codimensional normal-cell piece and suppress its index.  Fix
an anisotropy vector $\bm\beta\in(0,\infty)^d$, put
$Q:=Q_{\bm\beta}$.  Let $c:P\to(0,\infty)$ satisfy
\eqref{eq:c-bounds}, and let
$G:P\times\R^d\to[0,\infty)$ be a measurable degree-one homogeneous
model.  Define
\begin{equation}\label{eq:b-phi}
 b(p,s):=a(F(p,s))J(p,s),
 \qquad
 \phi(p,s):=\frac{f(F(p,s))-f_0}{c(p)},
\end{equation}
and, for $q>0$,
\begin{equation}\label{eq:finite-scaled-cell}
 \Afib_q(p):=D_{1/q}^{\bm\beta}\bigl(\Afib(p)\cap B_\eps\bigr).
\end{equation}
When $q=(zc(p))^{-1}$, this is the cell denoted by $E_z(p)$ in the
leading-order section.  The notation $\Afib_q$ is used here because the
expansion is organized in the independent small parameter $q$.

\begin{definition}[Scaled cell representation]
\label{def:scaled-cell-representation}
A \emph{scaled cell representation} of the family $\Afib_q(p)$ consists of a
fixed measurable family of reference cells $C(p)\subset\R^d$, whose
graph in $P\times\R^d$ is measurable, together with families
$\Theta_q(p,t)$ and $j_q(p,t)\ge0$ that are jointly measurable in $(q,p,t)$
on that graph and satisfy
\begin{equation}\label{eq:finite-cell-change-of-variables}
 \int_{\Afib_q(p)}\psi(u)\dd u
 =\int_{C(p)}\psi(\Theta_q(p,t))j_q(p,t)\dd t
\end{equation}
for every nonnegative measurable $\psi$.  Equivalently,
$(\Theta_q(p,\cdot))_\#(j_q(p,\cdot)\dd t)$ is Lebesgue measure restricted to
$\Afib_q(p)$, where $T_\#\mu$ denotes the pushforward of a measure $\mu$ by a
map $T$.  No injectivity or smoothness is required.  On the set where
$j_q>0$, the map may be taken to have values in $\Afib_q(p)$; its values elsewhere
are immaterial.  A smooth almost-everywhere bijection with $j_q$ equal to its
Jacobian is an important special case.

Every jointly measurable family $\Afib_q(p)$ has the tautological representation
\begin{equation}\label{eq:tautological-cell-representation}
 C(p)=\R^d,\qquad \Theta_q(p,t)=t,\qquad
 j_q(p,t)=\ind_{\Afib_q(p)}(t).
\end{equation}
Thus Theorem~\ref{thm:finite-expansion} imposes no invariance or fixed-shape
condition on the scaled cells.  A nontrivial representation is useful when it
turns cell motion into ordinary phase, amplitude, and Jacobian corrections.
\end{definition}

For a scaled cell representation, define on $\{j_q>0\}$
\begin{equation}\label{eq:transported-amplitude-phase}
 \widetilde b_q(p,t):=b\bigl(p,D_q^{\bm\beta}\Theta_q(p,t)\bigr)j_q(p,t),
 \qquad
 \Phi_q(p,t):=q^{-1}
 \phi\bigl(p,D_q^{\bm\beta}\Theta_q(p,t)\bigr),
\end{equation}
and set $\widetilde b_q=0$ and $\Phi_q=G$ on $\{j_q=0\}$.  The complete pulled-back
scaled density is
\begin{equation}\label{eq:straightened-density}
 \widetilde{\cD}_q(p,t):=\widetilde b_q(p,t)e^{G(p,t)-\Phi_q(p,t)}.
\end{equation}

\begin{assumption}[Weighted expansion of the transported density]
\label{ass:scaled-expansion}
Choose a scaled cell representation as in
Definition~\ref{def:scaled-cell-representation}.  Let $R>0$.  Assume that there
is a finite set $\Lambda_R\subset[0,R]$ and measurable profiles
$\cD_\lambda$ on the reference-cell graph such that
\begin{equation}\label{eq:D-moments}
 \sup_{p\in P}\int_{C(p)}|\cD_\lambda(p,t)|e^{-G(p,t)}\dd t<\infty
\end{equation}
for every $\lambda$, and
\begin{equation}\label{eq:scaled-L1-expansion}
 \sup_{p\in P}q^{-R}\int_{C(p)}
 \left|\widetilde{\cD}_q(p,t)-
 \sum_{\lambda\in\Lambda_R}q^\lambda \cD_\lambda(p,t)\right|
 e^{-G(p,t)}\dd t\longrightarrow0
 \qquad(q\downarrow0).
\end{equation}
\end{assumption}

\begin{theorem}[Finite expansion on one normal-cell piece]
\label{thm:finite-expansion}
Under the setup of~\eqref{eq:b-phi}--\eqref{eq:straightened-density}
and Assumption~\ref{ass:scaled-expansion}, set
\begin{equation}\label{eq:L-lambda}
 \cL_\lambda(p):=\int_{C(p)}\cD_\lambda(p,t)e^{-G(p,t)}\dd t.
\end{equation}
Then, uniformly in $p\in P$,
\begin{equation}\label{eq:finite-expansion-result}
 \cI(z,p)=e^{-zf_0}
 \sum_{\lambda\in\Lambda_R}
 z^{-(Q+\lambda)}c(p)^{-(Q+\lambda)}\cL_\lambda(p)
 +o\left(e^{-zf_0}z^{-(Q+R)}\right).
\end{equation}
If $P$ has finite volume, the same expansion holds after integration over
$P$, with $c(p)^{-(Q+\lambda)}\cL_\lambda(p)$ integrated against
$\dvol_P(p)$.
\end{theorem}

\begin{proof}
For $q_p=(zc(p))^{-1}$, the scaling
$s=D_{q_p}^{\bm\beta}u$ followed by the pushforward identity
\eqref{eq:finite-cell-change-of-variables} gives the exact formula
\begin{equation}\label{eq:exact-expansion-identity}
 e^{zf_0}q_p^{-Q}\cI(z,p)
 =\int_{C(p)}\widetilde{\cD}_{q_p}(p,t)e^{-G(p,t)}\dd t.
\end{equation}
Assumption~\ref{ass:scaled-expansion} yields
\[
 \int_{C(p)}\widetilde{\cD}_{q_p}e^{-G}\dd t
 =\sum_{\lambda\in\Lambda_R}q_p^\lambda \cL_\lambda(p)
 +o(q_p^R)
\]
uniformly in $p$.  Since $q_p\to0$ uniformly by~\eqref{eq:c-bounds},
multiplication by $q_p^Q$ proves~\eqref{eq:finite-expansion-result}.  The
graph measurability and Tonelli's theorem make each $\cL_\lambda$ measurable.
The integrated statement follows from~\eqref{eq:D-moments}, the bounds on
$c$, and the finite volume of $P$.
\end{proof}

\begin{corollary}[Ambient indicator form]
\label{cor:ambient-indicator-expansion}
Define the complete scaled density on $\R^d$ by
\begin{equation}\label{eq:Dq}
 \cD_q(p,t):=
 \begin{cases}
 b(p,D_q^{\bm\beta}t)
 \exp\left\{G(p,t)-q^{-1}\phi(p,D_q^{\bm\beta}t)\right\},
   &t\in \Afib_q(p),\\
 0,&t\notin \Afib_q(p).
 \end{cases}
\end{equation}
Suppose that a finite set $\Lambda_R\subset[0,R]$ and measurable functions
$\cD_\lambda:P\times\R^d\to\C$ satisfy
\begin{equation}\label{eq:D-moments-ambient}
 \sup_{p\in P}\int_{\R^d}|\cD_\lambda(p,t)|e^{-G(p,t)}\dd t<\infty
\end{equation}
for every $\lambda$, and
\begin{equation}\label{eq:scaled-L1-expansion-ambient}
 \sup_{p\in P}q^{-R}\int_{\R^d}
 \left|\cD_q(p,t)-
 \sum_{\lambda\in\Lambda_R}q^\lambda \cD_\lambda(p,t)\right|
 e^{-G(p,t)}\dd t\longrightarrow0.
\end{equation}
Then the conclusion of Theorem~\ref{thm:finite-expansion} holds with
\begin{equation}\label{eq:L-lambda-density}
 \cL_\lambda(p)=\int_{\R^d}\cD_\lambda(p,t)e^{-G(p,t)}\dd t.
\end{equation}
\end{corollary}

\begin{proof}
Use the tautological representation~\eqref{eq:tautological-cell-representation}.
Then $\widetilde{\cD}_q=\cD_q$, so the assertion is exactly
Theorem~\ref{thm:finite-expansion}.
\end{proof}

\begin{remark}[Why a nontrivial representation can matter]\label{rem:nontrivial-cell-representation}
The tautological representation covers arbitrary moving cells, but their
indicators need not possess an $L^1$ expansion.  For example, fix
$\vartheta_{\mathrm{cell}}>0$, let $h$ be smooth and compactly supported, let
$C=\{(y,r):r\ge0\}$, and put
$\Afib_q=\{(y,r):r\ge q^{\vartheta_{\mathrm{cell}}}h(y)\}$.  The translation
\[
 \Theta_q(y,r)=(y,r+q^{\vartheta_{\mathrm{cell}}}h(y))
\]
straightens $\Afib_q$ to $C$, and its effect enters the composed phase, amplitude,
and Jacobian in~\eqref{eq:straightened-density}.  If the two domains are
compared without straightening, then for every compactly supported smooth
$\psi$,
\[
 q^{-\vartheta_{\mathrm{cell}}}\left(\int_{\Afib_q}\psi-\int_C\psi\right)
 \longrightarrow-\int h(y)\psi(y,0)\dd y,
\]
so the first boundary correction is distributional rather than an ordinary
$L^1$ density.  A distributional formulation is possible, but is not needed
when a suitable reference-cell parametrization is available.
\end{remark}

\begin{corollary}[Global finite expansion]\label{cor:global-finite-expansion}
Assume the global hypotheses of Theorem~\ref{thm:global-leading}, including the
piecewise leading expansions~\eqref{eq:piece-integrated-leading}, and fix
$T\ge Q_*$.  For every piece with $Q_\alpha\le T$, suppose that there are a
finite set
\[
 \Lambda_{\alpha,T}\subset[0,T-Q_\alpha],
 \qquad 0\in\Lambda_{\alpha,T},
\]
and integrated coefficients $\widehat K_{\alpha,\lambda}\in\C$, with
$\widehat K_{\alpha,0}=K_\alpha$, such that
\begin{equation}\label{eq:piece-integrated-finite-expansion}
 I_\alpha(z)=e^{-zf_0}
 \sum_{\lambda\in\Lambda_{\alpha,T}}
 z^{-(Q_\alpha+\lambda)}\widehat K_{\alpha,\lambda}
 +o\left(e^{-zf_0}z^{-T}\right).
\end{equation}
When $Q_\alpha=T$, this assumption is just the leading expansion with
$\Lambda_{\alpha,T}=\{0\}$.  For a zero-codimensional piece the contribution is
exact, with $Q_\alpha=0$ and
\[
 \widehat K_{\alpha,0}=K_\alpha
 =\int_{P_\alpha}a(p)\dvol_{P_\alpha}(p).
\]
Define
\[
 \mathcal S_T:=\left\{Q_\alpha+\lambda:
 \alpha\in\cA,\ Q_\alpha\le T,\
 \lambda\in\Lambda_{\alpha,T}\right\},
\]
and, for $\xi\in\mathcal S_T$,
\begin{equation}\label{eq:global-expansion-coeff}
 K^{\mathrm{glob}}_\xi:=
 \sum_{\substack{\alpha\in\cA,\ Q_\alpha\le T,\
                   \lambda\in\Lambda_{\alpha,T}\\
                   Q_\alpha+\lambda=\xi}}
 \widehat K_{\alpha,\lambda}.
\end{equation}
Then
\begin{equation}\label{eq:global-finite-expansion}
 I(z)=e^{-zf_0}\sum_{\xi\in\mathcal S_T}z^{-\xi}K^{\mathrm{glob}}_\xi
 +o\left(e^{-zf_0}z^{-T}\right).
\end{equation}
Consequently, if $\xi_0$ is the smallest exponent in $\mathcal S_T$ for which
$K^{\mathrm{glob}}_{\xi_0}\ne0$, then
\[
 I(z)\sim e^{-zf_0}z^{-\xi_0}K^{\mathrm{glob}}_{\xi_0}.
\]
\end{corollary}

\begin{proof}
Sum~\eqref{eq:piece-integrated-finite-expansion} over the pieces with
$Q_\alpha\le T$ and group equal total exponents.  A piece with
$Q_\alpha>T$ is $o(e^{-zf_0}z^{-T})$ by
\eqref{eq:piece-integrated-leading}.  The localization error in
Lemma~\ref{lem:localization} is also $o(e^{-zf_0}z^{-T})$.  Since the family of
pieces is finite, the remainders sum to the stated remainder.
\end{proof}

\begin{remark}[Coefficients from local finite expansions]
\label{rem:integrated-finite-coefficients}
If Theorem~\ref{thm:finite-expansion} applies to a positive-codimensional piece
with baseline exponent $Q_\alpha$ and $R=T-Q_\alpha$, then one may take
\begin{equation}\label{eq:integrated-finite-coefficients}
 \widehat K_{\alpha,\lambda}
 :=\int_{P_\alpha}
 c_\alpha(p)^{-(Q_\alpha+\lambda)}
 \cL_{\alpha,\lambda}(p)\dvol_{P_\alpha}(p).
\end{equation}
A locally vanishing amplitude may shift the sharp baseline.  In that case the
shift is absorbed into $Q_\alpha$ and the local orders are reindexed before
using~\eqref{eq:integrated-finite-coefficients}.  A missing zero-order profile
may be adjoined with coefficient zero.
\end{remark}

\subsection{Verifiable criteria after cell straightening}

A practical criterion for the weighted $L^1$ expansion is obtained from
transported phase and amplitude jets.  It applies both to invariant cells and
to moving cells after a suitable straightening.  Nonquadratic local models occur in isolated Laplace expansions, while
geometrically degenerate models also arise in sub-Riemannian heat-kernel
asymptotics~\cite{Kirwin,NeelSacchelli}.

Retain the notation $\widetilde b_q$ and $\Phi_q$ from
\eqref{eq:transported-amplitude-phase}, so that
$\widetilde{\cD}_q=\widetilde b_q e^{G-\Phi_q}$.

\begin{proposition}[Weighted jets after cell straightening]
\label{prop:homogeneous-jets}
Assume that, uniformly in $p$,
\begin{equation}\label{eq:reference-cell-coercivity}
 g_-\rho_{\bm\beta}(t)\le G(p,t)\le g_+\rho_{\bm\beta}(t)
 \qquad(t\in C(p))
\end{equation}
for some $0<g_-\le g_+<\infty$.  Fix $R>0$.  Suppose there are finite sets
$\Lambda_b\subset[0,R]$ and $\Lambda_\phi\subset(0,R]$, measurable profiles
$b_\lambda$ and $\phi_\nu$ on the reference-cell graph, and an exponent
$L_{\mathrm{gr}}<\infty$ such that
\begin{equation}\label{eq:transported-profile-growth}
 |b_\lambda(p,t)|+|\phi_\nu(p,t)|
 \le C(1+\rho_{\bm\beta}(t))^{L_{\mathrm{gr}}}.
\end{equation}
For every fixed $A>0$, let $\ell_q:=A\log(1/q)$.  Assume that on
$C(p)\cap\{\rho_{\bm\beta}\le \ell_q\}$,
\begin{align}
 \widetilde b_q(p,t)&=\sum_{\lambda\in\Lambda_b}q^\lambda b_\lambda(p,t)
            +r_{b,q}(p,t),\label{eq:transported-b-jet}\\
 \Phi_q(p,t)&=G(p,t)+\sum_{\nu\in\Lambda_\phi}q^\nu
 \phi_\nu(p,t)+r_{\phi,q}(p,t),\label{eq:transported-phi-jet}
\end{align}
and
\begin{equation}\label{eq:transported-jet-remainders}
 \sup_{p\in P}\sup_{\substack{t\in C(p)\\
              \rho_{\bm\beta}(t)\le \ell_q}}
 \frac{|r_{b,q}(p,t)|+|r_{\phi,q}(p,t)|}
 {q^R(1+\rho_{\bm\beta}(t))^{L_{\mathrm{gr}}}}\longrightarrow0.
\end{equation}
Finally, assume that for all sufficiently small $q$,
\begin{equation}\label{eq:transported-global-bounds}
 |\widetilde b_q(p,t)|\le C(1+\rho_{\bm\beta}(t))^{L_{\mathrm{gr}}},
 \qquad
 \Phi_q(p,t)\ge\kappa_G G(p,t)
 \qquad(t\in C(p))
\end{equation}
for some $\kappa_G>0$.

Let $\cE_R$ be the additive semigroup generated by $\Lambda_\phi$, truncated
at $R$, with $0$ adjoined.  Define
\begin{equation}\label{eq:E-eta}
 E_0:=1,
 \qquad
 E_\eta(p,t):=
 \sum_{m\ge1}\frac{(-1)^m}{m!}
 \sum_{\substack{\nu_1,\ldots,\nu_m\in\Lambda_\phi\\
                  \nu_1+\cdots+\nu_m=\eta}}
 \prod_{r=1}^m\phi_{\nu_r}(p,t)
 \quad(0<\eta\le R).
\end{equation}
All sums in~\eqref{eq:E-eta} are finite.  Put
\begin{equation}\label{eq:D-xi-formula}
 \Lambda_R:=(\Lambda_b+\cE_R)\cap[0,R],
 \qquad
 \cD_\xi(p,t):=
 \sum_{\substack{\lambda\in\Lambda_b,\ \eta\in\cE_R\\
                  \lambda+\eta=\xi}}
 b_\lambda(p,t)E_\eta(p,t).
\end{equation}
Then Assumption~\ref{ass:scaled-expansion} holds with these profiles, and
Theorem~\ref{thm:finite-expansion} applies.
\end{proposition}

\begin{proof}
Put
\[
 V_q^{[R]}:=\sum_{\nu\in\Lambda_\phi}q^\nu\phi_\nu,
 \qquad
 \widetilde b_q^{[R]}:=\sum_{\lambda\in\Lambda_b}q^\lambda b_\lambda.
\]
On every logarithmic region $\rho_{\bm\beta}\le \ell_q$, one has
$V_q^{[R]}\to0$ uniformly: this is immediate if $\Lambda_\phi=\varnothing$, while
for $\Lambda_\phi\ne\varnothing$ it follows from polynomial growth and
$\min\Lambda_\phi>0$.  The remainder estimate gives
$r_{b,q}=o(q^R)(1+\rho_{\bm\beta})^{L_{\mathrm{gr}}}$ and
$r_{\phi,q}=o(q^R)(1+\rho_{\bm\beta})^{L_{\mathrm{gr}}}$ there.  Consequently,
\[
 \widetilde{\cD}_q=(\widetilde b_q^{[R]}+r_{b,q})e^{-(V_q^{[R]}+r_{\phi,q})}.
\]
If $\Lambda_\phi\ne\varnothing$, put
$\nu_0:=\min\Lambda_\phi$ and choose $m_*$ with
$(m_*+1)\nu_0>R$.  Taylor expansion of the exponential through order $m_*$,
followed by multiplication by $\widetilde b_q^{[R]}$, produces the profiles
\eqref{eq:E-eta}--\eqref{eq:D-xi-formula}.  If
$\Lambda_\phi=\varnothing$, this step is absent.  Since the generating set is
finite and contains only positive orders, its additive semigroup has only
finitely many elements in each bounded interval.  The discarded orders are
therefore separated from $R$ by a positive gap.  After division by $q^R$, the
discarded terms and both remainders converge to zero on the logarithmic region,
up to a fixed polynomial in $\rho_{\bm\beta}$.

By~\eqref{eq:reference-cell-coercivity}, multiplication by $e^{-G}$ makes that
polynomial integrable uniformly in $p$.  On the complementary region,
\eqref{eq:transported-global-bounds} gives
\[
 |\widetilde{\cD}_q(p,t)|e^{-G(p,t)}
 =|\widetilde b_q(p,t)|e^{-\Phi_q(p,t)}
 \le C(1+\rho_{\bm\beta}(t))^{L_{\mathrm{gr}}} e^{-\kappa_G G(p,t)}.
\]
Every formal profile times $e^{-G}$ has the same type of
polynomial-exponential bound.  Since $\ell_q=A\log(1/q)$ and
$G\ge g_-\rho_{\bm\beta}$ on $C(p)$, these tails decay faster than every
prescribed power of $q$ once $A$ is sufficiently large; in particular, they
are $o(q^R)$.  This proves~\eqref{eq:scaled-L1-expansion}.
\end{proof}

\begin{remark}[Graph straightening]\label{rem:moving-boundary-straightening}
In the graph situation of Remark~\ref{rem:nontrivial-cell-representation},
\[
 \Afib_q(p)=\{(y,r):r\ge h_q(p,y)\},
 \qquad C(p)=\{(y,r):r\ge0\},
\]
the map $\Theta_q(p,y,r)=(y,r+h_q(p,y))$ has Jacobian one.  A finite
$q$-expansion of $h_q$, together with Taylor expansions of the phase and
amplitude after composition with $\Theta_q$, gives the transported jets in
Proposition~\ref{prop:homogeneous-jets}.  More generally, a smooth family of
cell-straightening diffeomorphisms may be used whenever both $\Theta_q$ and its
Jacobian admit expansions with the polynomial control required there.  Such a
straightening need not exist for every moving cell.
\end{remark}

\begin{corollary}[Homogeneous jets on a locally stationary cell]
\label{cor:homogeneous-stable-cell}
Fix $R>0$.  Assume that the scaled cell is represented without moving points on the support
of the transported measure,
\begin{equation}\label{eq:stationary-cell-support}
 \Theta_q(p,t)=t\qquad\text{whenever }j_q(p,t)>0,
\end{equation}
and that the representation is locally stationary after scaling: for every
fixed $A>0$, for all sufficiently small $q$,
\begin{equation}\label{eq:locally-stationary-cell}
 j_q(p,t)=1
 \quad\text{when }t\in C(p),\ 
 \rho_{\bm\beta}(t)\le A\log(1/q).
\end{equation}
This holds, for example, with $C(p)$ equal to an anisotropically invariant
untruncated cell, $\Theta_q(p,t)=t$, and $j_q$ the indicator of the expanding
Euclidean truncation.

Assume that, uniformly for $p\in P$ and for $s\in C(p)$ with
$\rho_{\bm\beta}(s)$ sufficiently small, there are finite sets
$\Lambda_b\subset[0,R]$ and $\Lambda_\phi\subset(0,R]$ and measurable
homogeneous profiles satisfying
\begin{align}
 b(p,s)&=\sum_{\lambda\in\Lambda_b}b_\lambda(p,s)+r_b(p,s),
 &b_\lambda(p,D_u^{\bm\beta}s)&=u^\lambda b_\lambda(p,s),
 \label{eq:b-jet}\\
 \phi(p,s)&=G(p,s)+\sum_{\nu\in\Lambda_\phi}\phi_\nu(p,s)+r_\phi(p,s),
 &\phi_\nu(p,D_u^{\bm\beta}s)&=u^{1+\nu}\phi_\nu(p,s),
 \label{eq:phi-jet}
\end{align}
where the profiles are uniformly bounded on the anisotropic unit sphere and
\begin{align}
 |r_b(p,s)|&\le\omega_b(\rho_{\bm\beta}(s))
 \rho_{\bm\beta}(s)^R,
 &\omega_b(r)&\to0,\label{eq:b-jet-rem}\\
 |r_\phi(p,s)|&\le\omega_\phi(\rho_{\bm\beta}(s))
 \rho_{\bm\beta}(s)^{1+R},
 &\omega_\phi(r)&\to0.\label{eq:phi-jet-rem}
\end{align}
Assume also that $0\le j_q\le j_+<\infty$, that $b$ is uniformly bounded on the
truncated coordinate domains, and that, after shrinking $\eps$ if necessary,
\begin{equation}\label{eq:phase-lower-bound}
 \phi(p,s)\ge\kappa_G G(p,s)
\end{equation}
on the admissible fibers for some $\kappa_G>0$.  If
\eqref{eq:reference-cell-coercivity} holds, then the hypotheses of
Proposition~\ref{prop:homogeneous-jets} are satisfied.
\end{corollary}

\begin{proof}
On the logarithmic region,~\eqref{eq:locally-stationary-cell} gives
$\widetilde b_q(p,t)=b(p,D_q^{\bm\beta}t)$ and
$\Phi_q(p,t)=q^{-1}\phi(p,D_q^{\bm\beta}t)$.  Homogeneity yields
\begin{align*}
 \widetilde b_q(p,t)&=\sum_{\lambda\in\Lambda_b}q^\lambda b_\lambda(p,t)
 +r_b(p,D_q^{\bm\beta}t),\\
 \Phi_q(p,t)&=G(p,t)+\sum_{\nu\in\Lambda_\phi}q^\nu\phi_\nu(p,t)
 +q^{-1}r_\phi(p,D_q^{\bm\beta}t).
\end{align*}
Since $q\rho_{\bm\beta}(t)\to0$ uniformly on every logarithmic region, the
remainder bounds imply~\eqref{eq:transported-jet-remainders}.  Homogeneity and
unit-sphere boundedness give polynomial growth of the profiles.  On
$\{j_q>0\}$, condition~\eqref{eq:stationary-cell-support}, the boundedness of $b$ and $j_q$, homogeneity of $G$, and
\eqref{eq:phase-lower-bound} give
\[
 |\widetilde b_q(p,t)|\le C,
 \qquad
 \Phi_q(p,t)\ge\kappa_G G(p,t).
\]
On $\{j_q=0\}$ these bounds hold by the convention $\widetilde b_q=0$, $\Phi_q=G$.
Thus~\eqref{eq:transported-global-bounds} holds, and
Proposition~\ref{prop:homogeneous-jets} applies.
\end{proof}

\begin{remark}[Relation with Taylor expansions]\label{rem:Taylor}
For a multiindex $\alpha\in\mathbb{N}_0^d$, the monomial $s^\alpha$ has
anisotropic degree
\begin{equation}\label{eq:weighted-degree}
 |\alpha|_{\bm\beta}:=\sum_{i=1}^d\beta_i\alpha_i.
\end{equation}
Thus an ordinary Taylor expansion is converted into homogeneous normal jets by
grouping monomials with the same weighted degree.  A phase monomial of weighted
degree $1+\nu$ contributes to $\phi_\nu$.  For a moving cell, this grouping is
performed after composing with the chosen straightening map and including its
Jacobian in the transported amplitude $\widetilde b_q$.
\end{remark}

\begin{corollary}[Smooth finite-type criterion on a locally stationary cell]
\label{cor:smooth-finite-type}
Assume the locally stationary setup of
Corollary~\ref{cor:homogeneous-stable-cell}.  Fix $R>0$ and set
$\beta_{\min}:=\min_i\beta_i$.  Choose integers $m_b,m_\phi$ such that
\begin{equation}\label{eq:smooth-orders}
 m_b\beta_{\min}>R,
 \qquad
 m_\phi\beta_{\min}>1+R.
\end{equation}
Assume that $b$ and $\phi$ admit uniform Taylor expansions in the normal
variable,
\begin{align}
 b(p,s)&=\sum_{|\alpha|\le m_b}
 \frac{\partial_s^\alpha b(p,0)}{\alpha!}s^\alpha
 +o(|s|^{m_b}),\label{eq:b-uniform-taylor}\\
 \phi(p,s)&=\sum_{|\alpha|\le m_\phi}
 \frac{\partial_s^\alpha\phi(p,0)}{\alpha!}s^\alpha
 +o(|s|^{m_\phi}),\label{eq:phi-uniform-taylor}
\end{align}
uniformly in $p$, with uniformly bounded Taylor coefficients.  Suppose
\begin{equation}\label{eq:subprincipal-vanishing}
 \partial_s^\alpha\phi(p,0)=0
 \qquad\text{whenever }|\alpha|_{\bm\beta}<1,
\end{equation}
and define
\begin{equation}\label{eq:smooth-principal-phase}
 G(p,s):=\sum_{|\alpha|_{\bm\beta}=1}
 \frac{\partial_s^\alpha\phi(p,0)}{\alpha!}s^\alpha.
\end{equation}
Assume that, uniformly in $p$,
\[
 0<g_-\le G(p,t)\le g_+<\infty
 \qquad\text{for }t\in C(p),\ \rho_{\bm\beta}(t)=1.
\]
Define
\[
 \Lambda_b:=\bigl\{|\alpha|_{\bm\beta}:|\alpha|\le m_b,
 |\alpha|_{\bm\beta}\le R\bigr\},
 \qquad
 \Lambda_\phi:=\bigl\{|\alpha|_{\bm\beta}-1:|\alpha|\le m_\phi,
 1<|\alpha|_{\bm\beta}\le1+R\bigr\},
\]
and set
\begin{align}
 b_\lambda(p,s)&:=
 \sum_{|\alpha|_{\bm\beta}=\lambda}
 \frac{\partial_s^\alpha b(p,0)}{\alpha!}s^\alpha,
 \label{eq:smooth-b-lambda}\\
 \phi_\nu(p,s)&:=
 \sum_{|\alpha|_{\bm\beta}=1+\nu}
 \frac{\partial_s^\alpha\phi(p,0)}{\alpha!}s^\alpha.
 \label{eq:smooth-phi-nu}
\end{align}
Then, after shrinking $\eps$, the hypotheses of
Corollary~\ref{cor:homogeneous-stable-cell} hold with these profiles.
Hence Theorem~\ref{thm:finite-expansion} applies using ordinary smooth
Taylor data on the locally stationary reference cell.
\end{corollary}

\begin{proof}
Write $s=D_r^{\bm\beta}\theta$, where
$r=\rho_{\bm\beta}(s)$ and $|\theta|=1$.  Then
$|s^\alpha|\le r^{|\alpha|_{\bm\beta}}$ and
$|s|\le C r^{\beta_{\min}}$ for $0<r\le1$.  Every monomial of weighted degree
at most $R$ occurs in the Taylor polynomial of order $m_b$, and every phase
monomial of weighted degree at most $1+R$ occurs in the Taylor polynomial of
order $m_\phi$.  The omitted polynomial monomials have weighted degree strictly
above the relevant cutoff, while the Taylor remainders are respectively
$o(r^R)$ and $o(r^{1+R})$.  This proves
\eqref{eq:b-jet-rem} and~\eqref{eq:phi-jet-rem}.

By~\eqref{eq:subprincipal-vanishing}, the first nonzero weighted part of the
phase is~\eqref{eq:smooth-principal-phase}.  All remaining terms are
$O(r^{1+\delta})$ for some $\delta>0$, uniformly in $p$.  Since
$G(p,s)\ge g_-r$, shrinking $\eps$ gives
$\phi(p,s)\ge\frac12 G(p,s)$ on the admissible fibers.  The boundedness
requirements follow from the uniform Taylor coefficients.
Corollary~\ref{cor:homogeneous-stable-cell} now applies.
\end{proof}

\begin{remark}[How uniformity is checked]\label{rem:uniformity-check}
The uniform Taylor expansions above are automatic if $P$ has compact closure in
a frame chart and the relevant normal derivatives extend continuously to a
fixed neighborhood, with the highest derivatives uniformly continuous.  A
conservative choice is
\[
 m_b=\left\lfloor\frac{R}{\beta_{\min}}\right\rfloor+1,
 \qquad
 m_\phi=\left\lfloor\frac{1+R}{\beta_{\min}}\right\rfloor+1.
\]
If the graph of $p\mapsto C(p)\cap\{\rho_{\bm\beta}=1\}$ is compact and the
principal polynomial $G$ is continuous and strictly positive on it, the
coercivity constant is uniform by compactness.  For a moving cell, the same
checks are made after the cell has been straightened.
\end{remark}

\subsection{Fa\`a di Bruno calculus for weighted coefficients}\label{subsec:faa-di-bruno}

Formula~\eqref{eq:E-eta} is explicit but becomes cumbersome at high order.
Bell-polynomial formulas for classical Laplace coefficients, together with
related earlier recurrences, are reviewed by Nemes~\cite{Nemes}.  The Fa\`a di
Bruno formula packages the same combinatorics and yields the weighted recurrence
below.

\begin{proposition}[Bell-polynomial recurrence]\label{prop:bell-recurrence}
Let $E_\eta$ be the exponential coefficients in~\eqref{eq:E-eta}, and set
$E_\xi=0$ when $\xi$ does not belong to the additive semigroup generated by
$\Lambda_\phi$.  Then, for every $\eta\in\cE_R\setminus\{0\}$,
\begin{equation}\label{eq:generalized-bell-recurrence}
 \eta E_\eta
 =-\sum_{\substack{\nu\in\Lambda_\phi\\\nu\le\eta}}
   \nu\phi_\nu E_{\eta-\nu}.
\end{equation}
If the phase orders are commensurable, choose $\sigma_{\mathrm{ord}}>0$
such that $\Lambda_\phi\subset\sigma_{\mathrm{ord}}\mathbb N$, and write
$\psi_k:=\phi_{k\sigma_{\mathrm{ord}}}$, with $\psi_k=0$ when
$k\sigma_{\mathrm{ord}}$ is absent.  Then
\begin{equation}\label{eq:bell-polynomial-formula}
 E_{m\sigma_{\mathrm{ord}}}
 =\frac1{m!}B_m\bigl(-1!\psi_1,-2!\psi_2,\ldots,-m!\psi_m\bigr),
\end{equation}
for every $m$ with $m\sigma_{\mathrm{ord}}\le R$, where $B_m$ is the complete exponential
Bell polynomial.  Equivalently,
\begin{equation}\label{eq:commensurable-recurrence}
 E_0=1,
 \qquad
 mE_{m\sigma_{\mathrm{ord}}}=-\sum_{k=1}^m k\psi_kE_{(m-k)\sigma_{\mathrm{ord}}}.
\end{equation}
The first coefficients are
\begin{align*}
 E_{\sigma_{\mathrm{ord}}}&=-\psi_1,\\
 E_{2\sigma_{\mathrm{ord}}}&=\frac12\psi_1^2-\psi_2,\\
 E_{3\sigma_{\mathrm{ord}}}&=-\frac16\psi_1^3+\psi_1\psi_2-\psi_3.
\end{align*}
\end{proposition}

\begin{proof}
In the formal generalized power series
\[
 E(q):=\exp\left(-\sum_{\nu\in\Lambda_\phi}q^\nu\phi_\nu\right)
      =\sum_\eta q^\eta E_\eta,
\]
apply the Euler operator $q\frac{\dd}{\dd q}$.  The identity
\[
 qE'(q)=-\left(\sum_\nu\nu q^\nu\phi_\nu\right)E(q)
\]
and comparison of the coefficient of $q^\eta$ give
\eqref{eq:generalized-bell-recurrence}.  In the commensurable case, put
$x=q^{\sigma_{\mathrm{ord}}}$ and apply the standard generating identity
\[
 \exp\left(\sum_{k\ge1}\frac{x_k}{k!}x^k\right)
 =\sum_{m\ge0}\frac1{m!}B_m(x_1,\ldots,x_m)x^m
\]
with $x_k=-k!\psi_k$.
\end{proof}

\begin{remark}[Composite normal jets]\label{rem:faa-composition}
Fa\`a di Bruno is also useful before the weighted grouping is performed.
For a scalar function $h\in C^m(N)$ and a normal coordinate map
$F_p(s):=F(p,s)$ of class $C^m$ in $s$, choose a coordinate chart near $p$ and
use the same symbol $h$ for the coordinate representative of the outer
function.  Then
\begin{equation}\label{eq:multivariate-faa}
 D^m(h\circ F_p)(0)[v_1,\ldots,v_m]
 =\sum_{\pi\in\Pi_m}
 D^{|\pi|}h(F_p(0))
 \left[
 D^{|B|}F_p(0)[v_i:i\in B]:B\in\pi
 \right],
\end{equation}
where $\Pi_m$ is the set of partitions of $\{1,\ldots,m\}$.  Leibniz's rule
then gives the jet of $b=(a\circ F)J$.  In Euclidean Fermi coordinates the map
$F$ is affine in $s$, so all terms involving $D_s^k F$ with $k\ge2$ vanish.  On
a curved ambient manifold, provided the metric, normal frame, and Fermi map
have the corresponding $C^m$ regularity,~\eqref{eq:multivariate-faa} cleanly
organizes the higher derivatives of the exponential map and the Fermi
Jacobian.
\end{remark}

\section{Quadratic normal models: Morse--Bott and conic constraints}
\label{sec:quadratic-models}

The usual smooth Laplace method is classical, and the clean Morse--Bott
minimum-manifold expansion is treated, for example, by Ludewig~\cite{Ludewig}.
On a regular quadratic piece, the abstract hypotheses reduce to the standard
formula, with an additional Gaussian solid-angle factor when the admissible cell
is a cone.

\begin{corollary}[Conic quadratic normal formula]\label{cor:conic-quadratic}
Fix a normal-cell piece of codimension $d\ge1$, suppress its index, and
assume that it satisfies Assumption~\ref{ass:tube}.  Suppose that its
admissible cells are cones
$\Afib(p)=C(p)$, and suppose that $a$ is bounded on its truncated tube.
Suppose, uniformly in $p$ and for $s\in C(p)$ near zero,
\begin{align}
 f(F(p,s))-f_0
 &=\frac12\inner{H(p)s}{s}+O(|s|^3),\label{eq:quadratic-normal-phase}\\
 a(F(p,s))J(p,s)&=a(p)+O(|s|),\label{eq:quadratic-normal-amplitude}
\end{align}
where the symmetric matrices $H(p)$ satisfy
\begin{equation}\label{eq:quadratic-normal-ellipticity}
 0<\lambda_-|t|^2\le\inner{H(p)t}{t}\le\lambda_+|t|^2.
\end{equation}
Then
\begin{equation}\label{eq:conic-quadratic-result}
 \cI(z,p)=e^{-zf_0}z^{-d/2}
 \left[
 a(p)\int_{C(p)}e^{-\inner{H(p)t}{t}/2}\dd t
 +O(z^{-1/2})
 \right]
\end{equation}
uniformly in $p$.  If $Y_p$ is a centered Gaussian vector with covariance
$H(p)^{-1}$, then
\begin{equation}\label{eq:gaussian-cone-formula}
 \int_{C(p)}e^{-\inner{H(p)t}{t}/2}\dd t
 =\frac{(2\pi)^{d/2}}{\sqrt{\det H(p)}}
  \mathbb P\{Y_p\in C(p)\}.
\end{equation}
For $C(p)=\R^d$ this is the classical Morse--Bott coefficient; for a half-space
through the origin, the Gaussian probability is $1/2$.
\end{corollary}

\begin{proof}
Take $\beta_i=1/2$, so $Q=d/2$, the dilation is scalar, every cone $C(p)$ is
invariant, and $\rho_{\bm\beta}(s)=|s|^2$.  With
$G(p,s)=\inner{H(p)s}{s}/2$, conditions
\eqref{eq:quadratic-normal-phase}--\eqref{eq:quadratic-normal-amplitude} are exactly
\eqref{eq:quantitative-jets} with $\eta=1/2$; ellipticity gives the phase lower
bound after shrinking the tube.  Theorem~\ref{thm:quantitative-leading} proves
\eqref{eq:conic-quadratic-result}.  Formula~\eqref{eq:gaussian-cone-formula}
is the normalization identity for the Gaussian density.
\end{proof}

\begin{example}[Rank-deficient positive-semidefinite posterior]
\label{ex:psd-posterior}
Let $\Sym_d$ be the Euclidean space of real symmetric $d\times d$
matrices with the Frobenius metric, let $d_{\mathrm{sym}}:=d(d+1)/2$, and put
$\Omega=\Sym_d^+$.  This is the local Gaussian model for Bayesian
estimation of a covariance matrix constrained to be positive semidefinite.
The same rank-stratum geometry occurs for low-rank density matrices in
Bayesian quantum-state tomography~\cite{SixRouchon}.

Fix $\Sigma_0\in\Omega$ of rank $r<d$, put $k=d-r$, and consider
\[
 Z_z=\int_\Omega a(\Sigma)
 \exp\!\left\{-\frac z2\|\Sigma-\Sigma_0\|_F^2\right\}\dd\Sigma,
\]
where $a$ is bounded and continuous with $a(\Sigma_0)>0$.  Choose an
orthogonal basis in which $\Sigma_0=\diag(\Sigma_+,0_k)$ with
$\Sigma_+\succ0$, and define
\[
 T_{\Sigma_0}\Omega
 =\left\{
 \begin{pmatrix}A&B\\B^{\mathsf T}&C\end{pmatrix}:
 A\in\Sym_r,\ B\in\R^{r\times k},\ C\in\Sym_k^+
 \right\}.
\]
The Schur-complement criterion
\[
 \Sigma_0+t\begin{pmatrix}A&B\\B^{\mathsf T}&C\end{pmatrix}\succeq0
 \quad\Longleftrightarrow\quad
 C-tB^{\mathsf T}(\Sigma_++tA)^{-1}B\succeq0
\]
identifies the scaled feasible sets.  Indeed, for a fixed matrix
$\Delta=\begin{psmallmatrix}A&B\\B^{\mathsf T}&C\end{psmallmatrix}$, membership of
$\Delta$ in $\sqrt z(\Omega-\Sigma_0)$ is the same as positive semidefiniteness of
$\Sigma_0+z^{-1/2}\Delta$.  The displayed criterion, with $t=z^{-1/2}$, shows that
this membership converges to the condition $C\succeq0$ at every point with
$\det C\ne0$: if $C\succ0$ the condition holds for all large $z$, while if
$C$ has a negative eigenvalue it fails for all large $z$.  Thus the indicators
of the scaled feasible sets converge pointwise almost everywhere to
$\ind_{T_{\Sigma_0}\Omega}$.  The exceptional set $\det C=0$ is null, and the
Gaussian weight supplies the required weighted domination.  Hence the
relative-domain hypotheses of Corollary~\ref{cor:coercive-leading-piece} hold with all
weights $1/2$ and $G(\Delta)=\|\Delta\|_F^2/2$.

Applying Corollary~\ref{cor:coercive-leading-piece} gives
\[
 Z_z\sim a(\Sigma_0)z^{-d_{\mathrm{sym}}/2}
 \int_{T_{\Sigma_0}\Omega}e^{-\|\Delta\|_F^2/2}\dd\Delta
 =a(\Sigma_0)(2\pi)^{d_{\mathrm{sym}}/2}p_k\,z^{-d_{\mathrm{sym}}/2},
\]
where
\[
 p_k:=\mathbb P\{\mathbf G_k\succeq0\}
\]
and $\mathbf G_k$ is a standard Gaussian vector in the Euclidean space
$(\Sym_k,\|\cdot\|_F)$.
Thus $p_0=1$ and $p_1=1/2$.  For $k=2$, let $I_2$ denote the
$2\times2$ identity matrix and use the orthonormal basis
\[
 \frac{I_2}{\sqrt2},\qquad
 \frac1{\sqrt2}\begin{pmatrix}1&0\\0&-1\end{pmatrix},\qquad
 \frac1{\sqrt2}\begin{pmatrix}
 0&1
 \\
 1&0
 \end{pmatrix}
\]
of $\Sym_2$.  In the corresponding coordinates
$(x_0,x_1,x_2)$, the positive-semidefinite cone is the Lorentz cone
$x_0\ge\sqrt{x_1^2+x_2^2}$, whose spherical section is a cap of half-angle
$\pi/4$.  Since an isotropic Gaussian has a uniform direction on
$\mathbb S^2$,
\[
 p_2=\frac{2\pi(1-\cos(\pi/4))}{4\pi}
 =\frac{2-\sqrt2}{4}.
\]
For $k\ge2$, the boundary is nonsmooth at $\Sigma_0$; the normal-cell
formula supplies the nontrivial Gaussian solid-angle factor that the
unconstrained Hessian formula would miss.
\end{example}

\begin{remark}[Incident lower strata]
At a lower stratum incident to a higher-dimensional minimum stratum, a smooth
phase that is constant on the latter is typically degenerate in the directions
that point into the higher stratum.  For instance, if the minimum set contains
a half-line $\{(u,0):u\ge0\}$ and the origin is treated as a lower stratum,
then smoothness and the identity $f(u,0)=f_0$ for $u\ge0$ force every pure
$u$-derivative of $f$ at the origin to vanish.  The Hessian at the point
therefore cannot detect the collar direction $u$.  One must either keep the
higher-stratum collar as part of the base variable or introduce a joint
anisotropic scaling of the collar and normal variables.  Thus
Corollary~\ref{cor:conic-quadratic} usually applies to regular strata, whereas
lower strata incident to a family of minimizers may require the profile theorem
after a joint anisotropic rescaling.  In the full-space smooth quadratic case,
the profiles at odd half-integer orders are odd and integrate to zero,
so the classical expansion after the factor $z^{-d/2}$ proceeds in integer
powers of $z^{-1}$.
\end{remark}

\section{Euclidean and flat ambient manifolds}\label{sec:euclidean}

The Fermi-coordinate Jacobian and tube geometry used here are standard; see
Gray's treatment of tubes~\cite{GrayTubes}.

Let $\Sigma\subset\R^n$ be a $C^2$ embedded submanifold of dimension $j$, and
let $\eta_1,\ldots,\eta_d$ be a local orthonormal normal frame, $d=n-j$.  Put
\begin{equation}\label{eq:eta-s}
  \eta(p,s):=\sum_{i=1}^d s_i\eta_i(p),
  \qquad F(p,s):=p+\eta(p,s).
\end{equation}
Let $\mathrm{II}$ be the second fundamental form and define the shape operator
$A_\nu:T_p\Sigma\to T_p\Sigma$ by
\begin{equation}\label{eq:shape-operator}
  \inner{A_\nu X}{Y}=\inner{\mathrm{II}(X,Y)}{\nu}.
\end{equation}

\begin{proposition}[Euclidean Fermi Jacobian]\label{prop:euclidean-J}
On every normal-frame chart on which $F$ is nonsingular, its volume Jacobian is
\begin{equation}\label{eq:euclidean-J}
  J(p,s)=\left|\det\bigl(I-A_{\eta(p,s)}\bigr)\right|.
\end{equation}
For $|s|$ sufficiently small the determinant is positive, so the absolute
value may be omitted.  In particular,
\begin{equation}\label{eq:J-polynomial}
  J(p,s)=\det\left(I-\sum_{i=1}^d s_iA_{\eta_i(p)}\right)
\end{equation}
is a polynomial in $s$ of degree at most $j$ and satisfies $J(p,0)=1$.
\end{proposition}

\begin{proof}
For $X\in T_p\Sigma$ and $r\in\R^d$,
\begin{equation*}
 dF_{(p,s)}(X,r)
 =X+\nabla_X\eta(p,s)+\sum_{i=1}^d r_i\eta_i(p).
\end{equation*}
The Weingarten decomposition gives
$\nabla_X\eta=-A_\eta X+\nabla_X^\perp\eta$.  Relative to orthonormal bases of
$T_p\Sigma\oplus\R^d$ and
$T_p\Sigma\oplus N_p\Sigma$, the matrix of $dF$ is block triangular:
\[
 \begin{pmatrix}
 I-A_\eta & 0\\
 * & I
 \end{pmatrix}.
\]
Its absolute determinant is~\eqref{eq:euclidean-J}.  Since
$A_{\eta(p,s)}=\sum_i s_iA_{\eta_i(p)}$, formula~\eqref{eq:J-polynomial}
follows.  The determinant equals one at $s=0$ and hence remains positive in a
sufficiently small chart.
\end{proof}

\begin{corollary}[Euclidean finite expansions]\label{cor:euclidean-expansion}
Let $N=\R^n$ and choose a scaled cell representation as in
Definition~\ref{def:scaled-cell-representation}.  After any required cell
straightening, suppose that the transported quantities $\widetilde b_q$ and $\Phi_q$
satisfy the hypotheses of Proposition~\ref{prop:homogeneous-jets}.  Then
Theorem~\ref{thm:finite-expansion} gives the corresponding finite expansion.

In the locally stationary Fermi-cell case, this condition follows directly
from homogeneous normal jets.  Indeed, assume the locally stationary cell
condition and the homogeneous
normal-jet hypotheses of Corollary~\ref{cor:homogeneous-stable-cell} for the
normalized phase $\phi$ and the amplitude $a\circ F$, and assume that the
Fermi Jacobian obeys
the uniform normalization~\eqref{eq:J-uniform}.  Group the polynomial
\eqref{eq:J-polynomial} into weighted-homogeneous pieces:
\[
 J(p,D_u^{\bm\beta}t)
 =1+\sum_{\lambda\in\Lambda_J}u^\lambda J_\lambda(p,t)
 \qquad\bigl(\rho_{\bm\beta}(t)=1\bigr).
\]
When $\Lambda_J\ne\varnothing$, evaluate this identity at
$|\Lambda_J|$ fixed sufficiently small positive values of $u$.  The associated
generalized Vandermonde matrix is invertible, and
\eqref{eq:J-uniform} bounds the left-hand sides uniformly; hence every
$J_\lambda$ is uniformly bounded on the admissible weighted unit fibers.
Multiplying the jet of $a\circ F$ by this finite Jacobian polynomial produces
the jet of $b=(a\circ F)J$.  The constant term of the Fermi Jacobian is one, so
extrinsic curvature first enters through positive weighted orders.  For a
moving Euclidean cell, the straightening Jacobian is included in the
transported amplitude $\widetilde b_q$ in~\eqref{eq:transported-amplitude-phase}.
\end{corollary}

The same conclusion holds on a flat torus $\R^n/\Gamma_{\mathrm{lat}}$,
where $\Gamma_{\mathrm{lat}}\subset\R^n$ is a full-rank lattice, because the
quotient map is a local isometry and the Jacobian calculation is local.

\section{Radial models in polar Fermi coordinates}\label{sec:radial}

The leading normal-cell results of Section~\ref{sec:leading} use one fixed
dilation structure on each piece.  A different situation occurs when the radial order varies with the
normal direction.  We use the following uniform radial Watson--Laplace
statement; compare standard accounts of Laplace's method and Watson's
lemma~\cite{Wong}.

Fix a normal-cell piece, suppress its index, and let $P$ be its base and
$d\ge1$ its codimension.  Suppose its
admissible cells are cones, $\Afib(p)=C(p)$, and write
\begin{equation*}
  S(p):=C(p)\cap\mathbb S^{d-1}.
\end{equation*}
Let $\sigma$ denote surface measure on $\mathbb S^{d-1}$, restricted to the
measurable subsets $S(p)$.  Unless stated otherwise, a supremum over
$(p,\theta)$ in this section is taken over the measurable graph
$\{(p,\theta):p\in P,\ \theta\in S(p)\}$.  For $r>0$ and
$\theta\in S(p)$, write
\[
 F(p,r,\theta):=F(p,r\theta),
 \qquad
 \widehat J(p,r,\theta):=J(p,r\theta).
\]
Assume these polar Fermi coordinates satisfy the change-of-variables formula
\begin{equation}\label{eq:polar-cov}
 \cI(z,p)=\int_{S(p)}\int_0^\eps
 a(F(p,r,\theta))\widehat J(p,r,\theta)
 e^{-zf(F(p,r,\theta))}r^{d-1}\dd r\dd\sigma(\theta),
\end{equation}
where $\widehat J(p,r,\theta)\to1$ uniformly as $r\downarrow0$.
This is exactly the Cartesian Fermi formula followed by ordinary polar
coordinates in the normal fiber.

\begin{assumption}[Uniform radial model]\label{ass:radial}
There are measurable functions
$c(p,\theta)$ and $\alpha(p,\theta)$ such that
\begin{equation}\label{eq:radial-bounds}
 0<c_-\le c(p,\theta)\le c_+<\infty,
 \qquad
 0<\alpha_-\le\alpha(p,\theta)\le\alpha_+<\infty,
\end{equation}
and
\begin{equation}\label{eq:radial-phase}
 f(F(p,r,\theta))-f_0
 =c(p,\theta)r^{\alpha(p,\theta)}(1+\mathcal R_{\mathrm{rad}}(p,r,\theta)),
\end{equation}
where $\mathcal R_{\mathrm{rad}}(p,0,\theta)=0$ and $\mathcal R_{\mathrm{rad}}\to0$ uniformly in $(p,\theta)$ as
$r\downarrow0$.  The radius $\eps$ is chosen so that
$|\mathcal R_{\mathrm{rad}}(p,r,\theta)|\le\delta_0<1$ throughout the polar coordinate domain.
Also $a(F(p,r,\theta))\widehat J(p,r,\theta)$ is uniformly bounded and
converges to $a(p)$ uniformly as $r\downarrow0$.
\end{assumption}

For fixed $(p,\theta)$, define the radial fiber integral
\begin{equation}\label{eq:radial-fiber}
 \cI_{\mathrm{rad}}(z,p,\theta):=\int_0^\eps
 a(F(p,r,\theta))\widehat J(p,r,\theta)
 e^{-zf(F(p,r,\theta))}r^{d-1}\dd r.
\end{equation}

\begin{theorem}[Uniform radial Laplace theorem]\label{thm:radial}
Under Assumption~\ref{ass:radial},
\begin{equation}\label{eq:radial-uniform}
 \sup_{p,\theta}\left|
 e^{zf_0}z^{d/\alpha(p,\theta)}\cI_{\mathrm{rad}}(z,p,\theta)
 -\frac{a(p)}{\alpha(p,\theta)}
  c(p,\theta)^{-d/\alpha(p,\theta)}
  \Gamma\left(\frac{d}{\alpha(p,\theta)}\right)
 \right|\longrightarrow0.
\end{equation}
Consequently,
\begin{equation}\label{eq:radial-integrated}
 \cI(z,p)=e^{-zf_0}\int_{S(p)}
 z^{-d/\alpha(p,\theta)}
 \bigl(\cL_{\mathrm{rad}}(p,\theta)+\delta^{\mathrm{rad}}_z(p,\theta)\bigr)\dd\sigma(\theta),
\end{equation}
where
\begin{equation}\label{eq:radial-coefficient}
 \cL_{\mathrm{rad}}(p,\theta):=
 \frac{a(p)}{\alpha(p,\theta)}
 c(p,\theta)^{-d/\alpha(p,\theta)}
 \Gamma\left(\frac{d}{\alpha(p,\theta)}\right),
 \qquad
 \sup_{p,\theta}|\delta^{\mathrm{rad}}_z(p,\theta)|\to0.
\end{equation}
\end{theorem}

\begin{proof}
Set
$u=(zc(p,\theta))^{1/\alpha(p,\theta)}r$.  The rescaled upper limit tends to
infinity uniformly by~\eqref{eq:radial-bounds}.  On each bounded $u$-interval,
the amplitude and phase remainders converge uniformly.  For a fixed
$\delta$ with $\delta_0<\delta<1$, the absolute value of the rescaled
integrand is bounded by
\[
 C u^{d-1}e^{-(1-\delta)u^{\alpha(p,\theta)}}.
\]
For $u\ge1$, this is at most
$Cu^{d-1}e^{-(1-\delta)u^{\alpha_-}}$, which is integrable and has uniformly
vanishing tails.  Uniform dominated convergence gives
\[
 \int_0^\infty u^{d-1}e^{-u^{\alpha(p,\theta)}}\dd u
 =\frac1{\alpha(p,\theta)}
  \Gamma\left(\frac d{\alpha(p,\theta)}\right),
\]
proving~\eqref{eq:radial-uniform}.  Formula~\eqref{eq:radial-integrated}
follows by integration in $\theta$.
\end{proof}

\begin{corollary}[Constant radial order]\label{cor:constant-radial}
If $\alpha(p,\theta)\equiv\alpha_0$ on the piece, then
\begin{equation}\label{eq:constant-radial-result}
 I_\alpha(z)=e^{-zf_0}z^{-d/\alpha_0}
 \left[
 \frac{\Gamma(d/\alpha_0)}{\alpha_0}
 \int_P\int_{S(p)}a(p)c(p,\theta)^{-d/\alpha_0}
 \dd\sigma(\theta)\dvol_P(p)+o(1)
 \right].
\end{equation}
\end{corollary}

\begin{proof}
With $\alpha\equiv\alpha_0$, factor $z^{-d/\alpha_0}$ from
\eqref{eq:radial-integrated}.  Since $S(p)\subset\mathbb S^{d-1}$ and $P$
has finite volume,
\[
 \int_P\sigma(S(p))\dvol_P(p)
 \le\sigma(\mathbb S^{d-1})\vol_P(P)<\infty.
\]
The uniform remainder in~\eqref{eq:radial-integrated} is therefore integrable,
and integrating~\eqref{eq:radial-coefficient} gives
\eqref{eq:constant-radial-result}.
\end{proof}

\begin{remark}[Variable radial order]
If $\alpha$ varies with $\theta$, no single power of $z$ emerges, and
formula~\eqref{eq:radial-integrated} itself is the general leading statement.  Subject to
nonvanishing of the corresponding coefficient, the potentially dominant
directions maximize $\alpha$, and a further Laplace analysis with large
parameter $\log z$ is needed near that maximizing set.  Without hypotheses
on the maximizers, no universal power or logarithmic correction exists.
\end{remark}

\section{Scope of the normal-cell hypotheses and regimes requiring
additional construction}
\label{sec:geometric-scope}

The scope of the results is best described in terms of the data that must be
supplied, rather than by prescribing a particular class of stratified sets.
There are three logically distinct levels.

First, Assumption~\ref{ass:tube} requires a finite almost-everywhere disjoint
normal-cell disintegration, with measurable base-dependent cells and an exact
change-of-variables formula.  The paper does not assert that an arbitrary
stratified minimum set admits such a disintegration.

Second, the leading-order results require convergence of the rescaled phase,
amplitude, and cells, together with sufficient tightness.  This may be verified
uniformly, directly on noncoercive limiting cells, or pointwise in the base
under an integrable majorant, according to the theorem being used.

Third, a finite expansion requires a scaled-cell representation on a reference
cell for which the complete transported density admits the weighted
$L^1$ expansion of Assumption~\ref{ass:scaled-expansion}.  The framework allows
the cell motion to enter through the representation and its Jacobian, but it
does not construct such a representation automatically.

The following situations are therefore not covered automatically.  In each
case, the obstruction is not the named geometry by itself, but the failure to
supply the corresponding normal-cell data after an appropriate finite
refinement:

\begin{enumerate}[leftmargin=2.3em]

\item \emph{Unresolved branching or crossing assignments.}
Branching and crossing strata are admissible when a finite Borel refinement or
measurable tie-breaking rule assigns almost every nearby point to exactly one
base piece.  They fall outside Assumption~\ref{ass:tube} when incident strata
produce coordinate images that overlap on positive volume and no
almost-everywhere one-to-one assignment has been constructed.

\item \emph{Collapse of the adapted coordinate geometry.}
At a cusp or another singular point, ordinary normal coordinates may exist
along the regular strata while their usable radius tends to zero or their
Jacobians lose uniform control.  Such a geometry is still admissible if a
different finite family of adapted charts provides the required
change-of-variables formula.  It is not covered automatically when no such
uniform local description has been supplied.

\item \emph{Failure of scaled-cell convergence.}
A change of limiting cell across different base regions is not by itself an
obstruction: the regions may be separated into finitely many base pieces, and
the integrated pointwise-cell theorem also permits almost-everywhere cell
convergence under an integrable base majorant.  Additional analysis is required
when, on every available finite refinement, the weighted blow-ups fail to have
the uniform or pointwise limiting behavior required by the relevant theorem.

\item \emph{Iterated or unresolved frontier interactions.}
Several collar variables may be promoted together with the normal variables to
one joint scaled fiber, as in Section~\ref{sec:leading}.  The present results do
not automatically cover nested frontier interactions for which this
finite-dimensional joint scaling still leaves non-tight mass, creates another
unresolved frontier scale, or changes the dominant model on successively
smaller regions.

\item \emph{Uncontrolled higher-order motion of cell boundaries.}
Moving cells are covered at finite order when they admit a scaled-cell
representation for which the complete transported density has a controlled
weighted $L^1$ expansion.  This includes graph boundaries and other moving
domains that can be straightened with sufficiently controlled phase,
amplitude, and Jacobian expansions.  A separate shape or distributional
analysis may be needed when no such representation and expansion can be
verified; in that case the first correction can be supported on a limiting
boundary rather than represented by an ordinary $L^1$ density.

\item \emph{Absence of an integrable and tight homogeneous model.}
Ambient coercivity is not required: a model with flat ambient directions is
admissible when the limiting cell cuts off those directions and the resulting
weighted mass is integrable and tight.  Likewise, competing principal parts may
be handled by a finite sector decomposition.  The present hypotheses do not
apply directly when no finite collection of sectors and anisotropic scalings
produces homogeneous models that are integrable and tight on their admissible
limiting cells.

\item \emph{Non-locally-finite or non-tame decompositions.}
The setup uses finitely many base pieces of finite intrinsic volume and a finite
almost-everywhere disjoint coordinate cover.  A minimum set with infinitely
many strata or cell pieces accumulating in every neighborhood is outside the
stated framework when it cannot be reorganized into such a finite
normal-cell atlas.  An extension to countably many pieces would require
additional summability and uniformity assumptions not imposed here.

\end{enumerate}

These items are failure modes of the supplied data, not a blacklist of
singular geometric shapes.  For example, measurable tie-breaking may resolve
item~1, nonlinear adapted coordinates may resolve item~2, the pointwise-cell
theorem or a finite sector decomposition may resolve item~3, a joint
frontier-normal model may resolve item~4, and a scaled-cell straightening may
resolve item~5.  Likewise, item~6 concerns integrability on the admissible
limiting cell, not coercivity on the whole ambient fiber.  When one of these
constructions succeeds, the resulting pieces fall within the abstract
theorems.

Newton-polyhedron and resolution methods are complementary upstream tools,
rather than alternative cases in a linear hierarchy.  The present paper
neither assumes nor proves that arbitrary analytic or ideal-theoretic
singularities admit an adapted normal-cell disintegration.  In problems where a
Newton-sector decomposition or a resolution procedure produces finitely many
charts, transformed domains, and weighted principal models
\cite{WatanabeAGSLT,LinRLCT}, the results of this paper may be applied
chartwise once the normal-cell change of variables, cell convergence,
tightness, and, when needed, transported-density expansion have been verified.
Conversely, resolution methods may treat analytic singularities for which no
single finite normal-cell description of the present type is available.
Thus resolution is not needed when the adapted data can be checked directly,
but the two approaches are not linearly ordered.

\section{A planar inverse-kinematics problem with joint limits}
\label{sec:robotics-example}

\subsection{Problem formulation and terminology}

A configuration of a robot arm is a list of its joint variables.  In the
present example the joints are rotating hinges, so a configuration is a triple
of angles $\bm q=(q_1,q_2,q_3)$.  The \emph{end effector} is the tip of the last
link.  The \emph{forward-kinematics map} $X$ sends the joint angles to the
Cartesian position of that tip.  The inverse-kinematics problem reverses this
map: for a prescribed target $x^\star$, find all admissible configurations
$\bm q$ satisfying
\[
 X(\bm q)=x^\star.
\]
Bounds imposed on individual joint angles are called \emph{joint limits};
they make the admissible configuration set a domain with boundary and corners.

In redundant or uncertain inverse kinematics, one may seek the full exact
solution set or a probability distribution over exact and approximate
configurations.  We use \emph{soft inverse kinematics} for the latter
formulation: a configuration whose tip misses the
target by a vector $X(\bm q)-x^\star$ is assigned a weight proportional to
\[
 \exp\left\{-\frac z2|X(\bm q)-x^\star|^2\right\}.
\]
Here $z$ is the observation precision, so $z^{-1/2}$ is the characteristic
Cartesian error scale and the limit $z\to\infty$ is the small-noise limit.
Here we choose a Gaussian task-space likelihood.  Related probabilistic and
Bayesian formulations appear in sequential Monte Carlo inverse kinematics and
biomechanical posterior inference~\cite{CourtyArnaudSMCIK,PatakyBayesianIK}.
Optimization-based treatments of serial-chain inverse kinematics can incorporate
joint limits~\cite{MaricEtAlIK}, and generative inverse-kinematics methods seek
diverse solutions rather than one selected configuration~\cite{AmesEtAlIKFlow}.
The use of Laplace approximation for posterior moments and marginal densities is
standard; see Tierney--Kadane~\cite{TierneyKadane}.  The purpose of the example is the
asymptotic analysis of this posterior in a singular small-noise regime.

The dimensional mechanism behind the example is as follows.  Suppose
an $(m+1)$-dimensional configuration space is mapped to an $m$-dimensional
\emph{task space}, meaning the space of measured output coordinates.  The arm
then has one \emph{redundant degree of freedom}: at a regular exact solution,
the derivative of the task map has rank $m$, and the exact solutions locally
form a one-dimensional curve.  Motion along this curve changes the joints while
leaving the task output fixed; such self-motion manifolds are classical for
redundant manipulators \cite{BurdickSelfMotion}, and planar $3R$ self-motion
curves with joint limits were studied by Lenar\v{c}i\v{c} \cite{Lenarcic3R}.
There are $m$ transverse directions in which the task error is linear, so the
squared-error phase is quadratic and the curve contributes at order $z^{-m/2}$.

We use the term \emph{one-sided joint-limit fold} for the following boundary
singularity.  The configuration lies at a joint-limit corner, so two local
joint coordinates $u,v$ are allowed only in their inward cone.  The task map
loses first-order sensitivity in those two directions, and their first
nonzero task displacement is the quadratic form $\mathcal Q(u,v)$.  In suitable source
and target coordinates, the leading residual has the form
\begin{equation}\label{eq:robot-generic-fold}
 (x_1,\ldots,x_{m-1},u,v)
 \longmapsto
 \bigl(x_1,\ldots,x_{m-1},\mathcal Q(u,v)\bigr),
\end{equation}
where $\mathcal Q$ is strictly positive on the inward cone away from the origin.  The
word ``one-sided'' refers to the joint limits: only inward values of $(u,v)$
are feasible.  After squaring the residual, the ordinary directions have
quadratic scale $z^{-1/2}$, whereas the two fold directions have quartic scale
$z^{-1/4}$.  Equivalently, the weights are
\begin{equation}\label{eq:robot-generic-weights}
 \beta_{x_j}=\frac12,
 \qquad
 \beta_u=\beta_v=\frac14,
 \qquad
 Q_{\mathrm{fold}}
 =\frac{m-1}{2}+\frac14+\frac14
 =\frac m2.
\end{equation}
Thus an isolated configuration at such a corner can have exactly the same
small-noise mass order as an entire regular self-motion curve.  The concrete arm
below realizes this balance for a three-link planar arm whose task is the full
two-dimensional position of its tip.

\subsection{The concrete three-link arm and its soft posterior}

We identify the plane $\R^2$ with $\C$ only as a compact notation for planar
rotations: multiplication by $e^{i\theta}$ rotates a vector through the angle
$\theta$.  Let
$\mathbb T:=\R/(2\pi\mathbb Z)\cong\mathrm{SO}(2)$ denote the circle of possible
angles.  Consider three rigid links of lengths $2,1,2$, connected in sequence
by three revolute joints.  The first angle $q_1$ sets the orientation of the
first link, while $q_2$ and $q_3$ are the successive relative turning angles.
Without limits, the configuration space is the three-torus
$\mathsf Q:=\mathbb T^3$, equipped with its standard product metric.

The first joint may rotate freely, while the second and third joints satisfy
\begin{equation}\label{eq:robot-domain}
 \Omega:=\mathbb T\times
 \left[0,\frac{3\pi}{2}\right]\times
 \left[\pi,\frac{4\pi}{3}\right].
\end{equation}
Thus $\Omega$ is the set of physically allowed angle triples.  The two
intervals are closed arcs in their rotation groups, so their endpoints are
actual joint limits rather than points identified modulo $2\pi$.  For
$\bm q=(q_1,q_2,q_3)\in\Omega$, the forward-kinematics map (the position of
the arm tip) is
\begin{equation}\label{eq:robot-forward-map}
 X(\bm q)
 =2e^{iq_1}+e^{i(q_1+q_2)}+2e^{i(q_1+q_2+q_3)}\in\C.
\end{equation}

We prescribe the target $x^\star=1\in\C$.  Both Cartesian coordinates are
observed; this is what ``full-position'' means here.  Equivalently, one may
write the observation model as
\[
 x^\star=X(\bm q)+\bm\epsilon,
 \qquad
 \bm\epsilon\sim\mathcal N(0,z^{-1}I_2),
\]
where $I_2$ is the $2\times2$ identity matrix.
Bayes' rule then gives the soft inverse-kinematics posterior
\begin{equation}\label{eq:robot-posterior}
 \Pi_z(\dd \bm q)
 =\frac1{Z_z}\,\varpi(\bm q)
 \exp\left\{-\frac z2|X(\bm q)-1|^2\right\}\dd \bm q,
 \qquad \bm q\in\Omega,
\end{equation}
where $\varpi\in C^2(\Omega)$ is a strictly positive prior weight expressing
preferences among configurations before the target observation,
$\dd \bm q=\dd q_1\dd q_2\dd q_3$, and $Z_z$ is the normalizing constant.  We do
not require $\varpi$ itself to be normalized: multiplying it by a positive
constant changes $Z_z$ by the same factor and leaves the posterior unchanged.
The adjective ``soft'' emphasizes that configurations missing the target still
receive positive, though exponentially smaller, weight.

\subsection{Exact configurations and the task Jacobian}

Define the shape vector, the tip position when the first joint angle is set to
zero, by
\begin{equation}\label{eq:robot-shape-vector}
 S(u,v):=2+e^{iu}+2e^{i(u+v)}.
\end{equation}
Since $X(\bm q)=e^{iq_1}S(q_2,q_3)$, changing $q_1$ only rotates the shape
vector.  Therefore a pair of internal joint angles $(u,v)$ can reach the
target $1$ if and only if $|S(u,v)|=1$; in that case there is a unique first
joint angle, modulo a full turn, that rotates $S(u,v)$ onto the target.
Direct expansion gives the factorization
\begin{equation}\label{eq:robot-shape-factorization}
 |S(u,v)|^2-1
 =8\cos\frac{u+v}{2}
 \left(
  \cos\frac{u-v}{2}+2\cos\frac{u+v}{2}
 \right).
\end{equation}
For $u\in[4\pi/3,3\pi/2]$, set
\begin{equation}\label{eq:robot-vGamma}
 v_\Gamma(u):=
 2\pi-2\arctan\left(-\frac{3}{\tan(u/2)}\right),
\end{equation}
and let $\vartheta(u)\in\mathbb T$ be the unique angular class satisfying
\begin{equation}\label{eq:robot-vartheta}
 e^{i\vartheta(u)}S(u,v_\Gamma(u))=1.
\end{equation}
Then the exact inverse-kinematics set is
\begin{equation}\label{eq:robot-minimum-set}
 M:=\{\bm q\in\Omega:X(\bm q)=1\}
 =\{\bm q_0\}\sqcup\Gamma,
 \qquad
 \bm q_0:=(0,0,\pi),
\end{equation}
where
\begin{equation}\label{eq:robot-regular-arc}
 \Gamma:=\left\{
  \bigl(\vartheta(u),u,v_\Gamma(u)\bigr):
  \frac{4\pi}{3}\le u\le\frac{3\pi}{2}
 \right\}.
\end{equation}
The set $\Gamma$ is a smooth compact curve of exact configurations.  Along
$\Gamma$ the task Jacobian has rank two, so the two scalar equations specifying
the Cartesian target are locally independent constraints.  Moving along $\Gamma$ changes the three joint angles while
keeping the arm tip fixed at the target, so $\Gamma$ is the self-motion curve
in this example.  Its endpoints lie on the joint-limit faces $q_3=4\pi/3$ and
$q_2=3\pi/2$, respectively, and meet those faces transversely.  The point
$\bm q_0$ is a second, isolated exact configuration.  It is at the corner where
both constrained joints attain their lower limits, namely $q_2=0$ and
$q_3=\pi$.

The \emph{task Jacobian} $DX(\bm q)$ is the real $2\times3$ derivative of the
forward-kinematics map: it converts an infinitesimal change of the three joint
angles into the corresponding infinitesimal Cartesian displacement of the
arm tip.  Its two-dimensional Jacobian factor is
\begin{equation}\label{eq:robot-task-jacobian}
 J_X(\bm q):=\sqrt{\det\bigl(DX(\bm q)DX(\bm q)^{\mathsf T}\bigr)}.
\end{equation}
The scalar $J_X$ is positive exactly when $DX$ has full row rank two.
Proposition~\ref{prop:robot-posterior-splitting} verifies this positivity
everywhere on $\Gamma$.

\begin{lemma}[Relative normal cells for a transverse embedded arc]
\label{lem:transverse-arc-normal-cells}
Let $N$ be an $n$-dimensional Riemannian manifold, let $\Omega\subset N$ be a
relative domain, and let $\Gamma=\gamma([0,L])$ be a compact $C^2$ embedded
unit-speed arc.  Assume that $\Gamma^\circ\subset\operatorname{int}\Omega$.
For each endpoint $e$, let $\tau_e$ be the unit tangent pointing from $e$ into
$\Gamma^\circ$, and suppose that near $e$ there is a $C^2$ defining function
$h_e$ such that
\[
 \Omega=\{h_e\ge0\},\qquad \dd h_e(e)\ne0,
 \qquad \dd h_e(e)[\tau_e]>0.
\]
Then, after decreasing $\eps>0$, the nearest-point projection
$\pi_\Gamma:\cT_\eps(\Gamma)\to\Gamma$ is well defined and
$\Omega\cap\cT_\eps(\Gamma)$ has a relative normal-cell decomposition with
base pieces $\Gamma^\circ$ and the two endpoints.

More precisely, choose a $C^1$ orthonormal normal frame
$E_p:\R^{n-1}\to N_p\Gamma$ along the arc.  Over $\Gamma^\circ$ set
\[
 F_\Gamma(p,s):=\exp_p(E_p s),\qquad
 \Afib_\Gamma(p):=
 \{s\in\R^{n-1}:|s|<\eps,\ F_\Gamma(p,s)\in\Omega,\
                    \pi_\Gamma(F_\Gamma(p,s))=p\}.
\]
At an endpoint $e$, identify $T_eN$ with $\R^n$ by an orthonormal frame and set
\[
 F_e(v):=\exp_e(v),\qquad
 \Afib_e:=\{v\in T_eN:|v|<\eps,\ F_e(v)\in\Omega,\
                      \pi_\Gamma(F_e(v))=e\}.
\]
The cells are measurable.  Their three images are pairwise disjoint up to
null sets, cover $\Omega\cap\cT_\eps(\Gamma)$ up to null sets, and satisfy the
corresponding change-of-variables formulas.  Their Jacobians equal one on the
zero sections.
For each fixed $p\in\Gamma^\circ$, the scalar blow-ups of
$\Afib_\Gamma(p)$ converge locally to $\R^{n-1}$.  At an endpoint,
\begin{equation}\label{eq:transverse-arc-endpoint-cone}
 r^{-1}\Afib_e\longrightarrow
 C_e:=\{v\in T_eN:\inner{v}{\tau_e}\le0,\
                 \dd h_e(e)[v]\ge0\}
 \qquad(r\downarrow0)
\end{equation}
locally in measure.  In addition, the directions in the endpoint cell approach
those of the limiting cone:
\begin{equation}\label{eq:transverse-arc-directional-convergence}
 \sup_{\substack{v\in\Afib_e\\0<|v|\le r}}
 \dist\!\left(\frac{v}{|v|},
 C_e\cap\{w\in T_eN:|w|=1\}\right)\longrightarrow0
 \qquad(r\downarrow0).
\end{equation}
Moreover, for every $a>0$ and every continuous coercive function
$G_e:T_eN\to[0,\infty)$ satisfying $G_e(rv)=r^2G_e(v)$,
\begin{equation}\label{eq:transverse-arc-weighted-cell-limit}
 \int_{T_eN}
 \left|\ind_{r^{-1}\Afib_e}(v)-\ind_{C_e}(v)\right|
 e^{-aG_e(v)}\dd v\longrightarrow0.
\end{equation}
\end{lemma}

\begin{proof}
Extend $\gamma$ slightly past both endpoints as a $C^2$ embedded curve.
The tubular-neighborhood theorem gives unique normal coordinates near the
extended curve.  Away from the endpoints these are the ordinary nearest-point
coordinates for $\Gamma$; near an endpoint, the strict convexity argument below
shows that the distance to the truncated arc has a unique minimizer.  By
compactness, one value of $\eps$ works along the whole arc.  This is the local
tube/reach construction for embedded submanifolds; see
\cite{GrayTubes,Federer}.

Partition the relative tube according to whether this nearest point lies in
$\Gamma^\circ$ or is one of the endpoints.  On the first part the standard
normal exponential map gives the displayed Fermi coordinates.  On an endpoint
part the exponential map at that endpoint is a diffeomorphism on a sufficiently
small ball.  The projection partition makes the images disjoint, and the two
standard coordinate formulas give the change of variables.  The Jacobians are
one at the zero sections.  Since a fixed $p\in\Gamma^\circ$ lies in the interior
of $\Omega$, its admissible normal cell contains a neighborhood of the origin;
its scalar blow-ups therefore exhaust $\R^{n-1}$ locally.

It remains to identify an endpoint blow-up.  Parametrize the arc from $e$ by
unit speed, $\gamma_e(0)=e$ and $\dot\gamma_e(0)=\tau_e$, and define
\[
 \varphi_e(v,t):=\frac12\dist_g\bigl(\exp_e(v),\gamma_e(t)\bigr)^2.
\]
In normal coordinates at $e$, uniformly for small $v$ and $t$,
\[
 \partial_t\varphi_e(v,0)
 =-\inner{v}{\tau_e}+O(|v|^2),
 \qquad
 \partial_t^2\varphi_e(v,t)=1+O(|v|+t).
\]
After shrinking the neighborhood, $t\mapsto\varphi_e(v,t)$ is strictly convex
near zero, while the remainder of the compact arc stays farther away.  Hence
$e$ is the nearest point precisely when
\[
 -\inner{v}{\tau_e}+O(|v|^2)\ge0.
\]
The feasibility condition has the expansion
\[
 h_e(\exp_e v)=\dd h_e(e)[v]+O(|v|^2).
\]
Substituting $v=rw$ shows that the indicator of $r^{-1}\Afib_e$ converges at
every $w$ outside the two boundary hyperplanes to the indicator of $C_e$.
Those hyperplanes are null, so dominated convergence gives local $L^1$
convergence of the indicators, and hence local convergence in measure.  The
same expansions also give
\eqref{eq:transverse-arc-directional-convergence}.  Indeed, otherwise there
would be $v_j\in\Afib_e$ with $v_j\to0$ whose normalized directions remain a
fixed positive distance from $C_e$.  After passing to a subsequence,
$v_j/|v_j|\to w$; division of the two endpoint inequalities by $|v_j|$ then
gives $w\in C_e$, a contradiction.  Finally, a fixed-ball/tail decomposition
and coercivity of $G_e$ give
\eqref{eq:transverse-arc-weighted-cell-limit}.
\end{proof}

\begin{proposition}[A solution curve and an isolated joint-limit fold contribute
at the same order]\label{prop:robot-posterior-splitting}
For every continuous test function $\psi:\Omega\to\C$, put
\[
 Z_z[\psi]:=\int_\Omega
 \psi(\bm q)\varpi(\bm q)
 \exp\left\{-\frac z2|X(\bm q)-1|^2\right\}\dd \bm q.
\]
Thus $Z_z=Z_z[1]$.  Then
\begin{equation}\label{eq:robot-test-asymptotic}
 Z_z[\psi]=z^{-1}\left[
 2\pi\int_\Gamma
 \frac{\psi(p)\varpi(p)}{J_X(p)}\dd\ell(p)
 +C_{\mathrm{fold}}\,\psi(\bm q_0)\varpi(\bm q_0)
 +o(1)\right],
\end{equation}
where $\dd\ell$ is arclength for the product metric and
\begin{equation}\label{eq:robot-fold-constant}
 \begin{split}
 C_{\mathrm{fold}}
 &:=\int_{\R\times\R_+^2}
 \exp\left\{-\frac12\left[
 s^2+\bigl(u^2+4uv+3v^2\bigr)^2
 \right]\right\}\dd s\dd u\dd v\\
 &=\frac{\pi}{4}\log 3.
 \end{split}
\end{equation}
Consequently, if
\begin{align}
 K_{\mathrm{reg}}
 &:=2\pi\int_\Gamma
 \frac{\varpi(p)}{J_X(p)}\dd\ell(p),
 \label{eq:robot-Kreg}\\
 K_{\mathrm{fold}}
 &:=\frac{\pi}{4}\log 3\,\varpi(\bm q_0),
 \label{eq:robot-Kfold}
\end{align}
then
\begin{equation}\label{eq:robot-evidence}
 Z_z=z^{-1}\bigl(K_{\mathrm{reg}}+K_{\mathrm{fold}}+o(1)\bigr).
\end{equation}
Moreover, $\Pi_z$ converges weakly to the probability measure $\Pi_\infty$
defined by
\begin{equation}\label{eq:robot-weak-limit}
 \int_\Omega\psi\dd\Pi_\infty
 =\frac{
 \displaystyle
 2\pi\int_\Gamma
 \frac{\psi(p)\varpi(p)}{J_X(p)}\dd\ell(p)
 +C_{\mathrm{fold}}\psi(\bm q_0)\varpi(\bm q_0)}
 {K_{\mathrm{reg}}+K_{\mathrm{fold}}}.
\end{equation}
In particular,
\begin{equation}\label{eq:robot-atom-mass}
 \Pi_\infty(\{\bm q_0\})
 =\frac{K_{\mathrm{fold}}}
 {K_{\mathrm{reg}}+K_{\mathrm{fold}}}>0.
\end{equation}
For the flat prior $\varpi\equiv1$, the coefficients are explicit:
\begin{equation}\label{eq:robot-flat-prior-coefficients}
 K_{\mathrm{reg}}=\frac{\pi}{2}\log3,
 \qquad
 K_{\mathrm{fold}}=\frac{\pi}{4}\log3,
 \qquad
 \Pi_\infty(\{\bm q_0\})=\frac13.
\end{equation}
\end{proposition}

\begin{proof}
\smallskip
\noindent\emph{Step 1: Exact configurations and regularity.}
We first determine all exact configurations.  In the allowed rectangle for
the internal joint angles,
$u+v\in[\pi,17\pi/6]$.  Hence the first factor in
\eqref{eq:robot-shape-factorization} vanishes only when $u+v=\pi$, which gives
$(u,v)=(0,\pi)$.  Write
\[
 \Xi(u,v):=\cos\frac{u-v}{2}+2\cos\frac{u+v}{2}
 =3\cos\frac u2\cos\frac v2-\sin\frac u2\sin\frac v2.
\]
At $v=\pi$, the equation $\Xi=0$ again gives $u=0$.  For $v>\pi$, no solution
with $u\le\pi$ is possible.  If $u>\pi$, then $\Xi=0$ is equivalent to
\begin{equation}\label{eq:robot-tangent-equation}
 \tan\frac u2\tan\frac v2=3.
\end{equation}
On the prescribed intervals,~\eqref{eq:robot-tangent-equation} has a solution
precisely for $4\pi/3\le u\le3\pi/2$, and that solution is
$v=v_\Gamma(u)$.  This proves~\eqref{eq:robot-minimum-set} and
\eqref{eq:robot-regular-arc}.

The gradient of $\Xi$ cannot vanish on $\{\Xi=0\}$.  Indeed, if both partial
derivatives vanished, then
\[
 \sin\frac{u-v}{2}=\sin\frac{u+v}{2}=0,
\]
which is incompatible with $\Xi=0$.  Along $\Gamma$, the other factor
$\cos((u+v)/2)$ in~\eqref{eq:robot-shape-factorization} is nonzero, so
$\nabla_{u,v}|S|^2\ne0$.  At an exact solution,
$\partial_{q_1}X=i$ is tangent to the unit circle in task space, whereas a
nonzero derivative of $|S|^2$ supplies a radial task direction.  Thus
$\rank DX=2$ on $\Gamma$, and $J_X$ is bounded away from zero
there.  At $(u,v)=(4\pi/3,4\pi/3)$, implicit differentiation gives
$v_\Gamma'(u)=-1$; the other endpoint is parametrized with nonzero $u$-velocity at
$u=3\pi/2$.  Hence the two intersections with the joint-limit faces are
transverse.

\smallskip
\noindent\emph{Step 2: Localization near the two solution components.}
Choose disjoint neighborhoods $U_\Gamma$ and $U_0$ of $\Gamma$ and $\bm q_0$,
with smooth cutoffs $\chi_\Gamma$ and $\chi_0$ equal to one near the respective
components.  Since $\Omega$ is compact and~\eqref{eq:robot-minimum-set} lists
all exact solutions, the remaining part of $Z_z[\psi]$ is exponentially
small.

\smallskip
\noindent\emph{Step 3: The regular solution curve.}
For the regular component, put
\[
 e_-:=\bigl(\vartheta(4\pi/3),4\pi/3,4\pi/3\bigr),
 \qquad
 e_+:=\bigl(\vartheta(3\pi/2),3\pi/2,v_\Gamma(3\pi/2)\bigr),
 \qquad
 \Gamma^\circ:=\Gamma\setminus\{e_-,e_+\}.
\]
We verify the hypotheses of the normal-cell theorems rather than derive this
contribution by a separate coarea asymptotic.  The transversality established in
Step~1 permits the active joint-limit defining functions to be oriented so that
$\dd h_e(e)[\tau_e]>0$.  Lemma~\ref{lem:transverse-arc-normal-cells} then supplies
ordinary nearest-point normal cells over $\Gamma^\circ$ and two endpoint caps,
assigned to the zero-dimensional pieces $\{e_-\}$ and $\{e_+\}$.  These pieces
are almost everywhere disjoint, cover a sufficiently small relative
neighborhood of $\Gamma$, and satisfy the relative-domain form of
Assumption~\ref{ass:tube}.

Choose a $C^1$ orthonormal normal frame along $\Gamma^\circ$ and write
$E_p:\R^2\to N_p\Gamma$ for the associated isometry.  Define
\begin{equation}\label{eq:robot-regular-Hessian}
 H_\Gamma(p):=E_p^{\mathsf T}DX(p)^{\mathsf T}DX(p)E_p,
 \qquad
 G_\Gamma(p,s):=\frac12\inner{H_\Gamma(p)s}{s}.
\end{equation}
Since $\ker DX(p)=T_p\Gamma$ and $DX(p)$ has rank two on the compact arc,
$H_\Gamma(p)$ is uniformly positive definite up to both endpoints.  Taylor's
theorem in normal coordinates gives, uniformly as $s\to0$,
\begin{equation}\label{eq:robot-regular-normal-model}
 \frac12|X(F(p,s))-1|^2
 =G_\Gamma(p,s)+O(|s|^3),
 \qquad p\in\Gamma^\circ.
\end{equation}
For each fixed $p\in\Gamma^\circ$, the point lies in the interior of $\Omega$,
so its relative normal cell contains a neighborhood of zero in $\R^2$.
Consequently, under the scalar quadratic dilation,
\[
 z^{1/2}\Afib_\Gamma(p)\longrightarrow\R^2
\]
in the weighted-indicator sense of
Theorem~\ref{thm:integrated-pointwise-cells}.  This convergence need not be
uniform as $p$ approaches an endpoint, which is precisely why the pointwise
cell theorem is the appropriate result here.

Set $A_\Gamma=\chi_\Gamma\psi\varpi$.  In the corresponding Fermi
coordinates,
\[
 A_\Gamma(F(p,s))J(p,s)\longrightarrow\psi(p)\varpi(p).
\]
The amplitude is bounded, the base $\Gamma^\circ$ has finite length, and the
uniform ellipticity of~\eqref{eq:robot-regular-Hessian} supplies both the
coercive lower bound and an integrable Gaussian majorant.  Thus all hypotheses
of Theorem~\ref{thm:integrated-pointwise-cells} hold with
$\bm\beta=(1/2,1/2)$, $Q=1$, $c\equiv1$, and limit cell $\R^2$.  Denoting the
contribution of this normal-cell piece by $I_{\Gamma^\circ}(z)$,
Theorem~\ref{thm:integrated-pointwise-cells} gives
\begin{equation}\label{eq:robot-regular-interior-asymptotic}
 I_{\Gamma^\circ}(z)
 =z^{-1}\left[
 \int_{\Gamma}
 \psi(p)\varpi(p)
 \int_{\R^2}e^{-G_\Gamma(p,t)}\dd t\,\dd\ell(p)
 +o(1)\right],
\end{equation}
where adding the two endpoints to the outer integral does not change it.
Because $DX(p)$ annihilates the unit tangent to $\Gamma$, an orthonormal
splitting into that tangent and the columns of $E_p$ gives
\[
 \det H_\Gamma(p)
 =\det\bigl(DX(p)DX(p)^{\mathsf T}\bigr)
 =J_X(p)^2.
\]
The Gaussian integral in~\eqref{eq:robot-regular-interior-asymptotic} is
therefore $2\pi/J_X(p)$.

\smallskip
\noindent\emph{Step 4: Endpoint caps of the regular curve.}
It remains only to check that the two cap pieces are lower order.  Fix an
endpoint $e\in\{e_-,e_+\}$, let $\tau_e$ be the unit tangent to $\Gamma$
pointing from $e$ into $\Gamma^\circ$, and let $h_e$ be a local defining
function for the active joint-limit face, oriented so that
$\Omega=\{h_e\ge0\}$.  Transversality gives
$\dd h_e(e)[\tau_e]>0$.  Lemma~\ref{lem:transverse-arc-normal-cells} identifies the scalar blow-up of
the endpoint normal cell as
\begin{equation}\label{eq:robot-endpoint-cone}
 C_e:=\{v\in T_e\mathsf Q:
       \inner{v}{\tau_e}\le0,\ \dd h_e(e)[v]\ge0\}.
\end{equation}
The first inequality is the nearest-point end-cap condition and the second is
feasibility.  Since $\ker DX(e)=\R\tau_e$, the two inequalities and
$\dd h_e(e)[\tau_e]>0$ imply
\[
 C_e\cap\ker DX(e)=\{0\}.
\]
Hence $|DX(e)v|^2$ is bounded below by a positive multiple of $|v|^2$ on
$C_e$.  Choose a closed conic neighborhood $\mathcal V_e$ of $C_e\setminus\{0\}$ that
is still disjoint from the unit kernel directions.  By
\eqref{eq:transverse-arc-directional-convergence}, all nonzero directions in a
sufficiently small endpoint cell lie in $\mathcal V_e$.  We use the homogeneous model
\[
 \widetilde G_e(v)
 :=\frac12|DX(e)v|^2+\dist(v,\mathcal V_e)^2
\]
on the ambient tangent space.  It is coercive everywhere, and on the actual
cell the distance term vanishes, so $\widetilde G_e(v)=|DX(e)v|^2/2$ there.
Taylor expansion gives
\[
 \frac12|X(\exp_e v)-1|^2
 =\widetilde G_e(v)\bigl(1+O(|v|)\bigr)
\]
within that cell, while
\eqref{eq:transverse-arc-weighted-cell-limit} supplies the required weighted
convergence of the scaled cell to $C_e$.  The relative-domain form of
Corollary~\ref{cor:coercive-leading-piece}, with model
$\widetilde G_e$, three quadratic weights, and $Q=3/2$, therefore gives
\[
 I_e(z)=O(z^{-3/2}).
\]
Applying this to both endpoints and combining it with
\eqref{eq:robot-regular-interior-asymptotic} gives the regular contribution
\begin{equation}\label{eq:robot-regular-asymptotic}
 \int_\Omega A_\Gamma(\bm q)e^{-z|X(\bm q)-1|^2/2}\dd \bm q
 =z^{-1}\left[
 2\pi\int_\Gamma\frac{\psi(p)\varpi(p)}{J_X(p)}\dd\ell(p)
 +o(1)\right].
\end{equation}
This piece has exponent $Q_\Gamma=1$; its two endpoint strata have the larger
exponent $3/2$ and do not enter the leading coefficient.

\smallskip
\noindent\emph{Step 5: The isolated joint-limit fold.}
It remains to analyze the isolated exact configuration at the joint-limit
corner.  Near $\bm q_0$, use the following linear coordinates.  Their determinant
is one, so they preserve the volume element:
\begin{equation}\label{eq:robot-corner-coordinates}
 u:=q_2,
 \qquad
 v:=q_3-\pi,
 \qquad
 s:=q_1-u-2v.
\end{equation}
After shrinking the coordinate neighborhood, the relative cell is exactly
\[
 C_0:=\R\times\R_+^2
\]
up to the Euclidean truncation $B_\eps$; the upper joint limits are inactive
there.  The adapted map
\[
 F_0(s,u,v)=(s+u+2v,u,\pi+v)
\]
is linear with volume Jacobian one.  Assumption~\ref{ass:tube} does not require
an orthonormal normal frame, so the normal-cell framework applies directly in
these shear coordinates.  In particular, no auxiliary orthonormalization or
moving-cell comparison is needed at the corner.

Put
\begin{equation}\label{eq:robot-corner-quadratic}
 \mathcal P_{\mathrm{fold}}(u,v):=u^2+4uv+3v^2=(u+v)(u+3v).
\end{equation}
For $h\downarrow0$, Taylor expansion of~\eqref{eq:robot-forward-map} gives,
uniformly for $(\xi,U,V)$ in compact subsets of
$\R\times\R_+^2$,
\begin{equation}\label{eq:robot-weighted-residual}
 \begin{split}
 &X\bigl(h^2\xi+hU+2hV,hU,\pi+hV\bigr)-1\\
 &\hspace{7em}
 =h^2\bigl(\mathcal P_{\mathrm{fold}}(U,V)+i\xi\bigr)+O(h^3).
 \end{split}
\end{equation}
Equivalently, with the weighted quasi-radius
\[
 r_{\mathrm{fold}}(s,u,v):=\bigl[s^2+(u^2+v^2)^2\bigr]^{1/4},
\]
the unscaled residual admits the uniform estimate
\begin{equation}\label{eq:robot-weighted-residual-bound}
 X(s+u+2v,u,\pi+v)-1=\mathcal P_{\mathrm{fold}}(u,v)+is+\mathcal R_{\mathrm{fold}}(s,u,v),
 \qquad |\mathcal R_{\mathrm{fold}}(s,u,v)|\le C r_{\mathrm{fold}}(s,u,v)^3
\end{equation}
on the feasible cone near the origin.  This follows either by grouping the
ordinary Taylor series by the weights $(1/2,1/4,1/4)$ or by applying
\eqref{eq:robot-weighted-residual} on the compact weighted unit sphere.
Thus the squared-residual phase has weighted principal model
\begin{equation}\label{eq:robot-principal-model}
 G(s,u,v):=\frac12\left[s^2+\mathcal P_{\mathrm{fold}}(u,v)^2\right],
 \qquad
 \bm\beta=\left(\frac12,\frac14,\frac14\right),
 \qquad Q=1.
\end{equation}
On $\R\times\R_+^2$,
$\mathcal P_{\mathrm{fold}}(u,v)\ge u^2+3v^2$, so $G$ is comparable to $r_{\mathrm{fold}}(s,u,v)^4$ and is coercive
on the admissible cone.  From~\eqref{eq:robot-weighted-residual-bound},
$|\mathcal P_{\mathrm{fold}}(u,v)+is|=O(r_{\mathrm{fold}}^2)$ and $|\mathcal R_{\mathrm{fold}}|=O(r_{\mathrm{fold}}^3)$; hence the cross term in the squared
residual is $O(r_{\mathrm{fold}}^5)$ and $|\mathcal R_{\mathrm{fold}}|^2=O(r_{\mathrm{fold}}^6)$.  Dividing by the lower bound
$G\ge cr_{\mathrm{fold}}^4$ gives the uniform relative estimate
\begin{equation}\label{eq:robot-phase-relative-error}
 \frac12|X(\bm q)-1|^2
 =G(s,u,v)\left(1+O\left(
 [s^2+(u^2+v^2)^2]^{1/4}
 \right)\right)
\end{equation}
on the cone near the origin.  Although $\mathcal P_{\mathrm{fold}}$ is not coercive on all of $\R^2$,
Corollary~\ref{cor:coercive-leading-piece} requires only an ambient homogeneous
extension agreeing with the phase model on $C_0$.  Put
\begin{equation}\label{eq:robot-extended-model}
 \widehat G(s,u,v):=
 \frac12\left[
 s^2+\mathcal P_{\mathrm{fold}}(u_+,v_+)^2+u_-^4+v_-^4
 \right],
\end{equation}
with the standing positive- and negative-part notation.  This function is continuous,
nonnegative, coercive, and homogeneous for the weights in
\eqref{eq:robot-principal-model}.  On the admissible cell $C_0$ it coincides
with $G$.  The cell $C_0$ is anisotropically invariant, while the truncation by
$B_\eps$ recedes to infinity after scaling.  The relative-domain form of
Corollary~\ref{cor:coercive-leading-piece}, applied in the determinant-one
coordinates
\eqref{eq:robot-corner-coordinates}, therefore gives
\begin{equation}\label{eq:robot-fold-asymptotic}
 \int_\Omega\chi_0(\bm q)\psi(\bm q)\varpi(\bm q)
 e^{-z|X(\bm q)-1|^2/2}\dd \bm q
 =z^{-1}\left[
 C_{\mathrm{fold}}\psi(\bm q_0)\varpi(\bm q_0)+o(1)
 \right].
\end{equation}
This is the step for which the anisotropic scaling is essential.  At
$\bm q_0$, the ordinary Hessian of the squared error has rank one: only the $s$
direction is visible to second order.  The two inward joint motions $u,v$
first appear through the quartic term $\mathcal P_{\mathrm{fold}}(u,v)^2$, so an ordinary
three-dimensional Gaussian Laplace approximation cannot produce the correct
coefficient.

\smallskip
\noindent\emph{Step 6: Coefficient evaluation and global comparison.}
It remains to evaluate the model integral.  Integrating first in $s$ and then
putting $v=tu$ in the first quadrant gives
\begin{align*}
 C_{\mathrm{fold}}
 &=\sqrt{2\pi}\int_{\R_+^2}
 e^{-\mathcal P_{\mathrm{fold}}(u,v)^2/2}\dd u\dd v\\
 &=\sqrt{2\pi}\,\frac{\sqrt{2\pi}}4
 \int_0^\infty\frac{\dd t}{(1+t)(1+3t)}\\
 &=\frac{\pi}{4}\log3,
\end{align*}
which proves~\eqref{eq:robot-fold-constant}.  The regular interior piece
and the isolated fold both have exponent $Q=1$, whereas the two endpoint caps
have exponent $3/2$ and the complement is exponentially small.  The global
comparison theorem, Theorem~\ref{thm:global-leading}, therefore sums
\eqref{eq:robot-regular-asymptotic} and
\eqref{eq:robot-fold-asymptotic} and proves
\eqref{eq:robot-test-asymptotic}.  Taking $\psi=1$ gives
\eqref{eq:robot-evidence}, and division of the two expansions proves
\eqref{eq:robot-weak-limit} and~\eqref{eq:robot-atom-mass}.

\smallskip
\noindent\emph{Step 7: Flat-prior constants.}
For the flat prior, the regular coefficient supplied by
Theorem~\ref{thm:integrated-pointwise-cells} can be evaluated explicitly.
Put $\Delta(u,v):=|S(u,v)|^2-1$.  At this stage coarea is used
only to compute that already established geometric coefficient, not to derive
a separate asymptotic formula.  Since $q_1$ only rotates the shape vector,
polar task coordinates and the coarea formula give the following identity,
where $\dd\ell_{uv}$ denotes Euclidean arclength in the $(u,v)$-plane:
\begin{equation}\label{eq:robot-flat-regular-reduction}
 \int_\Gamma\frac{\dd\ell}{J_X}
 =\int_{\{\Delta=0\}\cap\mathrm{regular\ shape\ arc}}
 \frac{2}{|\nabla\Delta|}\dd\ell_{uv}
 =2\int_{4\pi/3}^{3\pi/2}
 \frac{\dd u}{|\partial_v\Delta(u,v_\Gamma(u))|}.
\end{equation}
Set $x=-\tan(u/2)$.  Along the arc,
$-\tan(v_\Gamma(u)/2)=3/x$, and direct differentiation of
\eqref{eq:robot-shape-factorization} gives
\[
 |\partial_v\Delta(u,v_\Gamma(u))|=\frac{8x}{1+x^2},
 \qquad
 \dd u=-\frac{2\dd x}{1+x^2}.
\]
As $u$ increases from $4\pi/3$ to $3\pi/2$, $x$ decreases from $\sqrt3$ to
$1$.  Hence
\[
 \int_\Gamma\frac{\dd\ell}{J_X}
 =\frac14\log3.
\]
Together with~\eqref{eq:robot-fold-constant}, this proves
\eqref{eq:robot-flat-prior-coefficients}.
\end{proof}

\begin{remark}[Why the isolated configuration matters]
A method that represents the inverse-kinematics solutions only by the smooth
curve $\Gamma$ misses the isolated configuration $\bm q_0$.  Equivalently, a
generative method trained to sample only from the regular solution curve
would omit a genuine small-noise mode.  For a flat prior that omission loses
exactly one third of the limiting posterior mass.  Writing the task-noise
standard deviation as
$\sigma_z:=z^{-1/2}$, the fold neighborhood has the anisotropic widths
\[
 s=O(\sigma_z),
 \qquad
 u,v=O(\sqrt{\sigma_z}).
\]
Thus the isolated mode is not a conventional Gaussian component.  The two
$u,v$ directions lie in the kernel of the first derivative of the task map,
meaning that infinitesimal motion in those directions does not move the tip to
first order.  Their neighborhoods are therefore much wider than the ordinary
$s$ direction, in which the tip error changes linearly.
The additive comparison in Theorem~\ref{thm:global-leading} places this
isolated anisotropic contribution on the same scale as the smooth solution
curve, whose localization is Gaussian only in the two transverse
directions.  In Proposition~\ref{prop:robot-posterior-splitting}, the regular curve is
treated by Theorem~\ref{thm:integrated-pointwise-cells}; the endpoint caps and
the fold are handled by Corollary~\ref{cor:coercive-leading-piece}, with
quadratic and anisotropic weights, respectively.  Theorem~\ref{thm:global-leading} then sums the pieces and gives the
limiting mixture in~\eqref{eq:robot-weak-limit}.
\end{remark}

\section{Weibullian chaos: tails of homogeneous functionals}
\label{sec:weibull-frontier}

This section studies the rare-event probability
\[
 \mathbb P\{H(\bm\eta)>x\},\qquad x\to\infty,
\]
for a homogeneous function $H$ of a random vector $\bm\eta$.  The word
\emph{chaos} is used here in its probabilistic sense: it refers to the random
variable obtained by applying a homogeneous function to a random input, and
has no connection with dynamical chaos.  The familiar special case in which
$\bm\eta$ is Gaussian and $H$ is a homogeneous polynomial is called a
homogeneous Gaussian chaos.  We allow a more general, direction-dependent
Weibull-type radial decay.

Let $D\ge2$, let $\sigma$ denote surface measure on $\mathbb S^{D-1}$,
and let $\bm\eta$ be an $\R^D$-valued random vector with density
\begin{equation}\label{eq:weibull-density}
 p_{\bm\eta}(v)=A(v)e^{-F(v)},
\end{equation}
where the right-hand side is assumed to be normalized.  We regard $F$ as the
radial cost of observing a large vector and $A$ as its homogeneous prefactor.
Assume that $A\ge0$ and $F\ge0$ are homogeneous of orders $\beta_A$ and
$\beta_F>0$, respectively, and let $H:\R^D\to\R$ be homogeneous of order
$\alpha_H>0$.  Thus, for $r>0$ and $\theta\in\mathbb S^{D-1}$,
\[
 A(r\theta)=r^{\beta_A}A(\theta),\qquad
 F(r\theta)=r^{\beta_F} F(\theta),\qquad
 H(r\theta)=r^{\alpha_H}H(\theta).
\]
Their angular restrictions are assumed measurable; the additional regularity
needed in the two local regimes is stated below.  For a centered isotropic
Gaussian vector one has $\beta_F=2$ and $F(\theta)=1/2$ on the sphere.  The
present model permits the radial cost $F(\theta)$ to vary with direction and,
on sets of spherical measure zero, even to vanish.

The homogeneity makes the tail event transparent on each ray
$v=r\theta$.  A direction contributes only when $H(\theta)>0$, and in that
direction the inequality $H(r\theta)>x$ is equivalent to
\[
 r>\left(\frac{x}{H(\theta)}\right)^{1/\alpha_H}.
\]
Consequently,
\[
 \mathbb P\{H(\bm\eta)>x\}
 =\int_{\{H>0\}} A(\theta)
   \int_{(x/H(\theta))^{1/\alpha_H}}^\infty
   r^{D-1+\beta_A}e^{-F(\theta)r^{\beta_F}}\dd r\dd\sigma(\theta).
\]
The competition is therefore between two angular quantities: $H(\theta)$,
which measures how efficiently a large radius produces a large value of the
functional, and $F(\theta)$, which measures how costly that radius is under
the density.

Set
\begin{equation}\label{eq:weibull-parameters}
 \lambda:=\frac{D+\beta_A}{\beta_F}>0,
 \qquad
 \varrho:=\frac{\beta_F}{\alpha_H},
 \qquad
 z:=x^\varrho.
\end{equation}
Here $\lambda$ is the shape exponent produced by the polar Jacobian and the
homogeneous amplitude, while $z$ is the large parameter associated with the
threshold $x$.  Write
\begin{equation}\label{eq:weibull-angular-data}
 \mathcal G:=\{\theta\in\mathbb S^{D-1}:H(\theta)>0\},
 \qquad
 \mathcal Z_F:=\{\theta\in\mathcal G:F(\theta)=0\},
 \qquad
 \Phi(\theta):=\frac{F(\theta)}{H(\theta)^\varrho}.
\end{equation}
The measurable set $\mathcal G$ is the set of directions in which a positive
exceedance is possible.  The quotient $\Phi$ is the \emph{effective angular
cost}: at the smallest radius that achieves $H(r\theta)>x$, the exponential
cost equals
\[
 F(\theta)\left(\frac{x}{H(\theta)}\right)^{\beta_F/\alpha_H}
 =z\Phi(\theta).
\]
Subject to the angular weight $W$ introduced below, the tail is governed by
directions where $\Phi$ is smallest.

On $\mathcal G\setminus\mathcal Z_F$, set
\begin{equation}\label{eq:weibull-angular-weight}
 W(\theta):=\frac{A(\theta)}{F(\theta)^\lambda}.
\end{equation}
We assume that $\mathcal Z_F$ has spherical measure zero, that $F>0$ on
$\mathcal G\setminus\mathcal Z_F$, and that $W$ admits an $L^1(\mathcal G,\sigma)$
representative.  Values of that representative on $\mathcal Z_F$ do not affect
any integral.  Radial integration itself produces exactly this angular weight,
so no auxiliary choice of amplitude is involved.  In particular, when $F$
approaches zero, the condition controls the compensating vanishing of $A$
needed for the density to remain integrable.

Throughout this section, $\Gamma(\lambda,y)$ denotes the upper incomplete
gamma function defined in~\eqref{eq:upper-incomplete-gamma}.  The substitution
$s=F(\theta)r^{\beta_F}$ in the preceding ray integral gives the
exact identity
\begin{equation}\label{eq:weibull-exact-angular-reduction}
 \mathbb P\{H(\bm\eta)>x\}
 =\frac{1}{\beta_F}\int_{\mathcal G}
 W(\theta)\Gamma\bigl(\lambda,z\Phi(\theta)\bigr)\dd\sigma(\theta).
\end{equation}
Since $\mathcal Z_F$ is null, the integrand there may be assigned any finite
representative value.  Formula~\eqref{eq:weibull-exact-angular-reduction}
reduces the original $D$-dimensional tail problem to an angular localization
problem on the sphere.  It resembles a Laplace integral, but the profile is the
incomplete gamma function rather than a pure exponential.

Three behaviors should be distinguished.  If the effective cost $\Phi$ stays bounded away from zero and has a
nondegenerate minimum, it produces an ordinary exponentially small Laplace
tail.  If $\Phi$ vanishes on a submanifold inside $\mathcal G$, a
shrinking angular neighborhood of that zero set produces a polynomial tail.
Finally, $\Phi$ may approach its minimum at the boundary $H=0$ of
$\mathcal G$ because $F$ vanishes there as well; the relative orders of these
two vanishings determine whether the boundary law is polynomial or
exponential.  In the two zero-cost regimes the argument $z\Phi$ remains of
order one on the leading scale, so the exact gamma profile cannot be replaced
by its large-argument expansion.

We shall repeatedly use the bound
\begin{equation}\label{eq:incomplete-gamma-bound}
 0\le\Gamma(\lambda,y)
 \le C_\lambda\bigl(1+y^{(\lambda-1)_+}\bigr)e^{-y/2},
 \qquad y\ge0.
\end{equation}
For comparisons between two strictly positive exponential rates we also use
the sharper large-argument estimate
\begin{equation}\label{eq:incomplete-gamma-sharp-bound}
 \Gamma(\lambda,y)
 \le C_\lambda\bigl(1+y^{(\lambda-1)_+}\bigr)e^{-y},
 \qquad y\ge1.
\end{equation}
Indeed,
\[
 \Gamma(\lambda,y)
 =e^{-y}\int_0^\infty (y+t)^{\lambda-1}e^{-t}\dd t.
\]
If $0<\lambda\le1$, the integral is at most $y^{\lambda-1}\le1$ for
$y\ge1$.  If $\lambda>1$, the inequality
$(y+t)^{\lambda-1}\le C_\lambda(y^{\lambda-1}+t^{\lambda-1})$ proves
\eqref{eq:incomplete-gamma-sharp-bound}.  On $0\le y\le1$, boundedness by
$\Gamma(\lambda)$, followed by a harmless weakening of the exponential,
gives~\eqref{eq:incomplete-gamma-bound}.

All asymptotic statements below are obtained by verifying the hypotheses of
Theorem~\ref{thm:profile-leading} for the exact angular identity
\eqref{eq:weibull-exact-angular-reduction}; the section does not introduce a
separate rescaling theorem for the gamma kernel.

\subsection{An interior zero set of the radial phase}
\label{subsec:weibull-interior-zero}

An interior zero is an angular direction $p\in\mathcal G$ for which
$F(p)=0$ while $H(p)>0$.  Along such directions the radial exponential factor no longer decays.  The
set of these directions has zero spherical measure, so the probability is
instead determined by how quickly the cost returns as the direction moves
away from the zero set.  This is the source of the resulting polynomial
tail: the dominant angular neighborhood shrinks as $x$ grows, but there
is no positive exponential rate on that neighborhood.

Because $H>0$ on $\mathcal G$, the zero sets of $F$ and $\Phi$ agree there.
The assumption that $W=A/F^\lambda$ extends boundedly near the zero set is the
corresponding integrability condition: the amplitude $A$ compensates for the
vanishing radial cost.  The anisotropic weights $\beta_j$ describe how fast
$\Phi$ rises in the different normal directions.  A direction with weight
$\beta_j$ has width $z^{-\beta_j}$, so the total angular volume of the
leading tube is $z^{-Q}$ with $Q=\sum_j d_j\beta_j$.

Write $N_{\mathbb S}\mathcal Z$ for the normal bundle of $\mathcal Z$ in
$\mathbb S^{D-1}$, equipped with the metric induced by the sphere, and let
\[
 \operatorname{Exp}_{\mathcal Z}(p,\xi):=\exp_p^{\mathbb S}(\xi)
\]
be its normal exponential map on a sufficiently small disk bundle.

\begin{theorem}[Interior zero set of the radial phase]
\label{thm:weibull-interior-zero}
Suppose that $\mathcal Z_F=\mathcal Z$ is a compact $C^2$ submanifold of
codimension $d\ge1$ in the sphere and is compactly contained in
$\operatorname{int}\mathcal G$, written
\[
 \mathcal Z\Subset\operatorname{int}\mathcal G.
\]
Assume that $W$ extends continuously and boundedly to a neighborhood of
$\mathcal Z$.
Suppose that the normal bundle has a $C^1$ orthogonal decomposition
\[
 N_{\mathbb S}\mathcal Z=E_1\oplus\cdots\oplus E_\ell,
 \qquad \rank E_j=d_j,
\]
with weights $\beta_j>0$.  Define the fiber dilations and their homogeneous
dimension by
\begin{equation}\label{eq:weibull-interior-dilation}
 D_q^{\bm\beta}(\xi_1+\cdots+\xi_\ell)
 :=q^{\beta_1}\xi_1+\cdots+q^{\beta_\ell}\xi_\ell,
 \qquad
 Q:=\sum_{j=1}^\ell d_j\beta_j.
\end{equation}
For $\xi\ne0$, let $\rho_{\bm\beta}(\xi)$ be the unique $r>0$ such
that $|D_{1/r}^{\bm\beta}\xi|=1$, and set
$\rho_{\bm\beta}(0)=0$.

Assume that there is a continuous function
$G_{\mathcal Z}:N_{\mathbb S}\mathcal Z\to[0,\infty)$ such that, uniformly in
$p\in\mathcal Z$,
\begin{align}
 G_{\mathcal Z}\bigl(p,D_q^{\bm\beta}\xi\bigr)
 &=qG_{\mathcal Z}(p,\xi),
 \label{eq:weibull-interior-model-homogeneity}\\
 0<g_-\le G_{\mathcal Z}(p,\xi)&\le g_+<\infty
 \quad\text{when }\rho_{\bm\beta}(\xi)=1,
 \label{eq:weibull-interior-model-coercivity}\\
 q^{-1}\Phi\!\left(\operatorname{Exp}_{\mathcal Z}\bigl(p,D_q^{\bm\beta}\xi\bigr)\right)
 &\longrightarrow G_{\mathcal Z}(p,\xi)
 \quad\text{locally uniformly in }\xi
 \quad(q\downarrow0).
 \label{eq:weibull-interior-model-convergence}
\end{align}
Assume in addition that, for some $c,\eps,\delta>0$,
\begin{align}
 \Phi(\operatorname{Exp}_{\mathcal Z}(p,\xi))&\ge c\rho_{\bm\beta}(\xi),
 &&p\in\mathcal Z,
 \quad |\xi|<\eps,
 \label{eq:weibull-interior-lower-bound}\\
 \Phi(\theta)&\ge\delta,
 &&\theta\in\mathcal G\setminus
 \operatorname{Exp}_{\mathcal Z}\bigl(\{(p,\xi):|\xi|<\eps\}\bigr).
 \label{eq:weibull-interior-gap}
\end{align}
Then, as $x\to\infty$ and $z=x^{\beta_F/\alpha_H}$,
\begin{equation}\label{eq:weibull-interior-asymptotic}
 \mathbb P\{H(\bm\eta)>x\}
 =z^{-Q}\bigl(C_{\mathcal Z}+o(1)\bigr),
\end{equation}
where
\begin{equation}\label{eq:weibull-interior-constant}
 C_{\mathcal Z}
 :=\frac{1}{\beta_F}\int_{\mathcal Z}W(p)
 \left[
   \int_{N_{\mathbb S,p}\mathcal Z}
   \Gamma\bigl(\lambda,G_{\mathcal Z}(p,\xi)\bigr)\dd\xi
 \right]\dvol_{\mathcal Z}(p).
\end{equation}
For every $p\in\mathcal Z$, the fiber integral has the equivalent forms
\begin{align}
 \int_{N_{\mathbb S,p}\mathcal Z}
 \Gamma\bigl(\lambda,G_{\mathcal Z}(p,\xi)\bigr)\dd\xi
 &=\Gamma(\lambda+Q)
   \vol\bigl\{\xi\in N_{\mathbb S,p}\mathcal Z:
   G_{\mathcal Z}(p,\xi)\le1\bigr\}
 \label{eq:weibull-interior-sublevel-form}\\
 &=\frac{\Gamma(\lambda+Q)}{\Gamma(Q+1)}
   \int_{N_{\mathbb S,p}\mathcal Z}e^{-G_{\mathcal Z}(p,\xi)}\dd\xi.
 \label{eq:weibull-interior-exponential-form}
\end{align}
In particular, if $W|_{\mathcal Z}$ is not identically zero, then
$C_{\mathcal Z}>0$ and
\begin{equation}\label{eq:weibull-interior-x-form}
 \mathbb P\{H(\bm\eta)>x\}
 \sim C_{\mathcal Z}x^{-\beta_F Q/\alpha_H}.
\end{equation}
\end{theorem}

\begin{proof}
Let $U_\eps$ be the tubular neighborhood appearing in
\eqref{eq:weibull-interior-gap}.  The contribution of
$\mathcal G\setminus U_\eps$ to
\eqref{eq:weibull-exact-angular-reduction} is exponentially small by
\eqref{eq:incomplete-gamma-bound}, the gap
\eqref{eq:weibull-interior-gap}, and $W\in L^1(\mathcal G,\sigma)$.

Choose a finite orthonormal trivializing cover of
$N_{\mathbb S}\mathcal Z$ and a subordinate Borel partition of the compact
base.  On each resulting normal-cell piece, write
\[
 \dd\sigma(\operatorname{Exp}_{\mathcal Z}(p,\xi))
 =J(p,\xi)\dd\xi\dvol_{\mathcal Z}(p),
 \qquad J(p,0)=1.
\]
Let $\Afib(p)$ be the normal disk in this trivialization and let
$\Afib^\infty(p)=N_{\mathbb S,p}\mathcal Z$.  Apply
Theorem~\ref{thm:profile-leading} with
\[
 \mu=0,
 \quad \mathcal K(v)=\Gamma(\lambda,v),
 \quad \mathcal N_{\mathcal K}(z)=1,
 \quad \mathcal K_\infty(v)=\Gamma(\lambda,v),
 \quad c\equiv1,
 \quad G=G_{\mathcal Z},
\]
with coordinate amplitude
$b(p,\xi)=W(\operatorname{Exp}_{\mathcal Z}(p,\xi))J(p,\xi)$ and limit amplitude $b_0(p)=W(p)$.
The local convergence~\eqref{eq:profile-local-convergence} is exactly
\eqref{eq:weibull-interior-model-convergence} together with continuity of
$W$ and $J$.  The lower bound
\eqref{eq:profile-scaled-lower-bound} follows from
\eqref{eq:weibull-interior-lower-bound} and the upper bound for $G_{\mathcal Z}$ on the weighted unit sphere in
\eqref{eq:weibull-interior-model-coercivity}.  The rescaled normal disks converge
to the full normal spaces, and coercivity of $G_{\mathcal Z}$ together with
Lemma~\ref{lem:moments} verifies
\eqref{eq:profile-cell-convergence}--\eqref{eq:profile-cell-tightness}.
Finally,~\eqref{eq:incomplete-gamma-bound} verifies
the profile envelope~\eqref{eq:profile-envelope}.  Summing the finitely many
pieces and restoring the factor $1/\beta_F$ from
\eqref{eq:weibull-exact-angular-reduction} gives
\eqref{eq:weibull-interior-asymptotic} and
\eqref{eq:weibull-interior-constant}.

For each $p$, Proposition~\ref{prop:profile-sublevel-formula}, with
$C=N_{\mathbb S,p}\mathcal Z$, $B_0\equiv1$, $G=G_{\mathcal Z}(p,\cdot)$, and
$h=Q$, gives
\eqref{eq:weibull-interior-sublevel-form} and
\eqref{eq:weibull-interior-exponential-form}.  Positivity follows from
$W\ge0$ and strict positivity of the profile integral.
\end{proof}

The most familiar local geometry is a quadratic, nondegenerate return of the
effective cost in directions normal to $\mathcal Z$; this is usually called
the Morse--Bott situation.  Every one of the $d$ normal directions then has
width $z^{-1/2}$, so the polynomial power is $z^{-d/2}$; the normal Hessian
determines only the coefficient.

\begin{corollary}[Quadratic (Morse--Bott) interior zero set]
\label{cor:weibull-interior-morse-bott}
Under the hypotheses of Theorem~\ref{thm:weibull-interior-zero}, suppose more
specifically that, uniformly for $p\in\mathcal Z$,
\begin{equation}\label{eq:weibull-interior-morse-bott-model}
 \Phi(\operatorname{Exp}_{\mathcal Z}(p,\xi))
 =\frac12\inner{L_p\xi}{\xi}+o(|\xi|^2),
\end{equation}
where $L_p:N_{\mathbb S,p}\mathcal Z\to N_{\mathbb S,p}\mathcal Z$ is a
continuous family of positive-definite self-adjoint maps.  Then
\begin{equation}\label{eq:weibull-interior-morse-bott-asymptotic}
 \mathbb P\{H(\bm\eta)>x\}
 \sim C_{\mathrm{MB}}x^{-\beta_F d/(2\alpha_H)}
\end{equation}
whenever $W|_{\mathcal Z}$ is not identically zero, where
\begin{equation}\label{eq:weibull-interior-morse-bott-constant}
 C_{\mathrm{MB}}
 :=\frac{(2\pi)^{d/2}\Gamma(\lambda+d/2)}
 {\beta_F\,\Gamma(d/2+1)}
 \int_{\mathcal Z}
 \frac{W(p)}{\sqrt{\det L_p}}\dvol_{\mathcal Z}(p).
\end{equation}
For an isolated interior zero, $d=D-1$.
\end{corollary}

\begin{proof}
The uniform expansion and compactness of $\mathcal Z$ verify
Theorem~\ref{thm:weibull-interior-zero} with the single normal weight $1/2$,
so $Q=d/2$ and
$G_{\mathcal Z}(p,\xi)=\inner{L_p\xi}{\xi}/2$.  The unit sublevel set is an
ellipsoid of volume
\[
 \vol\{\xi:G_{\mathcal Z}(p,\xi)\le1\}
 =\frac{(2\pi)^{d/2}}
 {\Gamma(d/2+1)\sqrt{\det L_p}}.
\]
The conclusion follows from
\eqref{eq:weibull-interior-sublevel-form} and
\eqref{eq:weibull-interior-x-form}.
\end{proof}

\subsection{A boundary frontier with simultaneous vanishing}
\label{subsec:weibull-boundary-frontier}

The second mechanism occurs at the boundary of
$\mathcal G=\{H>0\}$.  A boundary direction itself has $H=0$ and therefore
cannot produce a positive exceedance at any finite radius.  For interior directions close to the
boundary, however, $H$ is small, so achieving $H(r\theta)>x$ requires an
especially large radius.  If $F$ remains positive, this creates a positive
exponential cost.  A boundary frontier arises when $F$ also tends to zero and
may compensate for the required radius.

For the remainder of the section, assume $\mathcal Z_F=\varnothing$, so the
radial cost is positive at every direction with $H>0$; it may nevertheless
vanish as such directions approach $\partial\mathcal G$.  Let $\mathcal B$ be
a compact smooth component of $\partial\mathcal G$.  Assume that the
restriction of $H$ to the sphere is $C^2$ near $\mathcal B$ and that its
spherical differential, equivalently its gradient tangent to the sphere, does
not vanish there.  Thus $H$ itself can be used as a distance-like inward
coordinate.  Shrinking the neighborhood if necessary, set $u=H$ and choose a
$C^1$ collar, meaning a product parametrization by a boundary point $t$ and
the inward coordinate $u$,
\begin{equation}\label{eq:weibull-collar}
 \varphi:\mathcal B\times[0,\eps)\longrightarrow
 \overline{\mathcal G},
 \qquad
 \varphi(t,0)=t,
 \qquad
 H(\varphi(t,u))=u.
\end{equation}
Let $\sigma_{\mathcal B}$ be the Riemannian measure induced on
$\mathcal B$.  Write the spherical measure in this collar as
\begin{equation}\label{eq:weibull-collar-jacobian}
 \dd\sigma(\varphi(t,u))
 =J(t,u)\dd\sigma_{\mathcal B}(t)\dd u
\end{equation}
and define the effective angular density
\begin{equation}\label{eq:weibull-effective-density}
 \omega(t,u):=W(\varphi(t,u))J(t,u).
\end{equation}
Since $H=u$ in the collar, the effective cost is
$\Phi=F/u^\varrho$.  This gives a direct radial interpretation of the two
regimes.  If $\varrho=m+\zeta$ with $0<\zeta<1$ and
$F\asymp u^{m+1}$, then
$\Phi\asymp u^{1-\zeta}\downarrow0$.  A shrinking boundary layer has no
positive exponential rate and produces a polynomial tail.  If instead
$\varrho=m$ and $F\asymp u^m$, then $\Phi$ approaches a positive function on the
frontier.  The tail remains exponentially small, and a tangential Gaussian
scale combines with a one-sided linear boundary scale.

\begin{proposition}[Fractional and integer frontier laws]
\label{prop:weibull-frontier}
Assume that $\omega$ extends continuously and boundedly to
$\mathcal B\times[0,\eps]$.

\begin{enumerate}[(i)]
\item Suppose
\begin{equation}\label{eq:weibull-fractional-ratio}
 \varrho=m+\zeta,
 \qquad m\in\mathbb N_0,
 \qquad 0<\zeta<1,
\end{equation}
and, uniformly on the collar,
\begin{equation}\label{eq:weibull-fractional-phase}
 F(\varphi(t,u))
 =u^{m+1}\bigl(c(t)+u\mathcal R_{\mathrm{frac}}(t,u)\bigr),
 \qquad
 \inf_{t\in\mathcal B}c(t)>0,
\end{equation}
where $c$ is continuous and $\mathcal R_{\mathrm{frac}}$ is bounded.  Assume also that
$\Phi\ge\delta>0$ on the complement of the collar in $\mathcal G$.  Then
\begin{equation}\label{eq:weibull-fractional-asymptotic}
 \mathbb P\{H(\bm\eta)>x\}
 =C_{\mathrm{frac}}z^{-1/(1-\zeta)}+o\bigl(z^{-1/(1-\zeta)}\bigr),
\end{equation}
where
\begin{equation}\label{eq:weibull-fractional-constant}
 C_{\mathrm{frac}}
 :=\frac{\Gamma(\lambda+1/(1-\zeta))}{\beta_F}
 \int_{\mathcal B}
 \omega(t,0)c(t)^{-1/(1-\zeta)}\dd\sigma_{\mathcal B}(t).
\end{equation}
In particular, if the boundary trace of $\omega$ is not identically zero,
then $C_{\mathrm{frac}}>0$ and
\begin{equation}\label{eq:weibull-fractional-x-form}
 \mathbb P\{H(\bm\eta)>x\}
 \sim C_{\mathrm{frac}}
 x^{-\beta_F/[\alpha_H(1-\zeta)]}.
\end{equation}

\item Suppose $\varrho=m\in\mathbb N$ and, uniformly on the collar,
\begin{equation}\label{eq:weibull-integer-phase}
 F(\varphi(t,u))
 =u^m\bigl(c_0(t)+u c_1(t)+O(u^2)\bigr).
\end{equation}
Here $c_0\in C^2(\mathcal B)$, $c_1$ is continuous, and the remainder is
uniform on the collar.  Assume that $c_0$ has a unique minimum at
$t_0\in\mathcal B$,
\begin{equation}\label{eq:weibull-integer-nondegeneracy}
 \mu:=c_0(t_0)>0,
 \qquad
 \nabla_{\mathcal B}c_0(t_0)=0,
 \qquad
 \nabla_{\mathcal B}^2c_0(t_0)>0,
 \qquad
 c_1(t_0)>0,
\end{equation}
and that, for some $\delta>0$, one has $\Phi\ge\mu+\delta$ outside a
sufficiently small collar neighborhood of $(t_0,0)$.  If $\omega(t_0,0)>0$, then
\begin{equation}\label{eq:weibull-integer-asymptotic}
 \mathbb P\{H(\bm\eta)>x\}
 \sim C_{\mathrm{integer}}
 z^{\lambda-1-D/2}e^{-\mu z},
\end{equation}
where
\begin{equation}\label{eq:weibull-integer-constant}
 C_{\mathrm{integer}}
 :=\frac{(2\pi)^{(D-2)/2}}{\beta_F}
 \frac{
   \omega(t_0,0)\mu^{\lambda-1}
 }{
   c_1(t_0)
   \sqrt{\det\bigl(\nabla_{\mathcal B}^2c_0(t_0)\bigr)}
 }.
\end{equation}
For $D=2$, the determinant of the zero-dimensional Hessian and the factor
$(2\pi)^0$ are understood to be one.  Since $z=x^m$, the power in
\eqref{eq:weibull-integer-asymptotic} may equivalently be written as
$x^{m(\lambda-1-D/2)}$.
\end{enumerate}
\end{proposition}

\begin{proof}
For~(i), put $a:=1-\zeta\in(0,1)$.  Equation
\eqref{eq:weibull-fractional-phase} gives
\[
 \Phi(\varphi(t,u))
 =c(t)u^a\left(1+\frac{u\mathcal R_{\mathrm{frac}}(t,u)}{c(t)}\right).
\]
Let $c_-:=\min_{t\in\mathcal B}c(t)>0$.  Since
$\mathcal R_{\mathrm{frac}}$ is bounded, after reducing $\eps$ one has,
uniformly on the collar,
\begin{equation}\label{eq:weibull-fractional-two-sided-bound}
 \frac12c(t)u^a
 \le \Phi(\varphi(t,u))
 \le \frac32c(t)u^a.
\end{equation}
The complement of the collar contributes
$O((1+z^{(\lambda-1)_+})e^{-\delta z/2})$ by
\eqref{eq:incomplete-gamma-bound} and $W\in L^1(\mathcal G,\sigma)$, and is
therefore $o(z^{-1/a})$.  Use the ambient homogeneous extension
$G_{\mathrm{frac}}(s):=|s|^a$, which agrees with the phase model on the
relative cell $\R_+$.  Apply Theorem~\ref{thm:profile-leading} to the collar
with
\[
 P=\mathcal B,
 \quad \Afib(t)=(0,\eps),
 \quad \Afib^\infty(t)=\R_+,
 \quad \beta_u=\frac{1}{a},
 \quad Q=\frac{1}{a},
 \quad G=G_{\mathrm{frac}},
 \quad c=c(t),
\]
\[
 \mu=0,
 \quad \mathcal K(v)=\Gamma(\lambda,v),
 \quad \mathcal N_{\mathcal K}(z)=1,
 \quad \mathcal K_\infty(v)=\Gamma(\lambda,v),
 \quad b(t,u)=\omega(t,u).
\]
The local phase convergence follows from
\eqref{eq:weibull-fractional-phase}, while the lower estimate in
\eqref{eq:weibull-fractional-two-sided-bound} gives
\eqref{eq:profile-scaled-lower-bound}.  The exact scaled intervals are
\[
 (0,\eps(zc(t))^{1/a}),
\]
which exhaust $\R_+$ locally and uniformly in $t$ because $c$ is bounded
above and below on the compact base.  This local exhaustion, together with
coercivity of $G_{\mathrm{frac}}$ and Lemma~\ref{lem:moments}, verifies
\eqref{eq:profile-cell-convergence}--\eqref{eq:profile-cell-tightness}.
Finally,~\eqref{eq:incomplete-gamma-bound} verifies the profile envelope.
Formula~\eqref{eq:integrated-profile-leading} therefore yields
\begin{equation}\label{eq:weibull-fractional-profile-application}
 \mathbb P\{H(\bm\eta)>x\}
 =\frac{z^{-1/(1-\zeta)}}{\beta_F}\left[
 \int_{\mathcal B}\omega(t,0)c(t)^{-1/(1-\zeta)}\dd\sigma_{\mathcal B}(t)
 \int_0^\infty\Gamma(\lambda,s^{1-\zeta})\dd s+o(1)\right].
\end{equation}
Proposition~\ref{prop:profile-sublevel-formula}, applied on $\R_+$ with
$G(s)=|s|^{1-\zeta}$ and homogeneous dimension $1/(1-\zeta)$, gives
\[
 \int_0^\infty\Gamma(\lambda,s^{1-\zeta})\dd s
 =\Gamma\left(\lambda+\frac{1}{1-\zeta}\right).
\]
This proves~\eqref{eq:weibull-fractional-asymptotic}--
\eqref{eq:weibull-fractional-x-form}.

For~(ii), choose geodesic normal coordinates
$y\in\R^{D-2}$ on $\mathcal B$ centered at $t_0$, and let
$\mathsf H_0:=\nabla_{\mathcal B}^2c_0(t_0)$.  Writing $j_{\mathcal B}(y)$ for the
volume density, one has
\[
 c_0(t(y))=\mu+\frac12\inner{\mathsf H_0y}{y}+o(|y|^2),
 \qquad
 c_1(t(y))=c_1(t_0)+o(1).
\]
Dividing~\eqref{eq:weibull-integer-phase} by $u^m$ therefore gives
\begin{equation}\label{eq:weibull-integer-local-model}
 \Phi(t(y),u)
 =\mu+\frac12\inner{\mathsf H_0y}{y}+c_1(t_0)u
 +o(|y|^2+u).
\end{equation}
Positive definiteness of $\mathsf H_0$ and $c_1(t_0)>0$ imply, after shrinking
the coordinate neighborhood,
\begin{equation}\label{eq:weibull-integer-local-lower-bound}
 \Phi(t(y),u)-\mu\ge c(|y|^2+u),
 \qquad u\ge0.
\end{equation}
By the assumed gap and the sharp bound
\eqref{eq:incomplete-gamma-sharp-bound}, the complement is exponentially
small relative to $\Gamma(\lambda,z\mu)z^{-D/2}$.  Use the
ambient homogeneous extension
\[
 G_{\mathrm{integer}}(v,w)
 :=\frac12\inner{\mathsf H_0v}{v}+c_1(t_0)|w|,
\]
which agrees with the one-sided local model on $w\ge0$.  Apply
Theorem~\ref{thm:profile-leading} with $P=\{0\}$, $c\equiv1$, local
coordinates $s=(y,u)$, $\Afib$ equal to the chosen bounded coordinate
neighborhood with $u\ge0$, limiting cell
$\Afib^\infty=\R^{D-2}\times\R_+$, weights $(1/2,\ldots,1/2,1)$, homogeneous dimension
$D/2$, $G=G_{\mathrm{integer}}$, and
\[
 b(y,u)=\omega(t(y),u)j_{\mathcal B}(y),
\]
using
\[
 \mathcal K(v)=\Gamma(\lambda,v),
 \qquad
 \mathcal N_{\mathcal K}(z)=\Gamma(\lambda,z\mu),
 \qquad
 \mathcal K_\infty(v)=e^{-v}.
\]
Indeed, after the substitution $\xi=r+v$ in the numerator,
\[
 \frac{\Gamma(\lambda,z\mu+v)}{\Gamma(\lambda,z\mu)}
 =e^{-v}
 \frac{\int_{z\mu}^\infty(r+v)^{\lambda-1}e^{-r}\dd r}
      {\int_{z\mu}^\infty r^{\lambda-1}e^{-r}\dd r}.
\]
For $v$ in a fixed bounded interval,
$((r+v)/r)^{\lambda-1}\to1$ uniformly for $r\ge z\mu$.  Hence, locally
uniformly for $v\ge0$,
\begin{equation}\label{eq:weibull-shifted-profile-limit}
 \frac{\Gamma(\lambda,z\mu+v)}{\Gamma(\lambda,z\mu)}
 \longrightarrow e^{-v}.
\end{equation}
For the global envelope, if $0<\lambda\le1$, then
$(r+v)^{\lambda-1}\le r^{\lambda-1}$ and the quotient in the preceding
display is at most $e^{-v}$.  If $\lambda>1$, then for $z$ large enough that
$z\mu\ge1$,
\[
 \left(1+\frac vr\right)^{\lambda-1}
 \le (1+v)^{\lambda-1},\qquad r\ge z\mu.
\]
Thus, in both cases,
\begin{equation}\label{eq:weibull-shifted-profile-envelope}
 \frac{\Gamma(\lambda,z\mu+v)}{\Gamma(\lambda,z\mu)}
 \le C(1+v^{(\lambda-1)_+})e^{-v},
 \qquad v\ge0,
\end{equation}
for all sufficiently large $z$, so the profile envelope holds.  Equations
\eqref{eq:weibull-integer-local-model} and
\eqref{eq:weibull-integer-local-lower-bound} verify the rescaled phase
hypotheses.  The bounded coordinate half-neighborhood exhausts
$\R^{D-2}\times\R_+$ locally under the dilation with weights
$(1/2,\ldots,1/2,1)$.  This local exhaustion, together with coercivity of
$G_{\mathrm{integer}}$ and Lemma~\ref{lem:moments}, verifies the cell
convergence and tightness conditions.  Theorem~\ref{thm:profile-leading}
therefore gives
\begin{equation}\label{eq:weibull-integer-profile-application}
 \begin{aligned}
 \mathbb P\{H(\bm\eta)>x\}
 &=\frac{\Gamma(\lambda,z\mu)}{\beta_F} z^{-D/2}\Bigg[
 \omega(t_0,0)\int_{\R^{D-2}\times\R_+}e^{-G_{\mathrm{integer}}(v,w)}\dd v\dd w\\
 &\hspace{16em}{}+o(1)\Bigg].
 \end{aligned}
\end{equation}
The model integral equals
\[
 \frac{(2\pi)^{(D-2)/2}}
 {c_1(t_0)\sqrt{\det\mathsf H_0}},
\]
and
$\Gamma(\lambda,z\mu)\sim(z\mu)^{\lambda-1}e^{-z\mu}$.
Substitution proves~\eqref{eq:weibull-integer-asymptotic} and
\eqref{eq:weibull-integer-constant}.
\end{proof}

\begin{example}[Matched angular cost on the contributing hemisphere]
\label{ex:weibull-matched-angular-cost}
The boundary mechanism is sharpest when the radial cost matches the functional
on the contributing region.  If $F=H$ on $\mathcal G$, then the effective
angular cost of~\eqref{eq:weibull-angular-data} reduces to
$\Phi=H^{1-\varrho}$, which vanishes along $\partial\mathcal G$ whenever
$\varrho<1$: the cost of the large radius and the yield of the functional
cancel, and the cheapest directions are exactly those approaching the boundary.

Such a matching cannot hold on the whole sphere.  Since $F\ge0$, the identity
$F=H$ everywhere would force $H\ge0$, and a nonnegative $C^1$ function has
vanishing differential at each of its zeros, contradicting the assumption that
the spherical differential of $H$ is nonzero along $\mathcal B$.  We therefore
impose $F=H$ on $\mathcal G$ only and extend it by $F=|H|$ across the boundary,
which keeps the cost nonnegative and continuous.

For $D\ge2$, write $v=(v',v_D)\in\R^{D-1}\times\R$ and set
\begin{equation}\label{eq:weibull-matched-example-data}
 H(v):=|v|v_D,
 \qquad
 F(v):=|v_D|,
 \qquad
 A(v):=C_D\left(\frac{|v_D|}{|v|}\right)^D
 \quad(v\ne0),
\end{equation}
where $A(0):=0$ and
\[
 C_D:=\frac{1}{\Gamma(D)\vol(\mathbb S^{D-1})}.
\]
Then $\alpha_H=2$, $\beta_F=1$, $\beta_A=0$, and the density
$A(v)e^{-F(v)}$ is normalized.  Indeed, ignoring the null equator
$\{\theta_D=0\}$,
\[
 \int_{\R^D}A(v)e^{-F(v)}\dd v
 =C_D\int_{\mathbb S^{D-1}}|\theta_D|^D
 \left(\int_0^\infty r^{D-1}e^{-r|\theta_D|}\dd r\right)\dd\sigma(\theta)
 =C_D\Gamma(D)\vol(\mathbb S^{D-1})=1.
\]
On the sphere,
\[
 H(\theta)=\theta_D,
 \qquad
 F(\theta)=|\theta_D|.
\]
Hence
\[
 \mathcal G=\{\theta_D>0\},
 \qquad
 F=H=\theta_D\quad\hbox{on }\mathcal G,
\]
and
\[
 \lambda=D,
 \qquad
 \varrho=\frac12,
 \qquad
 W(\theta)=C_D\quad(\theta\in\mathcal G).
\]
The boundary component is the equator
$\mathcal B=\{\theta_D=0\}\simeq\mathbb S^{D-2}$.  With the collar
\[
 \varphi(t,u)=\bigl(\sqrt{1-u^2}\,t,u\bigr),
 \qquad t\in\mathbb S^{D-2},
\]
one has
\[
 H(\varphi(t,u))=u,
 \qquad
 F(\varphi(t,u))=u,
 \qquad
 J(t,u)=(1-u^2)^{(D-3)/2}.
\]
Thus Proposition~\ref{prop:weibull-frontier}(i) applies with
$m=0$, $\zeta=1/2$, $c\equiv1$, and
$\omega(t,0)=C_D$.  Equivalently, the exact angular reduction is
\[
 \mathbb P\{H(\bm\eta)>x\}
 =C_D\int_{\{\theta_D>0\}}
 \Gamma\bigl(D,\sqrt{x\theta_D}\bigr)\dd\sigma(\theta).
\]
The resulting boundary law is
\begin{equation}\label{eq:weibull-matched-example-tail}
 \mathbb P\{|\bm\eta|\eta_D>x\}
 \sim
 D(D+1)\frac{\vol(\mathbb S^{D-2})}{\vol(\mathbb S^{D-1})}\,x^{-1}.
\end{equation}
For $D=3$, this becomes
\[
 p_{\bm\eta}(v)
 =\frac{1}{8\pi}\left(\frac{|v_3|}{|v|}\right)^3e^{-|v_3|},
 \qquad
 \mathbb P\{|\bm\eta|\eta_3>x\}\sim\frac6x.
\]
\end{example}

\begin{remark}[Zero-cost and positive-cost profile limits]
\label{rem:weibull-exact-gamma}
When $0<\zeta<1$, the leading collar scale is
$u\asymp z^{-1/(1-\zeta)}$, and therefore
$z\Phi(\varphi(t,u))=O(1)$.  The expansion
\[
 \Gamma(\lambda,y)\sim y^{\lambda-1}e^{-y}
 \left(1+\frac{\lambda-1}{y}+\cdots\right)
\]
is not uniform on that scale and cannot be truncated before angular
integration.  In fact, after the induced powers of $\Phi$ are included, all
formal terms can contribute at the same power $z^{-1/(1-\zeta)}$; the exact
profile sums them into the single factor
$\Gamma(\lambda+1/(1-\zeta))$ in
\eqref{eq:weibull-fractional-constant}.  In the integer case, by contrast,
the positive minimum is removed by the normalization
$\mathcal N_{\mathcal K}(z)=\Gamma(\lambda,z\mu)$; the shifted profile converges to $e^{-v}$ as in
\eqref{eq:weibull-shifted-profile-limit}.  Thus both regimes are
instances of Theorem~\ref{thm:profile-leading}: the zero-cost case retains the
full gamma profile, whereas the positive-cost case has an exponential limiting
profile.
\end{remark}

\end{document}